\documentclass[12pt,reqno,a4paper]{amsart}
\usepackage{mathrsfs, tikz}
\usepackage{amssymb,graphics}
\usepackage{amsmath}
\usepackage{bbold}
\usepackage{a4wide}
\usepackage{mathtools}
\usepackage{bm}
\usepackage{amsthm}
\usepackage[utf8]{inputenc}
\usepackage{graphicx}
\graphicspath{ {images/} }
\usepackage{float}
\usepackage[width=150mm,top=35mm,bottom=25mm,bindingoffset=6mm]{geometry}
\usepackage{fancyhdr}

\usetikzlibrary{arrows}

\DeclarePairedDelimiter\floor{\lfloor}{\rfloor}
\newcommand{\Hquad}{\hspace{0.3em}}

\newtheorem{theorem}{Theorem}[section]
\newtheorem{proposition}[theorem]{Proposition}
\newtheorem{definition}[theorem]{Definition}

\newtheorem{lemma}[theorem]{Lemma}
\newtheorem{sublemma}[theorem]{Sublemma}

\newtheorem{remark}[theorem]{Remark}

\def\A{\mathcal{A}}

\def\C{\mathcal{C}}

\def\Q{\mathcal{Q}}
\def\I{\mathcal{I}}
\def\D{\Delta}
\def\NN{\mathbb{N}}
\def\N{\mathcal{N}}
\def\PP{\mathbb{P}}

\let\eps=\varepsilon

\def\D{\mathcal{D}}

\def\NN{{\mathbb N}}

\newcommand*{\logeq}{\ratio\Leftrightarrow}

\title{Quenched decay of correlations for random contracting Lorenz maps}

\author{Andrew Larkin}
\address{Department of Mathematical Sciences, Loughborough University,
Loughborough, Leicestershire, LE11 3TU, UK}
\email{a.larkin2@lboro.ac.uk}

\author{Marks Ruziboev}
\address{Faculty of Mathematics, University of Vienna, Oskar Morgensternplatz 1, 1090 Vienna, Austria}
\email{marks.ruziboev@univie.ac.at}

\begin{document}

\maketitle
\begin{abstract}
\noindent In this work we consider i.i.d. random perturbations of contracting Lorenz maps sufficiently close to a Rovella parameter. We prove that the quenched correlations of the random dynamical system decays exponentially.
\end{abstract}

\section{Introduction}
\hspace{0.5cm } The Lorenz system
was introduced in \cite{Lo} as a simplified model for atmospheric convection.  Numerical simulations have shown that the Lorenz system admits a strange attractor, called the Lorenz attractor, which became one of the most iconic examples in the field. 

A rigorous mathematical approach was developed with the introduction of the so called geometric Lorenz flow by \cite{ABSh, GW79}, which mimicks simulation of the dynamics of the Lorenz flow, and which has a robust strange attractor under $C^1$ perturbations. Later in \cite{Tu1, Tu2} it was shown that the actual Lorenz attractor is indeed a singular hyperbolic attractor, further showing that the geometric Lorenz attractor represents the Lorenz attractor well. Moreover, it admits the so called Sinai-Ruelle-Bowen (SRB) measure, which is ergodic \cite{APPV}. Its statistical properties, such as mixing rates, limit theorems and their stability under various perturbations, were studied intensively (see for example, \cite{LMP, AMV,  AM2, AS2, BR, BMR, DO}). 

Another class of systems with similar properties was introduced in \cite{Rov} called the \textit{contracting Lorenz flow}. A fundamental difference between these is that the attractor of the system introduced by Rovella is not robust under perturbations, but still abundant in a measure theoretic sense. The set of measures for which the system is chaotic is called Rovella parameters and satisfies strong chaotic properties \cite{AS}; moreover, restricted to this set the system is stochastically stable \cite{M1, M2}. In \cite{PT}  the authors addressed thermodynamic formalism for it. Up to now, the contracting Lorenz flow and one dimensional maps with critical points remain a profound example of a truly nonuniformly hyperbolic systems, which is studied via construction of induced schemes. We refer to \cite{Jose} for a comprehensive account of these constructions. 

Recently, there has been increased interest in studying statistical properties of random dynamical systems, especially quenched (path-wise) properties. When the system has good uniformly hyperbolic properties, spectral techniques are still applicable and imply strong statistical properties; we refer to \cite{Buzzi, DFGV1, DFGV2, DS, DH}  and references therein for results on quenched decay of correlations, limit theorems and stability results in this case. 

For the non-uniformly expanding (or non-uniformly hyperbolic) case, spectral techniques are not applicable directly. In this regards, it is customary to employ inducing techniques, in particular randomised version of Young Tower \cite{Y} construction called random Young towers. This was first carried out in  \cite{BBMD} for random towers with exponential tails, where the results were applied to random compositions of unimodal maps.  Later, similar towers with polynomial tails were studied in \cite{BBR}. An alternative approach for random towers was given in \cite{Du}. More recently, random Young Towers for hyperbolic systems were introduced in \cite{ABR}, where the construction of such towers was carried out for nonuniformly hyperbolic random maps and in \cite{ABRV} for non-uniformly expanding systems with an ergodic base. As in the deterministic case, the construction of the Random Young Towers is highly non-trivial \cite{L}. 

In section 2, we introduce the setup of the random dynamical system we consider and state our main theorem, Theorem \ref{thm:1d}, and the strategy of its proof. In section 3, we prove Theorem \ref{thm:1d} in a series of steps.
 
\begin{figure}
\begin{center}
\begin{tikzpicture} 
	\draw (0,0) rectangle (8,8);
	\draw (4,0) -- (4, 8) ;
	\draw (0,4) -- (8, 4)  ;
	\draw (0,0) -- (8, 8);
	\draw[thick] (0,1) .. controls (2, 8) and (3, 8)  .. (4,8);
	\draw[thick] (4,0) .. controls (5, 0) and (6,0)  .. (8,7);
  	\draw (0,4) -- (0 ,4) node[anchor=east] {$-1$};
	\draw (4,4) -- (4,4) node[anchor=north west] {$0$};
	\draw (8,4) -- (8,4) node[anchor=west] {$1$};
	\draw (4,8)  node[anchor=south] {$1$};
	\draw (4,0)  node[anchor=north] {$-1$};
\end{tikzpicture}
\end{center}
\caption{A one-dimensional contracting Lorenz-like map.}\label{fig:2}
\end{figure}
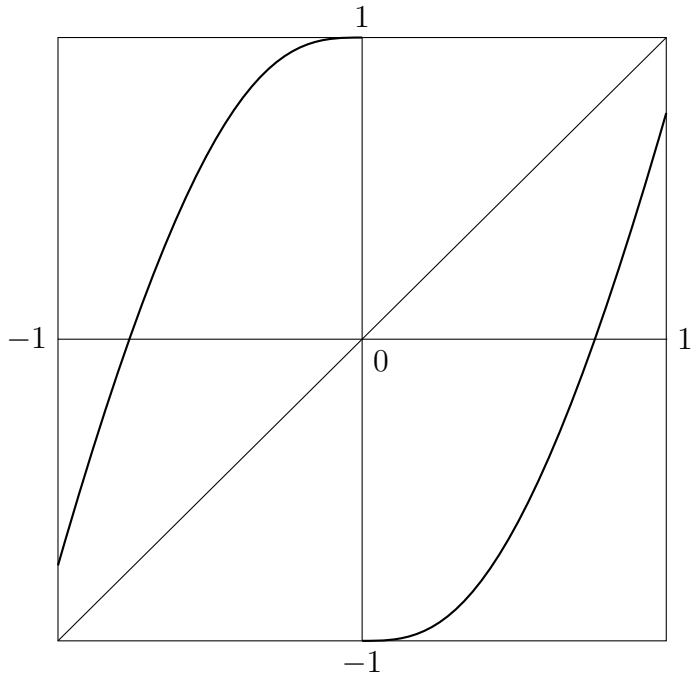
\section{Setup}
\subsection{Rovella maps}
	In the following section, we discuss two separate families of perturbed maps, the first one being obtained by perturbing some initial contracting Lorenz map (in the sense of \cite{Rov}), and the second one being obtained from perturbing an initial Rovella map taken from the first family of maps (in the sense of \cite{Shen}). This second family is the main focus of this paper, so for clarity, we denote maps in the first family using $f$ and maps in the second family by $T$.

The contracting Lorenz map $f: I \to I$, with $I = [-1,1]$, is piecewise $C^3$ which satisfies the following conditions:
\begin{itemize}
	\item[(C1)] $f$ has a singularity at the point $x=0$ with
	\begin{align}
		\lim_{x \to 0^+}f(x) = -1, \quad \lim_{x \to 0^-}f(x) = 1;
	\end{align} 
	\item[(C2)] $Df (x) > 0$ for all $x \neq 0$ with $\sup_{x \in (0,1]}Df(x) = Df(1)$, $\sup_{x \in [-1,0)}Df(x) = Df(-1)$ and
	\begin{align}
		\lim_{x \to 0}\frac{Df(x)}{|x|^{s-1}} \quad \text{exists and is nonzero for some } s > 1,
	\end{align}
	which implies there exist constants $K_1, K_2 > 0$ such that for all $x \in I \setminus \{0 \}$ we have
	\begin{align}
		K_1|x|^{s-1} \le Df(x) \le K_2|x|^{s-1};
	\end{align} 
	\item[(C3)] $f$ has negative Schwarzian derivative, i.e. for all $x \in I \setminus \{ 0 \}$ we have
	\begin{align}
		S(f) := D\Big( \frac{D^2f}{Df} \Big) - \frac{1}{2} \Big( \frac{D^2 f}{Df} \Big)^2 < 0.
	\end{align}

\end{itemize} 
Furthermore, in \cite{Rov} it was shown that there exists an interval $[0, a_*)  \subset \mathbb R$, where $a^*>0$ is small, and that there exists family of $C^3$ maps $\{f_a : I \to I \}_{a \in [0, a_*)}$, where $f_0$ corresponds to some initial unperturbed contracting Lorenz map satisfying conditions C1-C3, such that, for the remaining $a \in (0, a_*)$, all maps $f_a$ are close to the original $f_0$. Here, by `close' we mean their corresponding  flows from which they are projected are in a small $C^3$ neighborhood of the original contracting Lorenz flow. Moreover, there exists a Cantor-like set $E \subset (0, a^*)$ with positive Lebesgue measure such that, for all $a \in E$, $f_a$ is chaotic and satisfies the following conditions:
\begin{itemize}
\item[(R1)] there exists a constant $\lambda > 1$ such that for all $a \in E$, the points $x=1$ and $x=-1$ we have
\[
	Df_a^n(\pm 1) > \lambda^n, \quad \text{for all } n \ge 0;
\]
\item[(R2)] there exists a constant $\alpha >0$ such that for all $a \in E$ we have
\[
	|f_a^{n-1}(\pm 1)| > e^{-\alpha n}, \quad \text{for all } n \ge 1;
\]
\item[(R3)] the orbits of the points $x= \pm 1$ under $f_a$ are dense in $[-1, 1]$ for all $a \in E$.
\end{itemize} 
These maps are often referred to as \textit{Rovella maps}. Note that if $a \in (0, a_*) \setminus E$, then $f_a$ will not necessarily satisfy all of these conditions. Additionally, it is worth noting that property (R1) implies the following three conditions for any $a \in E$:
\begin{enumerate}
	\item \textbf{Large derivatives condition:} 
				$$
					\lim_{n \to \infty} | Df_a^n(T(0)) | = \infty;
				$$
	\item \textbf{Summability condition of exponent 1:} 
				$$
					\sum_{i=0}^\infty \frac{1}{| Df_a^n(T(0)) |} \le  \infty;
				$$
	\item \textbf{Collet-Eckmann condition:} 
				$$
					\liminf_{n \to \infty} \log \frac{| Df_a^n(T(0)) |}{n} >0.
				$$
\end{enumerate}

\subsection{Random dynamical systems setup}
Taking sufficiently small $\eps >0$, we assume that there exists a one-parameter  admissible family of maps $\{ T_t : I \to I  \}_{t \in [-\eps, \eps]}$, where we assume that for $t=0$ the map $T_0$ corresponds to some Rovella map $f_a$ with $a \in E$ (i.e. $T_0$ satisfies conditions C1-C3, R1-R3). We call a family \textit{admissible} if the following conditions hold:
\begin{enumerate}
	\item $|\partial_t F(t,x)| \le 1$ for all $x \in I$ and $t \in [-\eps,\eps]$, with $F(t,x) = T_t(x)$;\\
	\item for all $t \in [-\eps,\eps]$, the map $T_t$ will satisfy conditions C1 to C3;\\
	\item there exists a constant $C >0$ such that for all $x,y \in I$ and $t \in [-\eps,\eps]$, we have
				$$
					2 |x -y| < |x| \implies \Big| \log \frac{DT_t(x)}{DT_t(y)}   \Big| \le C \frac{|x - y|}{|x|}.
				$$
\end{enumerate}

Now, let $\nu_\eps$ denote the normalized Lebesgue measures on $[-\eps, \eps]$. Furthermore, we define $\Omega_\eps = [-\eps,\eps]^{\mathbb Z}$, and we take the product measure  $\mathbb P_\eps = \nu_\eps^{\mathbb Z}$ on $\Omega_\eps$. We define our random dynamical system with the usual skew product
	$$
		S : \Omega_\eps \times I \to \Omega_\eps \times I, \quad S(\omega, x) = (\sigma \omega, T_\omega (x)).
	$$

\subsection{Main Result}

Here we state the main result of this paper:

\begin{theorem}\label{thm:1d} 
Let $\eps >0$ be sufficiently small. The random contracting Lorenz system $\{T_\omega\}$ admits a unique equivariant family of absolutely continuous probability measures $\{\mu_\omega\}$. Moreover, there exists a $\mathbb P_\eps$-integrable constant $C(\omega)$, and there exists a constant $ b >0$ such, for almost every $\omega \in \Omega$, we have that for every $\varphi \in \C^{\eta}(I)$ with $\eta \in (0,1)$ and $\psi \in L^\infty(I)$, we have
\[
\left|\int(\varphi\circ T_\omega^n)\psi d\mu_\omega-\int \varphi d\mu_{\sigma^n\omega}\int \psi d\mu_\omega\right|\le C(\omega) C_{\varphi, \psi}  e^{-bn}
\] 
and 
\[
		\Big| \int (\varphi \circ T^n_{\sigma^{-n} \omega}) \psi d\mu_{\sigma^{-n}\omega} - \int \varphi d\mu_\omega \int \psi d\mu_{\sigma^{-n}\omega} \Big| \le C(\omega) C_{\varphi, \psi} e^{-bn}
\] 
for some $C_{\varphi, \psi}>0$ which only depends on $\varphi$ and $\psi$ and is uniform for all $\omega \in \Omega$. 

 \end{theorem}

\begin{remark}
	One should note that even though $T_0$ satisfies conditions (R1)-(R3), the other maps in the family do not necessarily satisfy these conditions. Because of this, for certain random compositions of maps, we may need to wait a long time before we see a map with the desired properties. This means that we will need to introduce a waiting time before we see the desired asymptotic behaviour for the tails of return times and the decay of correlations. We express this equivalently using a large constant factor $C(\omega)$ in the upper bound of the random correlation function for almost every $\omega \in \Omega_\eps$.
\end{remark}

\subsection{Strategy of the Proof}

We prove Theorem \ref{thm:1d} by showing that the random system $\{T_\omega\}$ admits a random Young tower structure in the sense of \cite{BBMD, Du} for almost every $\omega\in\Omega$ with exponential tails of the return times. For convenience, in the Appendix we include the definition of random Young towers as well as Theorem \ref{thm:randexp} from \cite{BBMD, Du}, which we use in the proof of Theorem \ref{thm:1d}. 

To build a random Young tower, we use a similar construction to that in Section 8.3 of \cite{Du}, where a random Young tower is constructed for a system of random unimodal maps. However, before we can use such a construction, we must prove that certain results hold for our random contracting Lorenz system that are analoguous to the results proved in \cite{Shen} for random non-uniformly expanding interval maps. In particular, we must prove results regarding random non-uniform expansion and the tails of `bad' $\omega$'s, where expansion is not guaranteed. Once we have these, then we move onto the tower construction, which involves the use of \textit{hyperbolic return times}, originating from the works of Alves.

Thus, the structure of the rest of the paper is as follows:

\begin{enumerate}
	\item Prove that if the random orbit stays outside of a particular neighborhood around the origin (which we will use as the base of our random tower), then we have exponential expansion. This will involve first proving some results regarding the deterministic dynamics;
	\item Prove that the set of `good' $\omega$'s grows to a full measure subset of $\Omega_\eps$ as $n \to \infty$, thereby guaranteeing that we eventually have expansion for almost all $\omega \in \Omega_\eps$;
	\item Prove that we have exponential tails of hyperbolic return times and that the hyperbolic return times of a point $x$ in the base give us a neighborhood around $x$ on which $T_\omega^n(x)$ is diffeomorphic, has bounded distortion, and uniform expansion. These neighborhoods will be used to define the return partition;
	\item Prove that we can construct a return partition using the neighborhoods from the previous step, and that we have exponential tails of return times along with the other necessary randon tower conditions to apply Theorem \ref{thm:randexp}.
\end{enumerate}

\section{Proofs}

\subsection{Random Expansion}

In \cite{AS}, Alves and Soufi already proved expansion results for a Rovella map with a `good' parameter $a \in E$ when the orbit stays outside of a neighbourhood of the form $(- \delta, \delta)$ for sufficiently small $\delta >0$. This can easily be further extended to random orbits using a continuity argument, which we will show later in this section. However, for our tower construction, we will require expansion results of a different form, which we state below.

Let us denote $\tilde B(\delta) = T^{-1}(B_\delta(1)) \cup T^{-1}(B_\delta(-1))$ and $\D(\delta) = \frac{|B_\delta(0)|}{|\tilde B(\delta)|}$. Additionally, throughout this chapter we assume that there exists a small constant $\delta_* >0$ such that for every $x \in \tilde B(\delta_*)$, $\delta = \max \{ x, \delta_* \}$ and $\omega \in \Omega_\delta$, we have
\begin{align}
		DT_\omega(x) \ge \D(\delta). \label{delta_star}
\end{align}

We prove the following Proposition, which is an adaption of Theorem 2.1 in \cite{Shen} and Lemma 8.1.6 in \cite{Du}:

\begin{proposition}\label{prop:nonuniform_expansion}
	There exists uniform constants $A, \kappa>0$, and for sufficiently small $\delta >0$ there exists a constant $\Lambda(\delta) >0$ satisfying $\lim_{\delta \to 0} \Lambda(\delta) = \infty$ such that for all sufficiently small $\eps >0$, all $\omega \in \Omega_\eps$ and all $x \in I$, we have
	\begin{enumerate}
		\item if $x, T_\omega(x), \dots, T_\omega^{n-1}(x) \notin \tilde B(\delta)$ and $T_\omega^n(x) \in \tilde B(2\delta)$, then 
			\begin{align}\label{ineq:nue2}
				|DT_\omega^n(x)| \ge A  e^{\kappa n};
			\end{align}
		\item if $|x - T(0)| \le 4 \delta$, $x, T_\omega(x), \dots, T_\omega^{n-1}(x) \notin \tilde B(\delta)$ and $T_\omega^n(x) \in \tilde B(2\delta)$, then
			\begin{align}\label{ineq:nue3}
				|DT_\omega^n(x)| \ge \frac{\Lambda(\delta)}{\D(\delta)} e^{\kappa n}.
			\end{align}
	\end{enumerate}
\end{proposition}

To prove this, we follow the steps used by Shen in \cite{Shen} to prove similar estimates for random non-uniformly expanding interval maps.

First, let us introduce the relevant notation and prove a series of smaller initial lemmas.  Set $I_+ = (0, 1]$ and $I_- = [-1, 0)$. Given some subset $J$, satisfying either $J \subset I_+$ or $J \subset I_-$ for every $\omega \in \Omega$, we define
$$
	\text{Dist}(T_\omega | J) := \sup_{x,y \in J} \log \frac{DT_\omega(x)}{DT_\omega(y)}
$$
as well as
$$
	\mathcal{N} (\varphi | J) := \sup_{J' \subset J} \text{Dist}(T_\omega | J') \frac{|J|}{|J'|}.
$$
Additionally, we set
$$
	A(\omega, x, n) = \sum_{i = 0}^{n-1} \frac{DT_\omega^i(x)}{|T_\omega^i(x)|}.
$$
\begin{lemma}\label{lemma:shen_2.3}
	For our family of maps $\{ T_\omega \}_{\omega \in \Omega}$, there exists a constant $\theta_0>0$ (independent of $\eps$) such that for all $\omega \in \Omega$ and $x \in I$, and for any $n \in \mathbb Z_+$ such that $A(\omega, x , n) < \infty$, setting
	$$
		J = \Big[ x - \frac{\theta_0}{A(\omega, x , n)},   x + \frac{\theta_0}{A(\omega, x , n)} \Big] \cap I_\pm,
	$$
we have that $T_\omega | J$ is a diffeomorphism and that $\mathcal N (T_\omega^n | J) \le 1$. Furthermore, for every $y \in J \setminus \{0\}$, we have
	\begin{align} \label{ineq:dist}
		e^{-1} \frac{DT_\omega^n(x)}{A(\omega, x, n)}  \le \frac{DT_\omega^n(y)}{A(\omega, y, n)} \le e \frac{DT_\omega^n(x)}{A(\omega, x, n)}.
	\end{align}
\begin{proof}
Let $\omega \in \Omega$ and let $L \subset I$ be some subinterval in either $I_-$ or $I_+$.  Recall the definition of $\mathcal N (T_\omega| L)$:
	$$
		\mathcal{N} (T_\omega | L) = \sup_{L' \subset L}\text{Dist}(T_\omega | L')  \frac{|L|}{|L'|} = \sup_{L' \subset L} \sup_{x_1,x_2 \in L'} \log \frac{DT_\omega(x_1)}{DT_\omega(x_2)} \frac{|L|}{|L'|}.
	$$
We want to show that for any $y \in L$, we have
\begin{align}\label{ineq:dist}
	2|L| \le |y| \implies \mathcal{N} (T_\omega | L) \le  C \frac{|L|}{|y|}.
\end{align}
First, we should confirm that $\mathcal{N} (T_\omega | L)$ actually has a finite value for all $L$ either in $I_-$ or $I_+$. For a given $L$, we may have the trivial case where there does exist a subinterval $L' \subset L$ of positive length that maximizes $\text{Dist}(T_\omega | L')/|L'|$, in which case $\mathcal{N} (T_\omega | L)$ is clearly finite. Now, consider the less trivial and more likely case that there exists no such $L'$. If so, we are effectively taking 
$$
	 \mathcal{N} (T_\omega | L) = |L| \cdot \Big( \lim_{h \to 0}  \sup_{\substack{ L' \subset L \\ |L'| = h }} \sup_{x_1,x_2 \in L'}  \log \frac{DT_\omega(x_1)}{DT_\omega(x_2)}  \frac{1}{h} \Big).
$$ 
Since the derivative is monotonically decreasing and increasing on $I_-$ and $I_+$ respectively, then the $x_1$ and $x_2$ that maximize the distortion are just the end points of $L'$, i.e. $|x_2| = |x_1| - h$. Thus, we have 
$$
	 \mathcal{N} (T_\omega | L) = |L| \cdot \Big( \lim_{h \to 0}  \sup_{x_1, \in L'} \big( \log DT_\omega(x_1) - \log DT_\omega(x_1 \pm h)  \big) \frac{1}{h} \Big).
$$ 
But this is just the derivative of $\log DT_\omega(x_1)$, and since we already know $T_\omega$ is $C^3$ for all $\omega \in \Omega$, we therefore have
\begin{align}\label{eq:dist}
 	\mathcal{N} (T_\omega | L) = \sup_{x_1 \in L} \frac{D^2T_\omega(x_1)}{DT_\omega(x_1)} |L|.
\end{align}
Notice that by using the order of singularity condition, we have
\begin{align*}
	\mathcal{N} (T_\omega | L) &\le \sup_{x_1 \in L} \frac{(s-1)K_2 |x_1|^{s-2}}{K_1 |x_1|^{s-1}} |L| 
	= \frac{K_2(s-1)}{K_1}  \sup_{x_1 \in L} \frac{1}{|x_1|}|L|.
\end{align*}
Now, clearly the $x_1 \in L$ that maximizes this expression is the one closest to $0$. Thus, we have $|x_1| \ge |y| - |L|$. Furthermore, using the assumption $2 |L| \le |y| $, we therefore have $|x_1| \ge 1/2|y|$. Thus, we have
\begin{align*}
	\mathcal{N} (T_\omega | L) &\le \frac{K_2(s-1)}{K_1} \frac{2}{|y|}|L|.
\end{align*}

We can assume that $C$ is small, say smaller than $1.5$. Therefore, we can choose some $\theta_0 \in (0, (6e)^{-1})$ such that
\begin{align*}
	2|L| \le |y| \implies \mathcal{N} (T_\omega | L) &\le   \frac{1}{4e\theta_0} \frac{|L|}{|y|}.
\end{align*}

Now, let us set $n_0$ as the largest integer in the set $\{ 1, 2, \dots , n  \}$ such that
	$$
		\sum_{i=0}^{n_0-1} \frac{|T_\omega^i(J)|}{|T_\omega^i(x)|} \le 2e \theta_0 < \frac 1 3.
	$$
Note that $|J| = 2\theta_0 / A(\omega, x, n)$. Also note that for any $\omega \in \Omega$ and $x \in I$, we have $A(\omega, x , 1) = 1$. Thus, the above inequality holds for $n_0=1$.

In fact, we want to show that $n_0=n$. First, we note that for all $0 \le i < n_0$, the above inequality implies that $2|T_\omega^i(J)| \le |T_\omega^i(x)|$, otherwise one of the terms in the sum would be equal or greater than $1/2$, thereby breaking the above inequality. Thus, for every $1 \le m \le n_0$, we have 
\begin{align*}
	\text{Dist}(T_\omega^m | J) &= \sup_{y, z \in J} \log \frac{DT_\omega^m(x)}{DT_\omega^m(y)} \\
	&=  \sup_{y, z \in J} \log \Big( \prod_{j=0}^{m-1 } \frac{DT_{\sigma^j \omega}(x)}{DT_{\sigma^j \omega}(y)} \Big)
	 = \sup_{y, z \in J}  \sum_{j=0}^{m-1} \log   \frac{DT_{\sigma^j \omega}(x)}{DT_{\sigma^j \omega}(y)} \\
	& \le  \sum_{j=0}^{m-1} \text{Dist}\big(T_{\sigma^j \omega} | T_\omega^j (J) \big)
	 \le \sum_{j=0}^{m-1} \mathcal{N} \big(T_{\sigma^j \omega} | T_\omega^j (J) \big) \\
	& \le  \frac{1}{4e\theta_0} \sum_{j=0}^{m-1} \frac{|T_\omega^i(J)|}{|T_\omega^i(x)|}
	 \le \frac 1 2.
\end{align*}
Clearly, this implies using the mean value theorem that there exists some $y \in J$ such that
$$
	DT_\omega^m(y) = \frac{|T_\omega^m(J)|}{|J|},
$$
and therefore $|T_\omega^m(J)| \le e (T_\omega^m)'(x) |J|$. Now, for contradiction, assume $n_0<n$. Consider
\begin{align*}
	\sum_{i=0}^{n_0} \frac{|T_\omega^i(J)|}{|T_\omega^i(x)|} &\le \sum_{i=0}^{n_0} \frac{e DT_\omega^m(x) |J|}{|T_\omega^i(x)|}\\
	&= e A(\omega, x, n_0 + 1)|J|
	\le 2 e \theta_0.
\end{align*}
But this contradicts the definition of $n_0$. Thus, it must be that $n_0 = n$.

For any subinterval $J' \subset J$, using the bounded distortion we have $|T_\omega^i(J')| \le e |J'| |T_\omega^i(J)|/|J|$. Thus, we have
\begin{align*}
	\text{Dist}(T_\omega^n | J') &\le \frac{1}{4e\theta_0} \sum_{j=0}^{m-1} \frac{|T_\omega^i(J')|}{|T_\omega^i(x)|}\\
	&  \le \frac{1}{4e\theta_0} \sum_{j=0}^{m-1} \frac{|J'|}{|J|} \frac{|T_\omega^i(J)|}{|T_\omega^i(x)|}
	\le \frac{|J'|}{2|J|}.
\end{align*}
Thus, we have
$$
	\mathcal{N} (\varphi | J) := \sup_{J' \subset J} \text{Dist}(\varphi | J') \frac{|J|}{|J'|} \le  \sup_{J' \subset J} \frac{|J'|}{2|J|} \frac{|J|}{|J'|} \le \frac 1 2 < 1.
$$

Now, notice that for any $y \in J$ and for $0 \le m < n$ we have
\begin{align*}
	\Big| \frac{|T_\omega^m(y)|}{|T_\omega^m(x)|} - 1   \Big| &=  \frac{|T_\omega^m(y) - T_\omega^m(x)|}{|T_\omega^m(x)|}    
	 \le  \frac{|T_\omega^m(J) |}{|T_\omega^m(x)|}  
	 \le \frac 1 2,
\end{align*}
which clearly implies $ 1 / 2 \le {|T_\omega^m(y)|}/{|T_\omega^m(x)| \le 3/2} $.  Using the bounded distortion, we have $e^{-1/2}(T_\omega^m)'(x) \le  (T_\omega^m)'(y) \le e^{1/2}(T_\omega^m)'(x)   $. Thus, combining these two, we obtain
$$
	e^{-1} \frac{DT_\omega^m(x)}{|T_\omega^m(x)|} \le \frac{DT_\omega^m(y)}{|T_\omega^m(y)|} \le e \frac{DT_\omega^m(x)}{|T_\omega^m(x)|}
$$
To prove inequality (\ref{ineq:dist}), we work with the reciprocal since it is slightly easier to work with:
\begin{align*}
	\frac{A(\omega, y, n)}{DT_\omega^n(y)} = \frac{\sum_{i =0}^{n-1} \frac{DT_\omega^i(y)}{|T_\omega^i(y)|}}{DT_\omega^n(y)} \le \frac{\frac{3}{2}e^{1/2}\sum_{i =0}^{n-1} \frac{DT_\omega^i(x)}{T_\omega^i(x)}}{e^{-1/2}|T_\omega^n(x)|} = \frac{3}{2}e\frac{A(\omega, x, n)}{DT_\omega^n(x)}.
\end{align*}
Likewise for the lower bound.
\end{proof}
\end{lemma}

It will also be useful to introduce the concept of a \textit{binding period} along with a relevant lemma:

\begin{definition}[Binding Period]
	For a given $\eps >0$, $v \in I\setminus \{ 0 \}$ and $C >0$, we say that $N \in \mathbb Z_+$ is a $C$-binding period for $(v, \eps)$ if for every $y \in I\setminus \{ 0 \}$ such that $|v - y| \le \eps$ and for every $0 \le j < N$ and $\omega \in \Omega$, we have the following:
	\begin{align}
		2 |T_\omega^j(y) - T^j(v)| &\le |T^j(v)| \label{def:bind1}; \\
		e^{-1} |DT^{j+1}(v)| &\le  |DT_\omega^{j+1}(y)| \le e |DT^{j+1}(v)| \label{def:bind2}; \\
		C \eps |DT^{j+1}(v)| &\ge |T_\omega^{j+1}(y) - T^{j+1}(v)  | \label{def:bind3}.
	\end{align}
\end{definition}

\begin{lemma} \label{shen:lemma_2.5}
	There exists a constant $\theta_1 >0$ such that for sufficiently small $\eps>0$, for any point $v \in [-1,1]\setminus \{0\} $, if there exists $N \in \mathbb Z_+$ such that
$$
	W := \sum_{j=0}^N |DT_\omega^j(v)| < \infty \text{ and } A(0, v, N)W \le \delta_1/\eps.
$$ 
Then $N$ is a $\eps W$- binding period for $(v, \eps).$

\begin{proof}
	Using the same proof as in one of the previous lemmas, we know there exists some $\theta_1 \in (0, (4e)^{-1})$ such that for any interval $L \subset [-1, 1] \setminus \{0 \}$, $z \in L$ and $\omega \in \Omega$, we have
\begin{align}\label{ineq:dist2}
	2|L| \le |z| \implies \mathcal{N} (T_\omega | L) \le  \frac{1}{2 e \theta_1} \frac{|L|}{|z|}.
\end{align}
Now, given some $y \in [-1, 1] \setminus \{0 \}$ such that $|v - y| \le \eps$ and some $\omega \in \Omega$, we define $L_i = [v_i, y_i]$ (or alternatively $L_i = [ y_i, v_i]$), where $v_i = T^i(v)$ and $y_i = T_\omega^i(y)$. Trivially we have
$$
	|I_{i+1}| \le |T(I_i)| + \eps.
$$
Now, we want to show that the following three inequalities hold for all $n$:
\begin{align}
	|I_n| \le e W \eps DT^n(v); \label{ineq:bind1} \\
	\sum_{0 \le i < n} \frac{|I_i|}{|T^i(v)|} \le e \theta_1;  \label{ineq:bind2} \\
	\sum_{0 \le i < n} \text{Dist} (T | I_i) \le \frac 1 2. \label{ineq:bind3}
\end{align}

For $n=0$ this is trivial. We have 
$
	|I_0| = |v - y| \le \eps \le e W \eps; \quad
	0 \le e \theta_1; \quad
	0 \le \frac 1 2.
$
Now, we prove the statement by induction and assume that we have proved this up to some time $0 \le n_0 < N$. For $n = n_0 + 1$, the proof is straightforward:
\begin{align*}
	\sum_{i = 0}^{n_0} \frac{|I_i|}{|v_i|} &\le  e W \eps \sum_{i = 0}^{n_0} \frac{ DT^n(v)}{|v_i|}  \\
	& \le e W \eps A(0, \omega, n_0 + 1) 
	 \le e W \eps A(0, \omega, N) 
	 \le e \theta_1. 
\end{align*}
For the first inequality we have used the fact that inequality \ref{ineq:bind1} holds for $n_0$, and for the last inequality we have used our starting assumption. Note that this implies that for $0 \le i \le n_0$, we have $2|I_i|\le |v_i|$. Thus, we can apply inequality \ref{ineq:dist2} to obtain the following:
\begin{align*}
	\sum_{0 \le i < n_0 + 1} \text{Dist} (T | I_i) \le \frac{1}{2e \theta_1} \sum_{0 \le i < n_0 + 1} \frac{|I_i|}{|v_i|} \le \frac 1 2 < 1,
\end{align*}
where in the last step we just use the fact that inequality \ref{ineq:bind3} holds for $n_0$.

Finally, by the mean value theorem there exists some $\xi_i \in I_i$  for all $i = 0, 1 , \dots , n_0$ such that $|I_{i+1}| \le DT(x_i)|I_i| + \eps$. Iterating this, we have
\begin{align*}
	|I_{n_0 + 1}| \le \prod_{i=0}^{n_0} DT(x_i)|I_0| + \eps \Big(1 + \sum_{k=1}^{n_0} \prod_{j = k}^{n_0} DT(x_j)  \Big).
\end{align*}
As we have already shown that inequality \ref{ineq:bind3} holds for $n = n_0 + 1$, we thus have for all $0 \le k \le n_0$
\begin{align*}
	\prod_{i = k}^{n_0} DT(x_i) \le e  \prod_{i = k}^{n_0} DT(v_i) = e \frac{DT^{n_0 + 1}(v)}{DT^k(v)}
\end{align*}
Plugging this into the previous inequality, we have
\begin{align*}
	|I_{n_0+1}| &\le  \prod_{i = 0}^{n_0} e DT(v_i) |I_0| + \eps \Big( 1 + \sum_{k=1}^{n_0}  \prod_{j = k}^{n_0} DT(x_j) \Big) \\
	&= e DT^{n_0 + 1}(v) \Big( |I_0| + \frac{\eps}{e  DT^{n_0 + 1}(v)}  +   \sum_{k=1}^{n_0} \frac{\eps}{DT^k(v)}  \Big) \\
	& \le e \eps DT^{n_0 + 1}(v) \Big( 1 + \sum_{k=1}^{n_0+1}\frac{1}{(T^k)'(v)}  \Big) 
	 \le e W DT^{n_0 + 1}(v) \eps.
\end{align*}

Our focus now is to show these imply that $eW$ is a binding time for sufficiently small $\eps >0$. Condition \ref{def:bind1} is clearly implied by inequality \ref{ineq:bind2}, and likewise for condition \ref{def:bind2} and inequality \ref{ineq:bind3}. To prove the remaining condition, we need to show that for $0 \le j \le N-1$ we have
\begin{align}
	\log \Big( \frac{|DT_\omega^{j+1}(y)|}{|DT^{j+1}(v)|} \Big) \le 1.
\end{align}
We split the above in the following way:
\begin{align*}
	\log \Big( \frac{|DT_\omega^{j+1}(y)|}{|DT^{j+1}(v)|} \Big) &=  \sum_{i = 0}^j \log \Big( \frac{|DT_{\sigma^i \omega}(y_i)|}{|T'(v_i)|}  \Big) \\
	& = \sum_{i = 0}^j \log \Big( \frac{|DT_{\sigma^i \omega}(y_i)|}{|DT(v_i)|} \cdot \frac{|DT(y_i)|}{|DT(y_i)|} \Big)\\
	& = \sum_{i = 0}^j \log \Big( \frac{|DT(y_i)|}{|DT(v_i)|}  \Big) +  \sum_{i = 0}^j  \log \Big( \frac{|DT_{\sigma^i \omega}(y_i)|}{|DT(y_i)|}  \Big).
\end{align*}
For the first sum, we already know from \ref{ineq:bind3} that
\begin{align*}
	 \sum_{i = 0}^j \log \Big( \frac{DT(y_i)}{DT(v_i)}  \Big) \le  \sum_{i = 0}^j \text{Dist}(T | I_i) \le \frac 1 2.
\end{align*}
Thus, we need to show the same upper bound for the second sum. To this end, we use the fact that $\log(1 + x) \le x$ for $x \ge 0$ to obtain the following:
\begin{align*}
	 \sum_{i = 0}^j  \log \Big( \frac{|DT_{\sigma^i \omega}(y_i)|}{|T'(y_i)|}  \Big) &=   \sum_{i = 0}^j  \log \Big(  \frac{|DT_{\sigma^i \omega}(y_i)|}{|DT(y_i)|}  + \frac{|DT(y_i)|}{|DT(y_i)|} - \frac{|DT(y_i)|}{|DT(y_i)|} \Big)\\
	& \le \sum_{i = 0}^j  \log \Big(1 +  \frac{|DT_{\sigma^i \omega}(y_i) - DT(y_i)|}{|DT(y_i)|}  \Big) \\
	&\le \sum_{i = 0}^j \frac{|DT_{\sigma^i \omega}(y_i) - DT(y_i)|}{|DT(y_i)|} 
	 \le \sum_{i = 0}^{N-1} \frac{e \eps}{|DT(v_i)|},
\end{align*}
where in the last inequality we use bounded distortion and the assumption that for all $x \in I \setminus \{ 0 \}$ and al $\omega, \tilde \omega \in \Omega_\eps$, we have $|DT_{\omega}(x) - DT_{\tilde \omega}(x)| \le \eps$. Now, using this Lemma's assumptions and the definitions of $W$ and $A$,  we have
\begin{align*}
	\frac{\theta_1}{\eps} \ge A(0, v, N)W &= \Big( \sum_{i=0}^{N-1} \frac{DT^i(v)}{|v_i|}\Big) \Big( \sum_{i=0}^{N-1}  \frac{1}{DT^{i+1}(v)} \Big) \\
	& \ge \sum_{i=0}^{N-1} \frac{DT^i(v)}{|v_i|} \frac{1}{DT^{i+1}(v)} 
	= \sum_{i=0}^{N-1} \frac{1}{DT(v_i)|v_i|} 
	\ge \sum_{i=0}^{N-1} \frac{1}{DT(v_i)}.
\end{align*}
Thus, 
\begin{align*}
	\sum_{i = 0}^j  \log \Big( \frac{|DT(T_\omega^i(y))|}{|DT(T^i(y))|}  \Big) \le e \theta_1 < \frac 1 2,
\end{align*}
thus completing the proof.
\end{proof}
\end{lemma}

Now, we recall that condition H2. above, proved in \cite[Lemma 2.1]{AS} by Alves and Soufi, tells us that the deterministic system $T_0$ has expansion when the orbit stays outside a small neighbourhood of $0$. Analogously to \cite[Proposition 2.7 (i)]{Shen}, we prove that the same holds true for our perturbed system as well:

\begin{lemma} \label{Shen:Prop_2.7i}
There exists a constant $\eta >0$, and for sufficiently small  $\delta > 0$ there exists a constant $C=C(\delta)>0$, such that, for all sufficiently small $\eps >0$, all $\omega \in \Omega_\eps$, all $n\ge 1$ and all $x \in I \setminus\{0\}$, if $x , T_\omega(x), \dots, T_\omega^{n-1}(x) \notin (-\delta, \delta)$, then $DT_\omega^n(x) \ge C e^{\eta n}$.

\begin{proof}
Assume we have $\delta>0$ already fixed sufficiently small. For the first result, we already know from condition H2) that if $x, T(x), \dots, T^{n-1}(x) \notin (-\delta, \delta)$, then $DT^n(x) \ge c(\delta) \lambda^n$. Now, depending on how small our constant $c(\delta)$ is, let us fix a sufficiently large iterate M. Clearly, for sufficiently small $\eps$, for each $\omega \in \Omega_\eps$ the orbit $\big\{ x, T_\omega(x), \dots, T_\omega^{M-1}(x) \big\} $ will not deviate too much from the unperturbed orbit, especially since we are only considering finite iterates. Thus, if $DT^M(x) \ge c(\delta) \lambda^M$, then by continuity we can say that $DT_\omega^M(x) \ge  c(\delta) \hat \lambda^M$, where $\lambda > \hat \lambda > 1$ is some smaller suitably chosen constant. 

Now, suppose $n \le M$, and suppose $x, T_\omega(x), \dots, T_\omega^{n-1}(x) \notin (-\delta, \delta)$. We have already assumed that $\eps$ is small enough so that all perturbed orbits are close to the unperturbed orbit up to iterate $M-1$.  However, one must note that $x, T_\omega(x), \dots, T_\omega^{n-1}(x) \notin (-\delta, \delta)$ does not necessarily imply $x, T(x), \dots, T^{n-1}(x) \notin (-\delta, \delta)$ because it may be that there exists some iterate $1\le j \le n-1$ such that $T^j(x) \in (-\delta, \delta)$, but $T^j(x)$ lands close enough to edge of $(-\delta, \delta)$ that the perturbed orbit can land outside. But this can be handled, as we can simply decompose the unperturbed orbit, with $x, T(x), \dots, T^{j-1}(x) \notin (-\delta, \delta)$ and $T^j(x) \in (-\delta, \delta)$, implying $DT^j(x) \ge \lambda^j$, and then  $T^{j+1}(x), \dots, T^{n-1}(x) \notin (-\delta, \delta)$ implying $DT^{n-(j+1)}(T_\omega^{j+1}(x)) \ge c(\delta) \lambda^j$, and by using continuity again, we have $x, T_\omega(x)$, $\dots$, $T_\omega^{n-1}(x) \notin (-\delta, \delta)$ implies $DT_\omega^n(x) \ge c(\delta) \hat \lambda^n$. Additionally, if $T_\omega^n(x) \in (-\delta, \delta)$, again this does not imply that $T^n(x) \in (-\delta, \delta)$ because it may be that $T^n(x)$ lands just outside $(-\delta, \delta)$ but the perturbation is large enough that $T^n(x) \in (-\delta, \delta)$. Thus, we cannot claim $DT^n(x) \ge \lambda^n$.  Nonetheless, by the size of $\eps$, $DT^n(x)$ must be extremely close to $(-\delta, \delta)$, so by continuity we can instead say $DT^n(x) \ge \tilde \lambda^n$ for some $\lambda > \tilde \lambda > 1$, and thus $DT_\omega^n(x) \ge \hat \lambda^n$ for $\tilde \lambda > \hat \lambda > 1$.

Now assume that $M <n \le 2 M$. After $M-1$ iterates, we lose that the perturbed orbits of our original $x$ are close to the unperturbed orbit of $x$, but this is not an issue. Take $x$ such that $x, T_\omega(x), \dots, T_\omega^{n-1}(x) \notin (-\delta, \delta)$, and let us decompose its orbit into $x, T_\omega(x), \dots, T_\omega^{M-1}(x) \notin (-\delta, \delta)$ and $T_\omega^M(x), \dots, T_\omega^{n-1}(x) \notin (-\delta, \delta)$. Along the first part of the perturbed orbit, clearly we have $DT_\omega^M(x) \ge c(\delta) \hat \lambda^M$. As for the second part, let us take $y = T_\omega^M(x)$ and then consider its unperturbed orbit $\{y, T(y) , \dots , T^{n-1 -M}(y)\}$. We can then use the exact same continuity argument as before to show that $DT^{n-M}(y) \ge c(\delta)\lambda^{n-M}$ implies $DT_{\sigma^M \omega}^{n-M}(T_\omega^M(x)) \ge c(\delta) \hat \lambda^{n-M}$. Thus, we have $DT_\omega^n(x) = \big[DT_\omega^M(x) \big] \cdot \big[ DT_{\sigma^M \omega}^{n-M}(T_\omega^M(x))  \big] \ge c(\delta)^2  \hat \lambda^n$.  Furthermore, we can use the same argument as above to show that if $T_\omega^n(x) \in (-\delta, \delta)$, then $DT_\omega^n(x) \ge c(\delta) \hat \lambda^n$.

For any $\ell  M < n \le (\ell +1) M$, for $x, T_\omega(x), \dots, T_\omega^{n-1}(x) \notin (-\delta, \delta)$ we simply decompose the orbit as before, giving us $DT_\omega^n(x)  \ge (c(\delta) \hat \lambda^M)^\ell c(\delta) \hat \lambda^{n- \ell M} =  c(\delta)^{\ell+1} \hat \lambda^n$. Clearly this $c(\delta)^{\ell +1}$ may be an issue if it shrinks too quickly, so we choose $M$ large enough and $\eps$ small enough (which prevents $\hat \lambda$ from becoming too small) so that $c(\delta) \hat \lambda^M \ge 1$. Thus, choosing an even smaller $\hat \lambda > \hat \lambda_0 > 1$ so that $ \hat \lambda_0^{\ell M} \le (c(\delta) \hat \lambda^M)^\ell$, we have $DT_\omega^n(x)  \ge  c(\delta) \hat \lambda_0^n$. Furthermore, using this and the argument above, if $T_\omega^n(x) \in (-\delta, \delta)$, then $DT_\omega^n(x) \ge \hat \lambda_0^n$.
\end{proof}
\end{lemma}

Additionally, we want to prove the tails of perturbed return times decays exponentially as well:

\begin{lemma}\label{lemma:tail_of_return_times}
There exists a constant $\eta >0$, and for sufficiently small  $\delta > 0$ there exists a constant $C=C(\delta)>0$, such that, for all sufficiently small $\eps >0$, all $\omega \in \Omega_\eps$, and all $n\ge 1$, we have
		$$
			\big|\big\{ x \in I: \Hquad T_\omega^j(x) \notin (- \delta, \delta) \text{ for } 0 \le j \le n-1 \big\}\big| \le C e^{- \eta n}
		$$
\begin{proof}
Let $\delta>0$ and $\eps>0$ already be fixed sufficiently small. Setting $U=(-\delta, \delta)$, for every $\omega \in \Omega_\eps$ and $n \ge 1$ we define
$$
	\Lambda_n^\omega (U) = \big\{ x \in I  : \Hquad T_\omega^j(x) \notin U \text{ for } 0 \le j \le n-1 \big\}
$$
as well as $\Lambda_\infty^\omega (U) = \cap_{n=1}^\infty \Lambda_n^\omega (U)$. Furthermore, let $U_0 \subset U$ also be an open neighbourhood of the critical point at $0$.

We note that there exists constants $C_1, C_2 >0$ such that for all $x \in \Lambda_n^\omega (U)$, we have
$$
	C_1 DT_\omega^n(x) \le A(\omega, x, n) \le C_2  DT_\omega^n(x).
$$
Indeed, recalling the definition of $A(\omega, x, n)$, we have
\begin{align*}
	A(\omega, x, n)& = \sum_{i = 0}^{n-1} \frac{DT_\omega^i(x)}{|T_\omega^i(x)|} 
	 = \frac{1}{|x|} + \frac{DT_\omega(x)}{|T_\omega(x)|} + \dots + \frac{DT_\omega^{n-1}(x)}{|T_\omega^{n-1}(x)|}.
\end{align*}
Since by definition of $x \in \Lambda_n^\omega (U)$ we have $x, T_\omega(x), \dots , T_\omega^{n-1}(x) \notin U$, we clearly have
\begin{align*}
	A(\omega, x, n) &\le \frac 1 \delta \Big( 1 + DT_\omega(x) + \dots + DT_\omega^{n-1}(x) \Big) \\
	& \le \frac 1 \delta \Big( 1 + (n-1)\sup DT_\omega  \Big) \\
	& \le C_1 c(\delta) e^{\eta n}
	 \le C_1 DT_\omega^n(x)
\end{align*}
The second-last inequality holds because the line above grows linearly with $n$, while the last line grows exponentially, thus we can just choose some large constant $C_1$ and put it in front. The last inequality holds by the previous lemma. The lower bound is trivial:
\begin{align*}
	\Big( \frac{1}{|x|} + \frac{DT_\omega(x)}{|T_\omega(x)|} + \dots + \frac{DT_\omega^{n-1}(x)}{|T_\omega^{n-1}(x)|} \Big) &\ge \Big( 1 + DT_\omega(x) + \dots + DT_\omega^{n-1}(x) \Big) \\
	& \ge DT_\omega^{n-1}(x) 
	 \ge C_2 DT_\omega^n(x),
\end{align*}
where for the last inequality we just use the fact that $DT_\omega$ is bounded for all $\omega \in \Omega_\eps$, i.e. $C_2 = 1/\sup_{\omega \in \Omega_\eps}\sup T_\omega'$.

Now, by Lemma \ref{lemma:shen_2.3}, we can choose some constant $\tau >0$ such that for all $x \in I \setminus \{0 \}$ and $\omega \in \Omega$, there exists some neighbourhood of $x$, $J_n(x)$, which $T_\omega^n$ maps diffeomorphically onto some interval of length $\tau$ and such that and such that   $\mathcal{N} ( T_\omega^n | J_n(x)) \le 1$. By choosing a sufficiently small $\tau$, we have $J_n(x) \subset \Lambda_n^\omega (U_0)$ for all $x \in \Lambda_n^\omega (U)$.

Let $\Lambda_n^0$ correspond to the unperturbed orbit. Since $\Lambda_\infty^0 (U_0)$ is of measure zero, then clearly for every $\rho >0$ there exists some $N=N(\rho) \in \mathbb Z_+$ such that $|\Lambda_N^0 (U_0)|< \rho$. Now, if $\eps >0$ is sufficiently small, then for every $\omega \in \Omega$ we have that $\Lambda_N^\omega (U_0)$ is contained in a small neighbourhood of $\Lambda_N^0 (U_0)$. Thus we also have $|\Lambda_N^\omega (U_0)|< \rho$ for all $\omega \in \Omega$.

Let us define $\eta_k = \sup_{\omega \in \Omega} |\Lambda_{kN}^\omega (U)|$. In order for us to obtain the desired result, clearly we must show that $\eta_{k+1} < \frac 1 2 \eta_k$.  Recall that the elements of the set $\A = \big\{ J_N(x) | x \in \Lambda_N^\omega (U) \big\}$ form a cover of $\Lambda_N^\omega (U)$. Thus, using Besicovitch's covering lemma, we can say the following: there exists a constant $c_1$ and there exist subsets $A_1, \dots , A_{c_1}$ of $I$ such that for each $i = 1 , \dots , c_1$ the elements of the set $\big\{ J_N(x) | x \in A_i \big\}$ are all disjoint. Furthermore, we can say 
$$
	\Lambda_N^\omega (U) \subset \bigcup_{i = 1}^{c_N} \Big( \bigcup_{x \in A_i} J_N(x) \Big) .
$$
Note that this implies each point $x \in \Lambda_N^\omega (U) $ can only be covered by at most $c_1$ open neighbourhoods from these $c_1$ sub-collections. 

Clearly $\Lambda_{(k+1)N}^\omega(U) \subset \Lambda_N^\omega(U)$, thus $\A$ is also an open cover of $\Lambda_{(k+1)N}^\omega(U)$, thus we can write
\begin{align}\label{ineq:intersect}
	|\Lambda_{(k+1)N}^\omega(U)| \le \sum_{i=1}^{c_1} \sum_{x \in A_i} |J_N(x) \cap \Lambda_{(k+1)N}^\omega (U)|
\end{align}
Therefore, we need to find an upper bound for $|J_N(x) \cap \Lambda_{(k+1)N}^\omega (U)|$.

Trivially, we have 
$$
	|J_N(x) \cap \Lambda_{(k+1)N}^\omega (U) | = \frac{|T_\omega^N(J_N(x))|}{|T_\omega^N(J_N(x))|} \cdot \frac{|T_\omega^N(J_N(x) \cap \Lambda_{(k+1)N}^\omega (U) )|}{|T_\omega^N(J_N(x) \cap \Lambda_{(k+1)N}^\omega (U) )|} \cdot |J_N(x) \cap \Lambda_{(k+1)N}^\omega (U) |
$$
Recall that we have $1/|T_\omega^N(J_N(x))| = \tau^{-1}$ by definition. Additionally, we have $T_\omega^{N} (\Lambda_{(k+1)N}^\omega(U)) \subset \Lambda_{kN}^{\sigma^N \omega}(U)$ for all $\omega \in \Omega$, thus 
$$|T_\omega^N(J_N(x) \cap \Lambda_{(k+1)N}^\omega (U) )| \le |T_\omega^N(\Lambda_{(k+1)N}^\omega (U) )| \le |\Lambda_{kN}^{\sigma^N \omega} (U) )| \le \eta_k. $$
Now, for the remaining $|T_\omega^N(J_N(x))|/ |T_\omega^N(J_N(x) \cap \Lambda_{(k+1)N}^\omega (U) )|$, we want to use bounded distortion. However, $J_N(x) \cap \Lambda_{(k+1)N}^\omega (U) $ is not necessarily a connected interval, so we cannot directly use the usual mean value theorem combined with bounded distortion technique on it. Instead, one needs to split $J_N(x) \cap \Lambda_{(k+1)N}^\omega (U) $ into its connected components and then apply the technique and recombine the results. Thus, let $\{ A_i \}_i$ be the set of connected components of $J_N(x) \cap \Lambda_{(k+1)N}^\omega (U) $. We can write
\begin{align*}
	&\frac{|T_\omega^N(J_N(x) \cap \Lambda_{(k+1)N}^\omega (U) )|}{|T_\omega^N(J_N(x))|} = \sum_i \frac{|T_\omega^N(A_i)|}{|T_\omega^N(J_N(x))|} \\
	& \ge \sum_i e^{-1}\frac{|A_i|}{|J_N(x)|} 
	 = e^{-1}\frac{|J_N(x) \cap \Lambda_{(k+1)N}^\omega (U) |}{|J_N(x)|}.
\end{align*}	
Taking the reciprocal of this and using the preceding two statements, we have
\begin{align}
	|J_N(x) \cap \Lambda_{(k+1)N}^\omega (U) | &\le e \tau^{-1} \eta_k \frac{|J_N(x)|}{|J_N(x) \cap \Lambda_{(k+1)N}^\omega (U) |} \cdot |J_N(x) \cap \Lambda_{(k+1)N}^\omega (U) | \nonumber \\
	&= e \tau^{-1} \eta_k |J_N(x)| \label{ineq:return_contraction}
\end{align}

Now, define
$$
	\hat \Lambda^\omega(U) = \bigcup_{x \in \Lambda_{N}^\omega (U)} J_N(x).
$$
Recall that $J_n(x) \subset \Lambda_n^\omega (U_0)$ for all $x \in \Lambda_n^\omega (U)$. Thus, $|\hat \Lambda^\omega(U)| \le |\Lambda_N^\omega (U_0)| < \rho$
Plugging in the previous inequality into (\ref{ineq:intersect}), we have
\begin{align*}
	|\Lambda_{(k+1)N}^\omega(U)| &\le \sum_{i=1}^{c_1} \sum_{x \in A_i} e \tau^{-1} \eta_k |J_N(x)|\\
	& = e \tau^{-1} \eta_k c_1 |\hat \Lambda^\omega(U)| 
	 < e \tau^{-1} \eta_k c_1 \rho,
\end{align*}
where in the second line we use the fact that any point $x \in \Lambda_N^\omega (U)$ is covered by at most $c_1$ open neighbourhoods $J_N(y)$ with $y \in A_1 \cup \dots \cup A_{c_1}$.

To finish the proof, all we need to do is take $\rho$ to be small enough such that $e \tau^{-1} c_1 \rho \le 1/2$. Since the above inequality holds for all $\omega$, it also holds for $\eta_{k+1}$, thus giving us our desired result.
\end{proof}

\end{lemma}

Now, let us prove certain results for the deterministic dynamics of the initial unperturbed map $T_0$. Specifically, we wish to prove the following two propositions:

\begin{proposition}\label{prop:Shen_4.1}
	For any fixed constants $L>0$, $\theta \in (0,1)$ and $\zeta > 0$, and for sufficiently small $\delta > 0$, there exists a positive integer $M=M(\delta)$, satisfying $M(\delta) \to \infty $ as $\delta \to 0$, such that
	\begin{align}
		&A( \underline{0}, 0, M(\delta)) \le \theta/\delta, \label{ineq:Shen_(4.1)}  \\
		&T_0(0), \dots , T_0^{M(\delta) - 1}(0) \notin \tilde B(L \delta), \label{ineq:Shen_(4.2)}
	\end{align}
	and
	\begin{align}
		DT_0^{M(\delta) + 1}(T_0(0)) \ge \Big( \frac{ \max \big\{ |T_0^{M+1}(0) |, \delta  \big\}  }{\delta}  \Big)^{1 - \zeta}.  \label{ineq:Shen_(4.3)}
	\end{align}
\end{proposition}

\begin{proposition}\label{prop:Shen_4.2}
	There exists a constant $\kappa_0 >0$ such that, for all sufficiently small $\delta>0$, if $x \in I \setminus \{ 0 \}$ and $x , T_0 (x), \dots, T_0^{n-1}(x) \notin \tilde B(\delta)$ and $T_0^n(x) \in \tilde B(\delta)$ for some $n \in \mathbb N$, then we have
	\begin{align}
		DT_0^n(x) \ge \frac{\kappa_0}{\D(\delta)} \Big( \frac{\delta}{ \max \{ |x - T_0(0)|, \delta \}}  \Big)^{1 - 1/s}.
	\end{align}
\end{proposition}

We have the following backwards contraction lemma, which is an adaption of Lemma 1 in \cite{BLSvS}:
\begin{lemma}\label{Shen:lemma_4.3}
	Assume that for sufficiently small $\delta >0$, we have a constant $\hat v(\delta)>1$, satisfying $\hat v(\delta) \to \infty$ and $\hat v(\delta) \delta \to 0$ as $\delta \to 0$. Then for every $n \in \mathbb Z_+$ and every component $W$ of $(T_0^n)^{-1}(\tilde B(\hat v(\delta) \delta))$, if $\textup{dist}(W, T_0(0) ) \le \delta$, then this implies $|W| < \delta$.

\end{lemma}

We will prove this proposition using the following sublemmas:

\begin{sublemma}\label{prop:BLSvS_Lemma_1}
	Assume that for sufficiently small $\delta >0$, we have a constant $\hat v(\delta)>1$, satisfying $\hat v(\delta) \to \infty$ and $\hat v(\delta) \delta \to 0$ as $\delta \to 0$. Then for all $n \in \mathbb Z_+$ such that $T_0^n(0) \in \tilde B (\delta)$, if $W$ is the component of $(T_0^n)^{-1}(\tilde B (\hat v(\delta) \delta))$ containing $0$, then $W \subset \tilde B( \delta)$.

\begin{proof}
	First, to show that there exists such a $\hat v(\delta)$ to guarantee the statement of this sublemma is not empty, just consider the function $\hat v(\delta) = \delta^{-1/2}$ for example. Now, let us consider the two sets of intervals $\big\{ G_j \big\}_{j=0}^n$ and $\big\{ H_j \big\}_{j=0}^n$ satisfying $G_n = \tilde B(\hat v(\delta)\delta ) \supset H_n = \tilde B(\delta)$,  $G_0 \supset H_0 = W$, and for each $j =  0, 1, \dots, n-1$, $G_j$ is a connected component of $T_0^{-1}(G_{j+1})$ and $H_j$ is a connected component of $T_0^{-1}(H_{j+1})$. 

Let $n_1 < n$ denote the largest time such that $G_{n_1}$ contains the point 0, and let $H_{n_1 + 1}'$ denote the convex hull of $H_{n_1 + 1} \cup \big\{ T_0^{n_1}(0)  \big\}$, i.e. since $T_0^{n_1+1}(0) \in G_{n_1 + 1}$, either $H_{n_1 + 1}' = H_{n_1 + 1}$ or $H_{n_1 + 1}'$ is the smallest interval containing $H_{n_1 + 1}$ and $\big\{ T_0^{n_1}(0)  \big\}$. Clearly, this implies $H_{n_1 + 1}' \subset G_{n_1 + 1}$.

Now, we want to show that for sufficiently small $\delta >0$, we have $H_{n_1} \subset \tilde B(\hat v (\delta) \delta)$. To this end, first we notice that $T_0^{n - n_1 - 1}: G_{n_1 + 1} \to G_n$ is a diffeomorphism for sufficiently small $|G_n|$, and thus by applying the one-sided Koebe principle to the connected components of $G_{n_1 + 1 } \setminus \big\{ T_0(0) \big\} $ that also intersect $H_{n_1+1}'$, we know that there exists a uniform constant $C >0$ such that,for every $x \in H_{n_1 + 1}'$, we have
\begin{align}
	DT_0^{n - n_1 -1}(x) \ge C |DT_0^{n - n_1 -1}(T_0(x))|. \label{ineq:Koebe}
\end{align}
Furthermore, we by property H1. (a), we know that
$$
	DT_0^{n - n_1} (T_0(x)) \ge \lambda.
$$
Furthermore, by using the order of singularity condition, we have that
\begin{align*}
	DT_0(T_0^{n - n_1}(0)) \le |(T_0^{n - n_1}(0)|^{s -1}  \le K_2 \frac{|T_0(G_n)|}{|G_n|}
\end{align*}
Combining these all together, we have
$$
	DT_0^{n - n_1 - 1}(T(0)) = \frac{DT_0^{n - n_1} (T_0(x))}{DT_0(T_0^{n - n_1}(0))} \ge \lambda K_2^{-1} \frac{|G_n|}{|T_0(G_n)|}.
$$
Combining this with inequality (\ref{ineq:Koebe}), we obtain
$$
	\frac{|G_n|}{|H_{n_1 + 1}'|} \ge C K_1^{-1} \lambda \frac{|G_n|}{|T_0(G_n)|},
$$
and thus we have
$$
	|H_{n_1 + 1}'| \le   \delta.
$$

When $n_1=0$, then this completes the proof. If this is not the case, we just repeat the above process until we obtain the desired result until we reach $0$.
\end{proof}
	
\end{sublemma}

\begin{proof}[Proof of Lemma \ref{Shen:lemma_4.3}]
Notice that, since we are assuming $\textup{dist}(W, T_0(0) ) \le \delta$, for sufficiently small $\delta$ and by the order of singularity condition, $W$ will be close to the boundary of $I$ and thus will only be experiencing expansion. Therefore, if we have $|T_0(W)| < \delta$, then clearly also $|W| < \delta.$ We therefore focus on the image of $W$.

Let $x \in W$, and let $W_k$ be the component of $(T_0^{n - k})^{-1}(\tilde B (\hat v(\delta)\delta))$ containing $T_0^k(x)$. To complete the proof, we simply need to show that $|W_1| < \delta$. To do this, we use proof by induction. For the case that $n = 1$, then plugging this and $k=1$ into the above preimage we simply have $|W_1| =|\tilde B (\hat v(\delta)\delta)| < \delta $. 

Now, assume that we have proved up to iterate $n_0$ that the desired result holds for all $n < n_0$. It is then a question of proving the result also holds for $n = n_0$. Let us again consider a set of intervals $\big\{ G_j  \big\}_{j = 0}^{n}$ satisfying $G_n = \tilde B(\hat v(\delta) \delta)$, $x \in G_0$, and each $G_j$ is a connected component of $T_0^{-1}(G_{j+1})$. 

For this set, there are two possible cases, the first being that there exists some $0 \le n_1 < n$ such that $0 \in G_{n_1}$. Then from Sublemma \ref{prop:BLSvS_Lemma_1}, this therefore implies $G_{n_1} \subset \tilde B(\delta)$. If $n_1 = 0$, then $W_0 \subset G_0 \subset \tilde B(\delta)$, thus proving the desired result. If $n_1 >0$, then we simply repeat this process as we did in the proof of the previous lemma, giving us the desired result.

Alternatively, the second possibility is that for any $0 \le k < n$, we have $ 0 \notin G_k$. But this immediately implies that $T_0^{n - 1}: G_1 \to G_n$ is a diffeomorphism. Again, using the Koebe principle, we have that $W_1 \subset G_1$, and thus since $0 \notin G_0$, $T_0(0) \notin G_1$ and $T_0(x) \in B(T_0(0), \delta)$, we therefore have that $|W_1| < \delta$.

\end{proof}

\begin{lemma}\label{lemma:Shen_4.4}
For all sufficiently small $\delta >0$, there exists a constant $v(\delta) > 1$, satisfying $v(\delta) \to \infty$ and $v(\delta)\delta \to 0$ as $\delta \to 0$, such that for all $x \in (- \delta, \delta)$ satisfying $x , T_0(x), \dots , T_0^{n-1}(x) \notin \tilde B(\delta)$ and $T_0^n(x) \in \tilde B(2 \delta)$ for some $n \in \mathbb Z_+$, then we have
	\begin{align}
		DT_0^n(x) \ge \frac{v(\delta)}{\mathcal D(\delta)}.
	\end{align}

\begin{proof}
	Let us take $\delta >0$ small enough so that $\hat v(\delta) > 4$. Since we have already assumed that $|x| \le \delta$, then there exists a a component $W$ of $(T_0^{n})^{-1}(\tilde B(\hat v (\delta) \delta))$ such that $x \in W$, and therefore by the above lemma, we know that $|W|<\delta$. Furthermore, we know that $T_0^n(x)$ is close to the centre of $\tilde B(\hat v (\delta) \delta)$, we therefore know that by the Koebe principle that there exists some constant $C>0$ such that
\begin{align*}
	DT_0^n(x) &\ge C \frac{|\tilde B(\hat v (\delta) \delta)|}{\delta}
	 \ge  C \frac{(\hat v (\delta) \delta)^{1/s}}{\delta} \\
	& \ge C \frac{\hat v (\delta)^{1/s}|\tilde B( \delta)|}{|B_\delta(0)|} 
 	 = C \frac{\hat v(\delta)^{1/s}}{\mathcal D (\delta)}.
\end{align*}
Setting $v(\delta) =C \hat v(\delta)^{1/s}$ completes the proof.
\end{proof}
\end{lemma}

\begin{lemma}\label{lemma:Shen_4.5}
	For all sufficiently small $\delta > 0$, there is a constant $ \theta(\delta) > 0$, satisfying $\theta(\delta) \to 0$ as $\delta \to 0$, such that for all $n \in \mathbb Z_+$ and $x \in I$ with $x , T_0(x), \dots , T_0^{n-1}(x) \notin \tilde B(\delta)$ and $T_0^n(x) \in \tilde B(2 \delta)$, we have
	$$
		A( \underline 0 , x, n) \le \theta(\delta) \frac{DT_0^n(x)}{|\tilde B(\delta)|}.
	$$ 

\begin{proof}
	First, let us start with some fixed $\delta > 0$. We already know from  [\cite{AS}, Lemma 2.1] that for each $x \in I$ such $x , T_0(x), \dots , T_0^{n-1}(x) \notin \tilde B(\delta)$ and $T_0^n(x) \in \tilde B(2 \delta)$, we have $DT_0^n(x) \ge \lambda^n$. Thus, we know there exists some minimal constant $\theta(\delta)$ such that the above inequality holds. To show the same inequality holds for smaller values of $\delta$, it remains to show that there exist constants $\tau(\delta)$, satisfying $\tau(\delta) \to 0$ as $\delta \to 0$, and $\kappa<1$, where
$$
	\kappa^2 = \sup_{\text{small } \delta > 0} \frac{\tilde B (\delta /2)}{\tilde B(\delta)},
$$
such that we have
	\begin{align*}
		\theta(\delta/2) \le \kappa (\theta(\delta) + \tau(\delta)).
	\end{align*}

Now, consider an $x \in I$ whose orbit up to time $n$ satisfies  $x , T_0(x), \dots , T_0^{n-1}(x) \notin \tilde B(\delta /2)$ and $T_0^n(x) \in \tilde B( \delta)$. Let $0 \le \nu_1 < \nu_2 < \dots < \nu_m = n$ denote all of the iterates where $T_0^{\nu_i}(x) \in \tilde B(\delta) $, and let $d_i = |T_0^{\nu_i}(x)|$. Clearly, for every $i = 1, 2, \dots m-1,$ we have $d_i \ge \delta/2$, and furthermore by Lemma \ref{shen:lemma_2.5}, we have $DT_0(T_0^{\nu_i}(x)) \ge \mathcal D(d_i)$. Thus, by Lemma \ref{lemma:Shen_4.4}, we have
	\begin{align*}
		DT_0^{\nu_{i+1} - \nu_i}(T_0^{\eta_i}(x)) \ge \frac{v (\delta) \mathcal D(d_i)}{\mathcal D(\delta)} = \frac{v(\delta) d_i}{\delta} \frac{|\tilde B(\delta)|}{|\tilde B(d_i)|} \ge \frac{v(\delta) }{2} \frac{|\tilde B(\delta)|}{|\tilde B(d_i)|}.
	\end{align*}
Using the order of singularity condition in the same way as we did in the previous lemma, the above implies that
	\begin{align*}
		\frac{|DT_0^{\nu_i}(x)|}{|\tilde B(\delta)|} \le \frac{|DT_0^{\nu_i}(x)|}{|\tilde B(d_i)|} \le \frac{2}{v(\delta)}\frac{DT_0^{\nu_{i+1}}(x)}{\tilde B(\delta)},
	\end{align*}
and by iterating this we obtain
	\begin{align}
		\frac{DT_0^{\nu_i}(x)}{|T_0^{\eta_i}|} \le C \frac{DT_0^{\nu_i}(x)}{|\tilde B (d_i)|} \le C \Big(\frac{2}{v(\delta)}  \Big)^{m-i}\frac{DT_0^{n}(x)}{|\tilde B(\delta)|}. \label{ineq:exp_dist}
	\end{align}

Now, since for each $i = 1, \dots , m-1$ we have $T_0^{\nu_i + 1}(x), \dots , T_0^{\nu_{i + 1}-1}(x) \notin \tilde B(\delta)$ and $T_0^{\nu_{i + 1}}(x) \in \tilde B(2 \delta)$, we can use the estimate at the beginning of the proof to obtain
$$
	\sum_{j = \nu_i + 1}^{\nu_{i+1}-1}\frac{DT_0^j(x)}{|T_0^j(x)|} \le \theta(\delta) \frac{DT_0^{\nu_{i+1}}(x)}{|\tilde B(\delta)|}.
$$
Furthermore, in the case that $\nu_1 \neq 0$, the same applies:
$$
	\sum_{j = 0}^{\nu_{1}-1}\frac{DT_0^j(x)}{|T_0^j(x)|} \le \theta(\delta) \frac{DT_0^{\nu_{1}}(x)}{|\tilde B(\delta)|}.
$$
Thus, taking the sum of these, we clearly have
\begin{align*}
	A( \underline{0}, x, n) &= \sum_{j = 0}^{n-1} \frac{DT_\omega^j(x)}{|T_\omega^j(x)|} \\
	& \le \sum_{i = 0}^{m-1} \sum_{j = \nu_i + 1}^{\nu_{i+1}-1} \frac{DT_\omega^j(x)}{|T_\omega^j(x)|} 
	\le  \sum_{i = 0}^{m-1} \theta(\delta) \frac{DT_0^{\nu_{i+1}}(x)}{|\tilde B(\delta)|}  \\
	&\le (1 + \theta(\delta))   \sum_{i = 1}^{m-1} \frac{DT_0^{\nu_{i}}(x)}{|\tilde B(\delta)|} + \theta(\delta) \frac{DT_0^{\nu_{m}}(x)}{|\tilde B(\delta)|},
\end{align*}
and thus combining this with inequality (\ref{ineq:exp_dist}), we obtain
$$
	A( \underline{0}, x, n) \le \big( \theta(\delta) + \tau(\delta) (1 + \theta(\delta) )  \big)\theta(\delta) \frac{DT_0^{\nu_{m}}(x)}{|\tilde B(\delta)|},
$$
where we have $\tau = 2 C / (v(\delta)(1 - 2/v(\delta)) )$. Now, notice that for sufficiently small $\delta > 0$, we have $1 + \tau(\delta) < \kappa^{-1}$, and thus, using the fact that $|\tilde B(\delta/2)| \le \kappa^2 |\tilde B(\delta )|$ by definition, we therefore have
$$
	\kappa (\theta(\delta) + \tau(\delta)) \frac{DT_0^{\nu_{m}}(x)}{|\tilde B(\delta)|},
$$
which completes the proof.
\end{proof}

\end{lemma}

Now, using the above lemmas we complete the proof for Proposition \ref{prop:Shen_4.1}:

\begin{proof}[Proof of Proposition \ref{prop:Shen_4.1}]
For our fixed $L>1$, notice that for every sufficiently small $\delta >0$ there will exist some $L(\delta) > L $ satisfying $\lim_{\delta \to 0 } L(\delta) = \infty$, $\lim_{\delta \to 0 } L(\delta)\delta = 0$ and $\lim_{\delta \to 0 } v(L(\delta) \delta)/ L(\delta) = \infty$. Indeed, sticking with the original example of $\hat v(\delta) = \delta^{-1/2}$, and thereby $v(\delta) = C \delta^{-1/2s} $, setting $L(\delta) =  \delta^{-1/100}$ would be an example of such a function:
\begin{align*}
\frac{v(L(\delta) \delta)}{ L(\delta)} & = C \frac{(L(\delta) \delta)^{-1/2s}}{L(\delta)} \\
	& = C L(\delta)^{-(1+1/2s)} \delta^{-1/2s} 
	 = C \delta^{(1+1/2s)/100} \delta^{-1/2s}
\end{align*}
Furthermore, we can choose some $L(\delta)$ such that, if $|T_0^n(0) | \le L(\delta) \delta$, then
\begin{align}
	 DT_0^{n+1}(T_0(0)) \ge L(\delta)^{1 - \zeta}. \label{ineq:Shen_(4.10)}
\end{align}
Indeed, using properties (R1) and (R2) of Rovella maps, we know that if $|T_0^n(0) | \le L(\delta) \delta$, then $n$ must be sufficiently large that so that $e^{-\alpha n} \le  L(\delta) \delta$, and thus $DT_0^{n+1}(T_0(0)) \ge \lambda^{n+1} > L(\delta)^{1 - \zeta}$. As long as $L(\delta)^{1-\zeta}$ grows sufficiently slowly, then this always holds for all sufficiently small $\delta$.

For sufficiently small $\delta >0$, let us denote $N = N(\delta)$ be the largest time such that we have $A( \underline{0}, x, N) \le \theta / \delta$. We want to prove the following: for every $j =  1, \dots, N-1$, we have $T_0^j(0) \notin \tilde B(L(\delta) \delta))$. Indeed, this can be shown by a simple proof by contradiction. If this is not the case, then there exists some smallest $t \in \{ 1 , \dots, N-1\} $ satisfying $T_0^t \in \tilde B(L(\delta)\delta)$. Then by Lemma \ref{lemma:Shen_4.4}, we have
$$
	DT_0^t(0) \ge  \frac{v(L(\delta) \delta)}{\mathcal D(L(\delta) \delta)}
$$
and thus we have
$$
	A( \underline{0}, x, N) \ge \frac{DT_0^t(0)}{|T_0^t(0)|} \ge \frac{v(L(\delta) \delta)}{|\tilde B(L(\delta) \delta)| \mathcal D(L(\delta) \delta)} = \frac{v(L(\delta) \delta)}{v(\delta)\delta} > \frac{\theta}{\delta},
$$
which clearly contradicts our initial assumption.

Now, suppose $T_0^N(0) \in \tilde B(L(\delta) \delta)$. In this case, we simply set $M(\delta) = N$. Furthermore, since we already know $\delta ' \le L \delta $, then inequality (\ref{ineq:Shen_(4.3)}) immediately follows from \ref{ineq:Shen_(4.10)}. Alternatively, assume the opposite, i.e. $T_0^N(0) \notin \tilde B(L(\delta) \delta)$. Then our procedure for defining $M$ becomes more involved.

Let us set $V_k = \tilde B (2^{-k} \delta)$, and let us denote $k_1$ as the first iterate such that $T_0^j(0) \notin V_{k_1}$ for every $0 \le j \le N$. Clearly, it follows that $\delta/2^{k-1} > L(\delta) \delta$. Thus, for sufficiently small $\delta >0$ we have
$$
	\sum_{k = 0}^{k_1 -1} \Big( \frac{\delta}{\delta/ 2^k} \Big)^{\zeta/2} \le \Big( \frac{\delta}{ \delta/2^{k - 1}} \Big)^{\zeta/2} \frac{1}{1 - 2^{- \zeta/2}} \le \frac{L(\delta)^{- \zeta/2}}{1 - 2^{- \zeta/2}} \le \frac{\theta}{ 2}.
$$

Now, let us set $\nu^*_{k_1} = -1$, and we set for each $0 \le k < k_1$
$$
	\nu^*_k = \max \big\{ 0 \le j \le N \Hquad | \Hquad T_0^j(0) \in V_k \big\}.
$$
Then there are two possibilities: either
\begin{align}
	\sum_{i = \nu^*_0 + 1}^N \frac{DT_0^i(T(0))}{|T_0^i(T(0))|} \ge \frac{\theta}{2 \delta}, \label{case1}
\end{align}
or alternatively there exists some $0 \le k < k_1$ such that
\begin{align}
\sum_{i = \nu^*_k + 1 }^{s_{k+1}}  \frac{DT_0^i(T(0))}{|T_0^i(T(0))|} \ge \frac 1 \delta  \Big( \frac{\delta}{\delta/ 2^k} \Big)^{\zeta/2}.
\end{align}
If the former holds, then we again set $M = N$. As we already have $T_0^n(c) \notin \tilde B(\delta)$, we therefore know that for every $\nu^*_0 < j \le N$ there will exist some constant $C_1>0$ such that
$$
	DT_0^{N+1}(T_0(0)) \ge C_1 \sum_{n = \nu^*_0 + 1}^N \frac{DT_0^n(T(0))}{|T_0^n(0)|} \ge \frac{C_1 \theta}{2 \delta} \ge \Big( \frac{\delta_*}{\delta}  \Big)^{1 - \zeta},
$$
assuming that $\delta > 0$ is sufficiently small. 

Suppose that \ref{case1} is not the case. Then we instead set $M = \nu^*_k$. By definition, $T_0^{\nu^*_k} \in \tilde B(\delta_k)$. By Lemma \ref{lemma:Shen_4.5}, there exists some constant $C_2>0$ such that
\begin{align*}
	\frac{DT_0^{s_k}(T_0(0))}{|\tilde B(\delta_k)|} \ge C_2 \sum_{n = \nu^*_k + 1}^{\nu^*_{k+1}} \frac{DT_0^n(T(0))}{|T_0^n(0)|} \ge \frac{C_2}{\delta}\Big( \frac{\delta}{\delta_k}  \Big)^{ \zeta / 2}.
\end{align*}
Now, using the fact that $\nu^*_k > \nu^*_{k+1}$ and the fact that $T_0^{\nu^*_k}(0) \in \tilde B(\delta_k/2)$, we therefore have $DT_0(T_0^{\nu^*_k}(0)) \cdot |\tilde B(\delta_k)| \ge C_3 \delta_k$, with $C_3>0$ being some constant. Combining this altogether, we have
\begin{align*}
	DT_0^{\nu^*_k + 1}(T_0(0)) = DT_0^{\nu^*_k}(T(0)) \cdot DT_0(T_0^{\nu^*_k}(T(0))) \ge C_2 C_3 \frac{\delta_k}{\delta}\Big( \frac{\delta}{\delta_k}  \Big)^{ \zeta / 2}.
\end{align*}
Using the fact that $\delta_k/\delta \ge 2 \delta_{k_1}/\delta \ge L(\delta)$, then if $\delta$ is sufficiently small, then we obtain our desired result.
\end{proof}

\begin{proof}[Proof of Proposition \ref{prop:Shen_4.2}]
	Let us take some constant $\delta_0>0$ sufficiently small such that for all $\delta \in (0, \delta_0)$ that \ref{lemma:Shen_4.4} can applied and that $v(\delta) > 2$. Because we already know that $T_0^n(x) \in \tilde B(2 \delta)$, we therefore know that there exists a sequences of iterates $n_1 < n_2 < \dots < n_m = n$ such that $n_1$ is the first iterate when $T_0^{n_1}(x) \in \tilde B(\delta_0)$ and such that, for every $i = 1, 2, \dots, m - 1$, the iterate $n_{i + 1}>n_i$ is the minimal integer such that
$$
	|T_0^{n_{i+1}+1}(x) - T_0(0)| \le \min \big\{ \delta_0, 2 |T_0^{n_{i}+1}(x) - T_0(0)|  \big\}.
$$ 

Additionally, for every $i = 1, 2, \dots, m - 1$, let $\rho_i = |T_0^{n_{i}+1}(x) - T_0(0)|$, which means $\rho_i \in [\delta, \delta_0)$ for every $i = 1, 2, \dots, m - 1$.

Now, for every $i = 1, 2, \dots, m - 1$, consider the segment of the orbit $T_0^{n_i+1}(x), \dots$ $, T_0^{n_{i+1}}(x)$. If we apply Lemma \ref{lemma:Shen_4.4} with respect to $\rho_i$ instead of $\delta$, then we have
\begin{align*}
	DT_0^{n_{i+1} - n_i}(T_0^{n_i}(x)) &= DT_0(T_0^{n_i}(x)) \cdot DT_0^{n_{i+1} - n_i - 1}(T_0^{n_i + 1}(x))\\
	& \ge \D(\rho_i) \frac{v(\rho_i)}{\D(\rho_i)}
	 \ge 2 \frac{\D(\rho_i)}{\D(\rho_i)}.
\end{align*}
Additionally, if we set $\rho = \delta''$, then this gives us $DT_0^{n_1}(x) \ge 2 \kappa_0 / \D(\rho_0)$. Indeed, if we have that $\delta'' < \delta_0$, then we can just apply Lemma \ref{lemma:Shen_4.4}. On the other hand, if $\delta'' \ge \delta_0$, then we can just use the non-uniform expansion. Therefore, combining this all together we obtain
\begin{align*}
	DT_0^n(x) &\ge \frac{2\kappa_0}{\D(\rho_0)} \prod_{i = 1}^{m - 1} \frac{\D(\rho_i)}{\D(\rho_i)} 
	=  \frac{2\kappa_0}{\D(\rho_{m-1})} \prod_{i = 1}^{m - 1} \frac{\D(\rho_i)}{\D(\rho_{i-1})} \\
	&\ge \frac{2\kappa_0}{\D(\rho_{m-1})} \frac{\D(\rho_i)}{\D(\rho_{i-1})} 
	\ge \frac{2\kappa_0}{\D(\rho_{m-1})} \Big( \frac{\rho_i}{\rho_{i-1}} \Big)^{1 - 1/s}   
	 \ge \frac{\kappa_0}{\D(\delta)} \Big( \frac{\delta}{\delta''}  \Big)^{1 - 1/s}.
\end{align*}
\end{proof}

Now, let us move onto proving certain results for the random dynamics. In particular, we wish to prove the following:

\begin{proposition}\label{Shen:Proposition_5.1}
	For sufficiently small $\delta> 0$ there exists some constant $\hat \Lambda(\delta) >0$ satisfying $\lim_{\delta \to 0} \hat \Lambda(\delta) = \infty$ such that, for all sufficiently small $\eps >0$, all $\omega \in \Omega_\eps$, all $x \in I$ satisying $|x| \le 4 \delta$, and all $n \in \mathbb Z_+$, if $x, T_\omega(x) , \dots , T_\omega^{n-1} \notin \tilde B(\delta)$ and $T_\omega^n(x) \in \tilde B(2 \delta) $, we have
	\begin{align}
		DT_\omega^n(x) \ge \frac{\hat \Lambda(\delta)}{\mathcal D (\delta)}.
	\end{align} 
\end{proposition}

Now, let us fix some constant $\eta_* \in (0,1)$, and let us set
$$
	W_0= \sum_{n = 0}^\infty \frac{1}{DT_0^n(T(0))}.
$$
Furthermore, let us choose some small $\theta > 0$ satisfying
\begin{align}
4 \theta W_0 \le \theta_1 \text{ and } 16e \theta W_0 \sup DT_0, \label{eta_star}
\end{align}
recalling $\theta_1$ from Lemma \ref{shen:lemma_2.5}. Now, it will be useful to introduce the notion of a \textit{preferred binding period}.

\begin{definition}[Preferred Binding Period]
Let us denote
$$
	W_0= \sum_{n = 0}^\infty \frac{1}{DT_0^n(T(0))},
$$
and let us also fix some small constant $\theta > 0$ satisfying $4 \theta W_0 \le \theta_1$, where $\theta_1$ is from Lemma \ref{shen:lemma_2.5}. Furthermore, let us fix constants $L > 2^{s+1}$ and $\zeta \in ( 0, s^{-1})$. For sufficiently small $\delta >0$, we say $M(\delta)$ is the preferred binding period if $M(\delta)$ is the smallest positive integer such that the results of Proposition \ref{prop:Shen_4.1} hold for our constants $\theta$, $L$ and $\zeta$. Associated to this, we also denote
$$
	\Lambda_0(\delta) := DT_0^{M(\delta) + 1}(0),
$$
where we clearly have $\lim_{\delta \to 0 }\Lambda_0(\delta) = \infty$.
\end{definition}

\begin{proposition}\label{Shen:prop_5.2}
	There exist constants $\zeta_1, \zeta_2 >0$ such that, for sufficiently small $\delta >0$, for all sufficiently small $\eps >0$, all $\omega \in \Omega_\eps$, and all $x \in I$ satisfying $|x| \le 4 \delta$, we have
	\begin{align}
		T_\omega^j(x) &\notin \tilde B(2 \delta) \quad \text{for all } 0 \le j < M, \label{Shen:eq_5.3} \\
		DT_\omega^M(y) &\ge \frac{\Lambda_0(\delta)^{\zeta_1}}{\D(\delta)} \label{Shen:eq_5.4}.
	\end{align}
Furthermore, if $T_\omega^M(x) \notin \tilde B(\delta)$, then we have
\begin{align}
	DT_\omega^{M+1}(y) \ge \Lambda_0(\delta)^{\zeta_1} \Big( \frac{|T_\omega^{M}(y) - T_0(0)|}{\delta} \Big)^{1 - 1/s} \label{Shen:eq_5.5} .
\end{align}

\begin{proof}
	Let $\delta>0$ be fixed and sufficiently small, and set $M=M(\delta)$. Using Lemma \ref{shen:lemma_2.5}, we can see that $M$ is a $eW_0$- binding period for $(0, 4 \delta)$. Furthermore, by inequalities \ref{ineq:Shen_(4.2)} and \ref{def:bind1}, we can see that \ref{Shen:eq_5.3} holds. 

Now, for inequality \ref{Shen:eq_5.4}, we notice that by property \ref{def:bind3} in the definition of a binding time, we have
\begin{align*}
	|T_\omega^M(x) - T^M(0)| &\le 4 e \delta W_0 DT^M(0) \\
	& \le 4 e \delta W_0 C_0 DT^{M-1}(T(0)) \\
	& \le 4 e \delta W_0 C_0 A(\underbar 0, T(0), M)\\
	& \le  4 e \delta W_0 C_0 \theta 
	 \le \eta_* / 4,
\end{align*}
where for the last inequality we have used the second inequality in \ref{eta_star}. This thereby implies that there exists some $C_1 > 1$ such that 
\begin{align}
	DT(T^M(0)) \le C_1 \D(\delta_1), \label{ineq:shen_5.6}
\end{align}
where $\delta_1 = \max \{  |T^M(0)|, \delta \}$. Now, let us set $\zeta_1 = (1/s - \zeta)/ (2 - 2 \zeta)$. By inequality \ref{ineq:Shen_(4.3)} and the definition of $\Lambda_0$, we have 
\begin{align}
	DT^{M+1}(T(0)) \ge \Lambda_0(\delta)^{2 \zeta_1} \Big( \frac{\delta_1}{\delta} \Big)^{1 - 1/s} \label{ineq:shen_5.7}
\end{align}
Combining the two above inequalities with property \ref{def:bind2} of the definition of a binding time, we have
\begin{align*}
	DT^M(x) &\ge e^{-1} \frac{DT^{M+1}(T(0))}{DT(T^{M+1}(0))} \\
	& \ge e^{-1} \frac{\Lambda_0(\delta)^{2 \zeta_1}}{C_1 \D(\delta_1)}  \Big( \frac{\delta_1}{\delta} \Big)^{1 - \frac 1 s} 
	\ge \frac{C_2 \Lambda_0(\delta)^{2 \zeta_1}}{ \D(\delta_1)},
\end{align*}
where $C_2 >0$ is some suitably chosen constant and we use the fact that $\delta^* \ge \delta$. Finally, if $\delta >0$ is sufficiently small, then $C_2 \Lambda_0(\delta)^{ \zeta_1} > 1$, which thereby proves inequality \ref{Shen:eq_5.4}.

Now, for the last property \ref{Shen:eq_5.5}, recall that we are assuming $T_\omega^M(x) \notin \tilde B(\delta)$, which implies $|T_\omega^M(x)| \ge \delta$. Setting $\delta_2 = |T_\omega^M(x)|$,  $\zeta_2 = \zeta_1 / s$ and using inequality \ref{delta_star}, we have
\begin{align}
	DT_{\sigma^M \omega}(T_\omega^M(0)) \ge C_3 \D(\delta_2) \label{Shen:eq_5.8}
\end{align}
for some suitably chosen constant $C_3 >0$. There are two possible cases, either 1) $\delta_2 > \Lambda_0(\delta)^{2 \zeta_1} \delta_1$, or 2) $ \delta_2 \le \Lambda_0(\delta)^{2 \zeta_1} \delta_1$.\\ \\
\textbf{Case 1:}\\ \\
If $\delta_2 > \Lambda_0(\delta)^{2 \zeta_1} \delta_1 \ge \Lambda(\delta)^{2 \zeta_1} \delta$, then for sufficiently small $\delta >0$ there exists a constant $C_4 > 0$ such that $\eta_M \ge C_4 |\tilde B(\delta _2)|$, and thus combining this with inequality \ref{Shen:eq_5.8}, we have
\begin{align*}
	\eta_M DT_{\sigma^M \omega}(T_\omega^M(0)) \ge C_3 \D(\delta_2) C_4 |\tilde B(\delta_2)| \ge C_3 C_4 \delta_2.
\end{align*}
Furthermore, by combining this with properties \eqref{Shen:eq_5.3} and \eqref{Shen:eq_5.4} of the definition of a binding time, we thereby have \eqref{Shen:eq_5.5}.
\begin{align*}
	DT_\omega^{M+1}(x) &\ge e^{-1} DT_0^M(T_0(0)) DT_{\sigma^M \omega}(T_\omega^M(T_0(0)))  \\
	&\ge \frac{\eta_M DT_{\sigma^M \omega}(T_\omega^M(T_0(0)))  }{4 e^2 W_0 \delta} \\
	& \ge C_5 \frac{\delta_2}{\delta} \\
	& \ge C_5    \Lambda_0(\delta)^{2 \zeta_1} \Big( \frac{\delta_1}{\delta} \Big)^{1 - 1/s},
\end{align*}
where $C_5 >0$ is a suitably chosen constant. Thus, if $\delta$ is sufficiently small, then \eqref{Shen:eq_5.5} holds.\\ \\
\textbf{Case 2:}\\ \\
Now suppose that $ \delta_2 \le \Lambda_0(\delta)^{2 \zeta_1} \delta_1$. If we combine this with property 1 of the definition of binding time, \ref{ineq:shen_5.6}, \eqref{ineq:shen_5.7} and \eqref{Shen:eq_5.8}, we have
\begin{align*}
	DT_\omega^{M+1}(x) &\ge e^{-1} DT_0^{M+1}(T_0(0)) \frac{DT_{\sigma^M \omega}(T_\omega^M(0))}{DT_0(T_0^M(0))} \\
	& \ge C_6 \Lambda_0(\delta)^{2 \zeta_1 } \Big( \frac{\delta_1}{\delta} \Big)^{1 - 1/s} \Big( \frac{\delta_2}{\delta_1} \Big)^{1 - 1/s}
	 \ge C_6 \Lambda_0(\delta)^{2 \zeta_1 } \Big( \frac{\delta_2}{\delta} \Big)^{1 - 1/s},
\end{align*}
where $C_5 > 0$ is a suitably chosen constant. Thus, \ref{Shen:eq_5.5} follows if $\delta$ is sufficiently small.
\end{proof}

\end{proposition}

\begin{lemma}\label{Shen:Lemma_5.3}
Let $\delta_3 = \max\big\{ |x - T_0(0)|, \delta  \big\}$. There exists a constant $\kappa >0$, and for sufficiently small $\delta >0$ there exists a constant $\hat \eta = \hat \eta (\delta) > 0$ such that, for sufficiently small $\eps>0$ and all $\omega \in \Omega_\eps$, and for all $x \in I\setminus \{0 \}$ and $n \in \NN_0$ satisfying $x, T_\omega(x), \dots , T_\omega^{n-1} \notin \tilde B(\delta)$ and $T_\omega^n(x) \in \tilde B(\delta)$, we have
	\begin{align}
		DT_\omega^n(x) \ge \frac{\kappa}{\D (\delta)} \Big( \frac{\delta}{\delta_3} \Big)^{1 -1/s} e^{\hat \eta n}
	\end{align}

\begin{proof}
Using Proposition \ref{Shen:Prop_2.7i}, we know that there exists some $C>0$ such that, if there exists some $\delta^* > 0$ such that $x, T_\omega(x), \dots , T_\omega^{n-1} \notin (- \delta^*, \delta^*)$, then there will exist some $\eta = \eta (\delta^*)$ such that $DT_\omega^n \ge C e^{\eta n}$. Now, $\tilde B(\delta)$ is not necessarily symmetrical, however, we can simply set $\delta^*$ as the smallest $|y| \in \partial \tilde B(\delta)$ and then use Proposition \ref{Shen:Prop_2.7i}. 

Now, if $C e^{\frac \eta 2 n} \ge 1/ \D (\delta)$, then the proof is complete with $\kappa = 1$ and $\hat \eta = \eta /2$. Assume this is not the case, i.e. $n \le N(\delta)$, where $N(\delta) \in \NN_0$ is the largest integer such that $C e^{\frac \eta 2 n} < 1/ \D (\delta)$. Furthermore, if $\eps$ is sufficiently small, then because $n$ is bounded we know that $x, T_\omega(x), \dots , T_\omega^{n-1} \notin \tilde B(\delta)$ implies that $x, T_0(x), \dots , T_0^{n-1} \notin \tilde B(0.9 \delta)$ as well as $DT_\omega^n \ge DT_0^n/2$. Thus, by Proposition \ref{prop:Shen_4.2}, we know that there exists some constant $\kappa_0>0$ satisfying
\begin{align*}
	DT_\omega^n(x) \ge \frac{\kappa_0}{2 \D (\delta)}.
\end{align*}
Thus, if we fix some $\eta' >0$ such that $e^{\eta' N(\delta)} < 2$, then setting $\hat \eta = \eta '$ and $\kappa = \kappa_0/4$ gives us the desired result.
\end{proof}

\end{lemma}

\begin{lemma}\label{Shen:lemma_5.4}
Let $\eps >0$ be sufficiently small, and let $\omega \in \Omega_\eps$. Furthermore, let $\delta_0>0$ be a sufficiently small constant, and let $0 < \delta \le \delta_0$. Assume that we have some some $n \in \NN_0$ and some $x \in B_{4 \delta}(0)$ satisfying either $n = M(\delta)$ and $T_\omega^{M(\delta)}(x) \in \tilde B(\delta_0)$, or alternatively $n > M(\delta)$,  $T_\omega^M(\delta)(x) \notin \tilde B(\delta_0)$, and $T_\omega^{M(\delta) + 1}(x), \dots , T_\omega^{n-1}(x) \notin  B(\delta_0)$. Then there exists a constant $\zeta_3 > 0$ (independent of $\delta_0$, $\delta$, $\eps$, $\omega$, $x$ and $n$) such that
\begin{align}
	DT_\omega^n(x) \ge \frac{ \Lambda_0(\delta)^{\zeta_3}}{\D(\delta)} \label{ineq:shen_5.10}.
\end{align}
Furthermore, if $T_\omega^{n}(x) \notin \tilde B(\delta)$, then we have
\begin{align}
	DT_\omega^{n+1}(x) \ge  \Lambda_0(\delta)^{\zeta_3} \Big( \frac{|T_\omega^n(x) - T_0(0)|}{\delta}  \Big)^{1 - 1/s}. \label{shen:ineq_5.11}
\end{align}

\begin{proof}
	For the first case where $n = M(\delta)$, this follows directly from Proposition \ref{Shen:prop_5.2}, where all we need to do is replace $\zeta_1$ and $\zeta_2$ with $\zeta_3 = \max \{\zeta_1, \zeta_2\}/2$. For the second case, we set $\delta_3 = |T_\omega^M(x) - T_0(0)|$ and $\delta_4 = |T_\omega^{M+1}(x) - T(0)|$. Notice that there exists a constant $C_1> 0$ such that $\delta_4 \le C_1 \delta_3 $. Indeed, if we have $T_\omega^M(x) \notin \tilde B(\delta_*)$, then $\delta_4 \le 1$ and $\delta_3 \ge \delta_*$, and thus the inequality holds for $C_1 = 1 / \delta_*$ (recall $\delta_*$ from \ref{delta_star}). If this is not the case, then we must have $\delta_4 \le \delta_3 + \varepsilon \le 2 \delta_3$, where we recall $\eps$ is the maximum perturbation size, in which case the statement holds for $C_1=2$.

Now, by inequality \ref{Shen:eq_5.5} in Proposition \ref{Shen:prop_5.2} and Lemma \ref{Shen:Lemma_5.3}, we have 
\begin{align*}
	DT_\omega^n(x) &= DT_\omega^{M+1}(x) \cdot \prod_{j = M + 1}^{n-1} DT_{\sigma^j \omega}(T_\omega^{j}(x)) \\
	& \ge \frac{\kappa \Lambda_0 (\delta)^{\zeta_2}}{\D(\delta_0)} \Big( \frac{\delta_0}{\delta} \cdot \frac{\delta_3}{\delta_4} \Big)^{1 - 1/s}.
\end{align*}
Thus, combining this with the statement above, we have
\begin{align}
 DT_\omega^n(x) \ge \frac{C_2 \Lambda_0 (\delta)^{\zeta_2}}{\D(\delta_0)}\Big( \frac{\delta_0}{\delta} \Big)^{1 - 1/s}. \label{eq:Shen_5.12}
\end{align}
Furthermore, since $\delta \le \delta_0$, we have
\begin{align*}
	DT_\omega^n(x) \ge C_3 \frac{ \Lambda_0 (\delta_0)^{\zeta_2}}{\D(\delta)} ,
\end{align*}
and since $\delta_0$ is sufficiently small, then \ref{ineq:shen_5.10} follows. 

Now, if $T_\omega^{n}(x) \notin \tilde B(\delta)$, then we set $\rho = |T_\omega^n(x) - T_0(0)|$, and from the fact that $T_\omega^{M(\delta) + 1}(x), \dots , T_\omega^{n-1}(x) \notin  B(\delta_0)$ we therefore have $\rho \ge \delta$. Clearly then $DT_{\sigma^n \omega}(T_\omega(x)) \ge \D(\rho)$. Furthermore, since we have $\rho < \delta_0$, then by inequality \ref{eq:Shen_5.12}, we know there must exist some constant $C_4>0$ such that we have
	\begin{align*}
		DT_\omega^{n+1} (x) \ge C_4 \Lambda_0(\delta)^{\zeta_2} \Big( \frac{\rho}{\delta}  \Big)^{1 - 1/s}.
	\end{align*} 
Since $\delta_0$ is assumed to be sufficiently small, the desired result holds. 
\end{proof}
\end{lemma}

\begin{proof}[Proof of Proposition \ref{Shen:Proposition_5.1}]
Let us assume our $\delta_0 $ is sufficiently small so that for all $0 < \delta \le \delta_0$ we have $\Lambda_0(\delta) > 1$. Let $s_1$ be the first time that $\nu_1 \ge M(\delta)$ and $T_\omega^{\nu_1}(x) \in \tilde B(\delta_0)$. If $n = \nu_1$, then the desired result follows from Lemma \ref{Shen:lemma_5.4}.

Alternatively, if $\nu_1 < n$, then we set $\rho_1 = |T_\omega^{\nu_1}(x)|$. By inequality \ref{shen:ineq_5.11}, we have
\begin{align}
	DT_\omega^{\nu_1 + 1}(x) \ge \Lambda_0(\delta)^{\zeta_3} \Big( \frac{\rho_1}{\delta}  \Big)^{1 - \frac 1 s}. \label{Shen:eq_5.13}
\end{align}	
Now, let $\nu_2 > \nu_1$ be the second time that $\nu_2 \ge M(\delta_1)$ and $T_\omega^{\nu_1}(x) \in \tilde B(\delta_0)$. If $\nu_2 = n$, then we stop, otherwise we define $\delta_2$ in a similar fashion to $\rho_1$ and then find $\nu_3$. We continue this process until we reach some $\nu_k = n$, and then apply inequality \ref{shen:ineq_5.11} to obtain for all $i = 1, \dots , k - 2$
\begin{align*}
	\prod_{j =\nu_i+1}^{\nu_{i + 1}} DT_{\sigma^j \omega}(T_\omega^j(x)) \ge \Lambda_0(\rho_i)^{\zeta_3} \Big( \frac{\rho_{i + 1}}{\delta_i}  \Big)^{1 - \frac 1 s},
\end{align*}
and using Lemma \ref{Shen:lemma_5.4}, we have
\begin{align*}
	\prod_{j = \nu_{k -1}+ 1}^{n -1}DT_{\sigma^j \omega}(T_\omega^j(x)) \ge \frac{\Lambda_0(\delta_{k-1})^{\zeta_3}}{\D(\delta)}.
\end{align*}
Thus, combining the above inequalities together, we have
\begin{align*}
	DT_\omega^n(x) &= DT_\omega^{n+1}(x) \Big( \prod_{i = 1}^{k - 2} \prod_{j = \nu_i + 1}^{\nu_{i + 1}} DT_{\sigma^j \omega}(T_\omega^j(x))  \Big) \prod_{j = \nu_{k-1}+1}^{n -1} DT_{\sigma^j \omega}(T_\omega^j(x)) \\
	& \ge \frac{\Lambda_0(\delta)^{\zeta_3}}{\D(\delta)} \prod_{i = 1}^{k - 2} \Lambda_0(\rho_j)^{\zeta_3} \Big( \frac{\rho_{k-1}}{\delta}  \Big)^{1 - \frac 1 s} 
	 \ge \frac{\Lambda_0(\delta)^{\zeta_3}}{\D(\delta)}.
\end{align*}
Thus, the result holds.
\end{proof}

Now, using all the previous results, we can prove Proposition \ref{prop:nonuniform_expansion}:

\begin{proof}
\textbf{Statement 1):}\\ \\
This follows directly from Lemma \ref{Shen:Lemma_5.3}. Indeed, from Lemma \ref{Shen:Lemma_5.3} we know that because $x, T_\omega(x), \dots, T_\omega^{n-1}(x) \notin \tilde B(\delta)$ and $T_\omega^n(x) \in \tilde B(2\delta)$, then 
\begin{align*}
		DT_\omega^n(x) \ge \frac{\kappa}{\D (\delta)} \Big( \frac{\delta}{\delta_3} \Big)^{1 - \frac 1 s} e^{\hat \eta n},
\end{align*}
where $\delta_3 = \max\{ |x - T_0(0)|, \delta  \}$. The smallest value $\Big( \frac{\delta}{\delta_3} \Big)^{1 - \frac 1 s}$ can take is $(\delta/2)^{1 - \frac 1 s}$. Thus, if we set $A = A(\delta) = \frac{\kappa}{\D (\delta)} (\delta/2)^{1 - \frac 1 s}$, then this gives us the desired result for fixed $\delta >0$.\\ \\
\textbf{Statement 2):}\\ \\
This follows directly from Proposition \ref{Shen:Proposition_5.1} and Lemma \ref{Shen:Lemma_5.3}. Indeed, as in the proof for the first statement, from Lemma \ref{Shen:Lemma_5.3} we know that because $x, T_\omega(x), \dots,$ $ T_\omega^{n-1}(x) \notin \tilde B(\delta)$ and $T_\omega^n(x) \in \tilde B(2\delta)$, then 
\begin{align*}
		DT_\omega^n(x) \ge \frac{\kappa}{\D (\delta)} \Big( \frac{\delta}{\delta_3} \Big)^{1 - \frac 1 s} e^{\hat \eta n},
\end{align*}
However, we already know that $|x - T_0(0)| \le 4 \delta$, so we therefore have 
$$
	\Big( \frac{\delta}{\delta_3} \Big)^{1 - \frac 1 s} \ge \Big( \frac 1 4 \Big)^{1 - \frac 1 s},
$$
which can just be absorbed into the $\kappa > 0$ term, thereby giving us
\begin{align}
		DT_\omega^n(x) \ge \frac{\kappa}{\D (\delta)} e^{\hat \eta n}.
\end{align}
Finally, if we set a constant $\beta(\delta)$ that satisfies the following criteria:
\begin{enumerate}
	\item $\beta(\delta) \to 0 $ as $\delta \to 0$
	\item $\hat \Lambda(\delta) = \Lambda(\delta)^{\beta(\delta)}e^{1 - \beta(\delta)} \ge 2e$ for all sufficiently small $\delta >0$
	\item $\Lambda(\delta)^{\beta(\delta)} \to \infty$ as $\delta \to 0$,
\end{enumerate}
and if we set
$$
	\eta_0(\delta) = (1 - \beta(\delta)) \hat \eta(\delta),
$$
then we have
\begin{align*}
	DT_\omega^n(x) \ge \frac{\hat \Lambda(\delta)}{\D (\delta)} e^{ \eta_0 n}.
\end{align*}
\end{proof}

\subsection{Good and Bad Parameter sets}
In order for us to use the random Young tower to obtain the desired result, we must ensure that we satisfy the expansion criteria. However, since the derivative goes to zero near the discontinuity at zero, if a typical orbit goes too close to the discontinuity too often, then we lose expansion. In this section, I prove that as $n \to \infty$, the measure of $\omega$'s for any given fixed $x$ for which this happens decays exponentially.

\begin{definition}[Return depth]
For some small $\delta >0$ and for fixed $\omega \in \Omega$, we denote
$$
	r_\delta (x) = r_\delta (\omega, x) := \inf \big\{ r \in \mathbb Z_+ \Hquad| \Hquad  T_\omega'(x) \cdot |x| \ge e^{-r}\delta  \big\}.
$$
Note that for shorthand we will write $r_\delta(T_\omega^j(x))= r_\delta(\sigma^j\omega, T_\omega^j(x))$.
\end{definition}

One should note that by definition, one can show an equivalent condition for return depth using the order of singularity condition: if $r_\delta (x) = r^*$, then using the order of singularity condition we have
\begin{align*}
	|x| \ge \frac{e^{-r^*}\delta}{T_\omega'(x)} \ge \frac{e^{-r^*}\delta}{K_2 |x|^{s-1}}, \\ 
\end{align*}
which implies
$$
	|x|^{s} \ge \frac{e^{-r^*}\delta}{K_2} \logeq |x| \ge \frac{e^{-r^*}\delta^{\frac 1 s }}{K_2}.
$$

We now need to define a `bad' set of pairs $(\omega, x) \in \Omega_\eps \times I$ for a given iterate $n$ and show that the measure of this set decays exponentially. To understand what we mean by `bad', note the following proposition:

\begin{proposition}\label{prop:bounded_return_depth_expansion}
	There exist constants $c, C, \lambda >0$ such that for any sufficiently small $\delta>0$ and $0 < \eps \le \delta $, and for all $\omega \in \Omega$, if $\sum_{i=0}^{n-1} r_\delta(T_\omega^i(x)) < cn$ and $T_\omega^n(x) \in \tilde B(\delta)$, then 
$$
	DT_\omega^n(x) \ge C e^{\lambda n}.
$$
\end{proposition} 

In the proof in \cite{Du}, Du uses Lemma 7.7 from \cite{Shen}. Thus, before we can prove this proposition, we need to adapt a version of Shen's lemma:

\begin{lemma}[Lemma 7.7, \cite{Shen}]\label{lemma:return_depths}
	Let $\Gamma= \Gamma(\omega, y, n)$ denote the number of returns for $(\omega, x)$ up to time $n$ (excluding $t=0$) and let $\nu_1, \dots, \nu_\Gamma$ be the ordered return times. There exists a constant $C>0$ such that, if $y \in \tilde B(\delta)$, then
	$$
		\log \Big( \frac{DT_\omega^n(y)|y| |T_\omega^{\nu_1}(y)| \dots |T_\omega^{\nu_\Gamma}(y)|}{|\tilde B(\delta)|^\Gamma} \Big) \ge \log C + \kappa n - \sum_{i=0}^{n-1} r_\delta(T_\omega^i(y)).
	$$ 
	Furthermore, if both $y, T_\omega^n(y) \in \tilde B(\delta)$, then for $(\omega, x)$, we have
	$$
		\log \Big( \frac{DT_\omega^n(y)|y| |T_\omega^{\nu_1}(y)| \dots |T_\omega^{\nu_\Gamma}(y)|}{|\tilde B(\delta)|^\Gamma} \Big) \ge \kappa n - \sum_{i=0}^{n-1} r_\delta(T_\omega^i(y)).
	$$

\begin{proof}
	We give the proof of the case where both $y, T_\omega^n(y) \in \tilde B(\delta)$. The proof for the case where just $y \in \tilde B(\delta)$ follows in exactly the same way, with the exception that one uses the weaker lower bound for the derivative between the last return time and $n$.

	First, we assume the case that $T_\omega^j(y) \notin \tilde B(\delta)$ for all $1 \le j < n$, i.e. $\Gamma = 1$. Since we are already assuming $y \in \tilde B(\delta)$, we therefore have $|T_\omega(y) -T(0)| \le 4\delta$. Thus, we can use the non-uniform expansion result (\ref{ineq:nue2}), so clearly we have
	$$
		DT_{\sigma \omega}^{n-1}(T_\omega(x)) \ge \frac{1}{\D(\delta)}e^{\kappa n},
	$$
where $\kappa>0$ is some suitably chosen constant since $\delta$ is fixed. Thus, we have
	\begin{align*}
		\frac{DT_\omega^n(y)|y|}{|\tilde B (\delta)|} &= DT_{\sigma \omega}^{n-1}(T_\omega(y)) \frac{DT_\omega(y)|y|}{|\tilde B (\delta)|}
		 \ge DT_{\sigma \omega}^{n-1}(T_\omega(y)) \frac{e^{-r_\delta (\omega, y)}\delta}{|\tilde B (\delta)|} \\
		& \ge \frac{1}{\D(\delta)}e^{\kappa n}e^{-r_\delta (\omega, y)} \D(\delta) 
		 =  e^{\kappa n -r_\delta (\omega, y)}.
	\end{align*}
	Thus, clearly
	$$
		\log \Big( \frac{DT_\omega^n(y)|y|}{|\tilde B(\delta)|} \Big) \ge \kappa n -r_\delta (\omega, y).
	$$
	Now, consider when it is not the case that $T_\omega^j(y) \notin \tilde B(\delta)$ for all $1 \le j < n$, i.e. $\Gamma = 2$. Say there is one iterate $1 \le k < n$ where $T_\omega^k(y) \in \tilde B(\delta)$. Then we simply split the orbit into $y, T_\omega(y), \dots , T_\omega^{k-1}(y)$ and $T_\omega^{k}(y), \dots, T_\omega^n(y)$ and then apply the same procedure above to each segment. This will give us
$$
	DT_{\sigma \omega}^{k-1}(T_\omega(x)) \ge \frac{1}{D(\delta)}e^{\kappa k} \quad \text{and} \quad  DT_{\sigma^{k+1} \omega}^{n-k-1}(T_\omega^{k+1}(x)) \ge \frac{1}{D(\delta)}e^{\kappa (n-k)}
$$
and thus
	\begin{align*}
		&\frac{DT_\omega^{n-k}(T_\omega^k(y))|y| |T_\omega^k(y)|}{|\tilde B (\delta)|^2} = \frac{DT_{ \omega}^{k}(y)|y|}{|\tilde B (\delta)|}  \frac{DT_{\sigma^k \omega}^{n-k}(T_\omega^k(y))|T_\omega^k(y)|}{|\tilde B (\delta)|} \\
		& \ge DT_{\sigma \omega}^{k-1}(T_\omega(y)) \frac{DT_\omega(y)|y|}{|\tilde B (\delta)|} DT_{\sigma^{k+1} \omega}^{n-k-1}(T_\omega^{k+1}(y)) \frac{DT_{\sigma^k \omega}(T_\omega^k(y))|T_\omega^k(y)|}{|\tilde B (\delta)|}\\
		& \ge  \exp (\kappa k -r_\delta (\omega, y))  \exp (\kappa (n-k) -r_\delta (\sigma^k \omega, T_\omega^k(y))) \\
		& = \exp(\kappa n -r_\delta (\omega, y) - r_\delta (\sigma^k \omega, T_\omega^k(y))).
	\end{align*}
One follows the same procedure if there are more returns.
\end{proof}
\end{lemma}

\begin{proof}[Proof of proposition 2.2]
	From Lemma 2.3, we know that for $y, T_\omega^n(y) \in \tilde B(\delta)$ we have
	$$
		\log \Big( \frac{DT_\omega^n(y)|y| |T_\omega^{\nu_1}(y)| \dots |T_\omega^{\nu_\Gamma}(y)|}{|\tilde B(\delta)|^\Gamma} \Big) \ge \kappa n - \sum_{i=0}^{n-1} r_\delta(T_\omega^i(y)).
	$$
	Choosing some small $0 < c < \kappa$ and setting $\lambda = \kappa - c$, if our starting point $x$ is in $\tilde B(\delta)$, then we simply consider $y=x$ and directly use the above Lemma:
	$$
		DT_\omega^n(y) \ge \frac{|\tilde B(\delta)|}{|x|} \frac{|\tilde B(\delta)|}{|T_\omega^{\nu_1}(y)|} \dots \frac{|\tilde B(\delta)|}{|T_\omega^{\nu_\Gamma}(y)|} \exp(\kappa n - \sum_{i=0}^{n-1} r_\delta (T_\omega^i(y)) ) > e^{\lambda n}.
	$$
Here we can get rid of each $|\tilde B(\delta)|/|T_\omega^i(x)|$ for $i \in \{0 , \nu_1, \dots , \nu_\Gamma\}$ because $T_\omega^i(x) \in \tilde B(\delta)$. 

Now, if $x \notin \tilde B(\delta)$, then we instead consider the starting point $y=T_\omega^{\nu_1}(x)$ in the context of the above lemma. Clearly we have
$$
	\sum_{i=0}^{n-\nu_1-1} r_\delta(T_{\sigma^{\nu_1} \omega}^{i}(y)) = \sum_{i=0}^{n-1} r_\delta( T_\omega^i(x)) < cn,
$$ 
thus we can write
$$
	DT_{\sigma^{\nu_1} \omega}^{n-\nu_1}(y) \ge \exp \big({\kappa(n-\nu_1) - \sum_{i=0}^{n-\nu_1-1} r_\delta(T_{ \omega}^i(y))})  > \exp(\kappa(n-\nu_1) - cn \big) = \exp(\lambda n - \kappa \nu_1).
$$
Furthermore, since we have assumed $x \notin \tilde B(\delta)$, and by definition $T_\omega^{\nu_1}(x) \in \tilde B(\delta)$, we can apply inequality (\ref{ineq:nue2}) to obtain
$$
	DT_\omega^{\nu_1}(x) \ge A \exp(\kappa \nu_1).
$$
Combining these together, we obtain
$$
	DT_\omega^n(x) = DT_{\sigma^{\nu_1} \omega}^{n-\nu_1}(y) \cdot DT_\omega^{\nu_1}(x)= \exp(\lambda n - \kappa \nu_1) \cdot A \exp(\kappa \nu_1) = A e^{\lambda n}.
$$
\end{proof}

It is clear that for fixed $\delta$ and $\eps$ we consider a pair $(\omega, x)$ `good' if it satisfies the sum of return depths condition above. If a pair does not satisfy the sum of return depths condition, then we consider it to be `bad'. Formally, we define the set
$$
	E_n = E_n(\delta, \eps) := \Big\{ (\omega, x) \in \Omega_\eps \times I \Hquad | \Hquad \sum_{j=0}^{n-1} r_\delta(T_\omega^j(x)) \ge cn   \Big\}.
$$
Clearly we want to show that this set decays exponentially:

\begin{proposition}\label{Shen_prop_7.1}
	There exists a small constant $c >0$ such that, for all sufficiently small $\delta >0$, there exist constants $ \gamma(\delta), C(\delta)>0$ such for sufficiently small $\eps > 0$ we have
	\begin{align}
		(\PP_\eps \times \mathrm{Leb})(E_n) \le C(\delta) e^{- \gamma(\delta)n}.
	\end{align}
\end{proposition}

We prove this using the following sequence of lemmas.

\begin{lemma}\label{Shen_lemma_7.2}
	Let $\delta > 0$ be sufficiently small, and denote $\hat R_\omega (x) = \min \{ n \in \NN \Hquad | \Hquad T_\omega^n(x) \in \tilde B(\delta)  \} $. For any sufficiently small $\eps > 0$, all $\omega \in \Omega_\eps$, and all $x \in I$ such that $| x - T_0(0)| \le 4 \delta$ and $n = \hat R_\omega (x)< \infty$, if $J$ is the component of $(T_\omega^n)^{-1}(\tilde B(\delta))$ containing $x$, then $T_\omega^n|_J : J \to T_\omega^n(J)$ is a diffeomorphism with bounded distortion $\N (T_\omega^n | J) \le 1$ and size $|J| \le \delta$. 

\begin{proof}
	Recall the constant $\theta_0>0$ from Lemma \ref{lemma:shen_2.3}, and define
	$$
		\hat J_{x,n} = \Big[ x - \frac{\theta_0}{A(\omega, x , n)},   x + \frac{\theta_0}{A(\omega, x , n)} \Big] \cap I_\pm,
	$$ 
Let us set $\theta = \theta_0/e$. Then $T_\omega |_{\hat J_{x,n} } : \hat J_{x,n} \to T_\omega^n(\hat J_{x,n})$ is a diffeomorphism and $\N (T_\omega^n | \hat J_{x,n}) \le 1$. We know that for any $y \in \partial \hat J_{x,n} \setminus \{ -1 , 1  \}$, we have $| T_\omega^n(x) - T_\omega^n(y)   | \ge |DT_\omega^n(x)| \theta_0 / (e A(\omega, x, n)) \ge |\tilde B(\delta)| $. Therefore, we know that $J \subset J_{x,n} $, and thus $\N (T_\omega^n | \hat J_{x,n}) \le 1$. Furthermore, since we have $|x - T_0(x)| \le 4 \delta$, from Proposition \ref{prop:nonuniform_expansion} we know that $DT_\omega^n(x) > e/ \D(\delta)$, and thus $|J| \le e |\tilde B(\delta)|/ DT_\omega^n(x) < \delta$.
\end{proof}
\end{lemma}

\begin{lemma}\label{Shen_lemma_7.3}
	There exist constants $K_1 > 0$ such that, for sufficiently small $\delta >0$, sufficiently small $\eps >0$, all $\omega \in \Omega$, and all $r \in \N$, for the set
	$$
		Y_\delta^\omega (r) = \big\{ y \in B(0, 2 \delta) \Hquad | \Hquad r_\delta(T_\omega^{\hat R_\omega}(x)) \ge r  \big\}
	$$
	we have
	\begin{align}\label{Shen_eq_7.5}
		|Y_\delta^\omega (r)  | \le K_1 \delta e^{- r/s}.
	\end{align}
\end{lemma}	

\begin{sublemma}\label{Shen_prop_5.7}
	For any fixed constants $0 < \xi < \xi' \le 2$, and for sufficiently small $\eps >0$, $\omega \in \Omega_\eps$ and $n \in \NN$, if $J$ is an interval such that $J \cap \tilde B(\xi \delta) \neq \emptyset$ and $T_\omega^n(J) \subset \tilde B(2 \delta)$, then $J \subset \tilde B(\delta)$.
	\begin{proof}
		Consider the following condition:
		\begin{align}\label{shen_eq_5.20}
			T_\omega^j(J) \cap \tilde B (\delta) = \emptyset \text{ for every $1 \le j < n$}.
		\end{align}
		First, we consider the case when this holds, and then the case when it does not.

		For the first case, our main aim is to show that every component $J'$ of the intersection $J \cap (\tilde B(\xi' \delta) \setminus \tilde B( \xi \delta)) $ is a proper subset of $\tilde B(\xi ' \delta)$, thereby implying that $J' \subset \tilde B (\xi' \delta)$. Furthermore, we know that there exists some constant $C = C(\xi)$ such that, as long as $\delta$ is sufficiently small, we have $DT_\omega(x) \ge C \D(\delta) $ for all $x \in \tilde B(2 \delta) \setminus \tilde B(\xi \delta)$. Using statement 2 of Proposition \ref{prop:nonuniform_expansion}, we therefore have 
		$$
			DT_\omega^n(x) = DT_{\sigma \omega}^{n-1}(T_\omega(x)) \cdot DT_\omega(x) \ge \frac{\Lambda(\delta)}{\D(\delta)} C \D(\delta) = C \Lambda(\delta)
		$$
		for every $x \in J'$, and thus
		$$
			|J'| \le \frac{|\tilde B(2 \delta)|}{\inf_{x \in J'}|DT_\omega^n(x)|} \le \frac{2 |\tilde B(\delta)|}{C\Lambda(\delta)}.
		$$
		For sufficiently small $\delta$, we therefore know $|J'|$ is significantly smaller than $|\tilde B(\delta)|$. Furthermore, knowing that $J' \cap \partial \tilde B(\xi \delta) \neq \emptyset$, we therefore know that $J'$ is a proper subset of $\tilde B( \xi' \delta)$, thereby proving this sublemma.

		If Condition \ref{shen_eq_5.20} does not hold, then consider the iterates $0 \le \nu_1 < \nu_2 < \dots < \nu_k < n$ where $T_\omega^{\nu_i} \cap \tilde B(\delta) \neq \emptyset$ for each $i = 1, \dots, k$ and no other iterates. If we use $2$ and $1$ in place of $\xi'$ and $\xi$ respectively, then we know that $T_\omega^{\nu_k}(J) \subset \tilde B(2 \delta)$. If we repeat this process, we have $T_\omega^{\nu_{k-1}}(J) \subset \tilde B(2 \delta), \dots , T_\omega^{\nu_1}(J) \subset \tilde B(2 \delta)$ until we obtain $J \subset \tilde B(2 \delta)$.
 	\end{proof}
\end{sublemma}

\begin{proof}[Proof of Lemma \ref{Shen_lemma_7.3}]
	First, for any $\omega \in \Omega$ consider the set 
	$$
		Z_\delta^{\omega}(r) = \big\{ x \in \tilde B(\delta) \Hquad | \Hquad r_\delta(\omega, x) \ge r \big\}.   
	$$
	Using the definition of $r_\delta$ and the order of singularity condition, there exists some constant $K_2 >0$ such that
	$$
		|Z_\delta^{\omega}(r)| \le K_2 e^{-r/s}|\tilde B(\delta) |,
	$$
	and furthermore, there exists some constant $r^*>0$ such that $Z^{\omega}(r) \subset \tilde B(\delta /e)$ for all $r < r^*$.

	Now, for all $r < r^*$, inequality \ref{Shen_eq_7.5} holds for some sufficiently large $K_1$. For the case where $r \ge r^*$, for any $y \in Y_\delta^\omega (r) $, we set $n = \hat R_\omega(y)$ and set $J= J(y)$ as the component of $(T_\omega^n)^{-1}(\tilde B(\delta))$ which contains $y$. We wish to prove that there exists a constant $K_3 >0 $ such that
	\begin{align}
		|J \cap Y_\delta(\omega, n)| \le K_3 e^{-r/s}|J|
	\end{align}
	To prove this, recall that since $r \ge r^*$, we have $T_\omega^n(y) \in \tilde B(\delta /e)$. Thus, at least one component of $\tilde B(\delta) \setminus \tilde B(\delta / e)$ is contained within   $T_\omega^n(J)$, thereby implying that there exists some $\gamma >0$ such that $|T_\omega^n(J)|/|\tilde B(\delta)| \ge \gamma > 0$. Using Lemma \ref{Shen_lemma_7.2}, we know that that $\N(T_\omega^n | J) \le 1$. 
	
	To prove \ref{Shen_eq_7.5}, we need to show that
	\begin{align}\label{eq:shen_7.8}
		T_\omega^n(J \cap Y_\delta^\omega(r)) \subset Z_\delta^{\sigma^n \omega}.
	\end{align}
	Now, we prove by contradiction that $\hat R_\omega(y') = \hat R_\omega(y)$ for all $y' \in J \cap Y_\delta^\omega(q)$. Indeed, suppose instead we have $1 \le n' < n$, where $n' = \hat R_\omega(y')$. We know that $T_\omega^{n'}(y') \in \tilde B( \delta / e)$, which combined with Proposition \ref{prop:nonuniform_expansion} implies $T_\omega^{n'}(J(y')) \subset \tilde B(\delta)$. But this contradicts our assumption that $n' $ is the first return.

We now conclude the proof by once again using Besicovitch's covering lemma: the intervals in the set $\{J(y) \Hquad | \Hquad y \in Y_\delta(r)\}$ form an open cover of $Y_{\delta}(r)$, and thus a sub-covering can be found that has bounded multiplicity of intersection. From Lemma \ref{Shen_lemma_7.2}, we already now that $|J| \le \delta$, which implies that $J \subset B(0, 3 \delta)$. Thus, inequality \ref{Shen_eq_7.5} holds.
\end{proof}

\begin{lemma}\label{Shen_lemma_7.4}
	There exist constants $K_0, \rho_0 > 0$ such that, for sufficiently small $\delta >0$ and sufficiently small $\eps >0$, if $x \in \tilde B(\delta)$ and 
	for the set
	$$
		\hat{Y}_\delta (x, r) = \Big\{ \omega \in \Omega_\eps \Hquad | \Hquad r_\delta(T_\omega^{\hat R_\omega}(x)) \ge r   \Big\}
	$$
	we have
	\begin{align}
		\PP_\eps (\hat Y_\delta(x, r)) \le K_0 e^{-\rho_0 r}.
	\end{align}

\begin{proof}
	
For all $\omega \in \tilde B(\delta)$ and all $\omega \in \Omega_\eps$, consider the set 
$$
	X_\delta(\omega, x , r) = \Big\{ t \in [- \eps, \eps] \Hquad | \Hquad r_\delta(   T_\omega^{R_\omega}(T_t(x))) \ge r \Big\}. 
$$	
Clearly, if $t \in X_\delta(\omega, x , r)$, then $T_t(x) \in Y_\delta(r)$, and thus
$$
	P_\eps(X_\delta(\omega, x , r)) \le |Y_\delta^\omega (r)  | \le K_1 \delta e^{- r/s},
$$
and thus
$$
	\PP_\eps (\hat Y_\delta(x, r)) = \int_{\Omega_\eps} P_\eps(X_\delta(\omega, x , r)) d \PP_\eps \le K_1 \delta e^{- r/s}.
$$
\end{proof}

\end{lemma}

\begin{lemma}\label{Shen_lemma_7.5}
	For all fixed $n \in \NN$ and all $(r_1, \dots, r_n) = \mathbf r \in \NN^n$, and for all sufficiently small $\eps, \delta >0$ and $x \in \tilde B(\delta)$, for the set
	$$
		\hat Y_\delta^n(x, \mathbf r) = \Big\{ \omega \in \Omega \Hquad | \Hquad  r_\delta(T_\omega^{\nu_i }(x)) \ge r_i \text{ for all } i = 1, \dots, n \Big\}, 
	$$
	where 
	\begin{align*}
		\nu_1 &= \hat R_\omega(x), \\
		\nu_2 &= \hat R_{\sigma^{\nu_1}\omega}(T_\omega^{\nu_1}(x)),\\
		& \vdots \\
		\nu_n &= \hat R_{\sigma^{\nu_1 + \dots + \nu_{n-1}}\omega}(T_\omega^{\nu_1 + \dots + \nu_{n-1}}(x)),
	\end{align*}
	we have
	\begin{align}
		\PP_\eps (\hat Y_\delta^n(x, \mathbf r)) \le K_0^n e^{- \rho_0 | \mathbf{q}|} 
	\end{align}

	\begin{proof}
		Lemma \ref{Shen_lemma_7.4} gives us the desired result for $n = 1$. For the case $n >1$, we simply need to show that 
		$$
			\PP_\eps(\hat Y_\delta^n (x , \mathbf r)) \le K_0 e^{- \rho_0 r_n} \PP_\eps(\hat Y_\delta^n (x , \mathbf{\hat r})), 
		$$
		where $\mathbf{\hat r} = (r_1, \dots , r_{n-1})$. Indeed, consider the set
		$$
			W_m^n = \Big\{  \omega \in \hat Y_\delta^{n-1}  (x, \mathbf{\hat r})  \Hquad | \Hquad \sum_{i = 0}^{n - 2} \hat R_\omega(T_\omega^i (x) )  = m \Big\}.
		$$
		Clearly $\hat Y_\delta (x, \mathbf{\hat r}) \subset \cup_{i = 1}^{\infty} W_i^n$. Furthermore, for any two $\omega, \tilde \omega \in \Omega_\eps$ such that $\omega_i = \tilde \omega_i$ for all $0 \le i < m$, then we know that $\omega \in W_m^n$ if and only if $\tilde \omega \in W_m^n$ as well. This is equivalent to saying $W_m = [- \eps, \eps]^{- \NN} \times A_m \times [-\eps , \eps]^\NN$, $A_m$ being a measurable subset of $[-\eps, \eps]^m$. Clearly $\PP_\eps(W_m^n) = P_\eps^m(A_m)$.

Now, for any $\omega \in W_m^n$, from Lemma \ref{Shen_lemma_7.4} we have
\begin{align}
	\PP_\eps (\hat Y_\delta (T_\omega^m(x), r))  \le K_1 \delta e^{- r/s}
\end{align} 
Furthermore, notice that we have
\begin{align*}
	W_m^n \cap \hat Y_\delta^n (x, \mathbf r) = \Big\{  (\hat{  \mathbf{\omega}}, \Hquad \omega_0, \Hquad \omega_1, \dots, \Hquad &\omega_{m-1}, \tilde \omega  ) \Hquad | \Hquad \hat{ \mathbf{ \omega}} \in [- \eps, \eps]^{-\NN}, \\
 &(\omega_0, \dots, \omega_{m-1}) \in A_m, \Hquad \tilde \omega \in \hat Y_\delta(T_\omega^m(x), r_n)   \Big\}.
\end{align*}
Thus, by Fubini's theorem, we have
\begin{align*}
	\PP_\eps(W_m^n \cap \hat Y_\delta^n (x, \mathbf r) ) = \int_{A_m} \PP_\eps (\hat Y_\delta(T_{\omega_{m-1}}\circ \dots \circ T_{\omega_0}(x), r_n)) \le K_1 \delta e^{- r/s} \PP_\eps(W_m^n).
\end{align*}
Thus, we have
\begin{align*}
	\PP_\eps(\hat Y_\delta^n(x, \mathbf r)) &= \sum_{i=1}^\infty \PP_\eps(W_i^n \cap \hat Y_\delta^n (x, \mathbf r) ) \\
	&\le K_1 \delta e^{- r_n/s} \PP_\eps (\hat Y_\delta^{n-1}(x, \mathbf r)).
\end{align*}
Using recursion, this proves the lemma.
	\end{proof}	
\end{lemma}

\begin{proof}[Proof of Proposition \ref{Shen_prop_7.1}]
	Take the constants $K_0 , \rho_0 >0 $ from Lemma \ref{Shen_lemma_7.4}, and set $\rho= \min \{ \rho_0/5, 1/2s  \}$. Using Stirling's approximation, we know that there will exist some constant $\kappa > 1$ such that, for all constants $n, m \in \NN$ with $n > \kappa m/2$, we have ${ n + m -1 \choose m-1  } \le e^{\rho n}$. 

Let us set
$$
	\Xi = \Big\{ (\omega, x) \in \Omega \times I \Hquad | \Hquad x \in \tilde B(\delta), \Hquad r_\delta (x) \ge cn/2 - \kappa   \Big\}
$$
as well as
$$
	Q_{n_1}^{n_2} (\omega, x) = \sum_{j = n_1}^{n_2} r_\delta(T_\omega^j(x)). 
$$
Furthermore, for any $m', n \ge 0$ we set
$$
	\Xi_\ell^m = \Big\{ (\omega, x) \in \Omega_\eps \times I \Hquad | \Hquad \exists N \ge 1 \text{ such that } \Gamma(\omega, x,N) =  \ell, Q_1^N (\omega, x) = m   \Big\},
$$
and define $\I = \{ (\ell, m) \in \NN^2 \Hquad | \Hquad 2 \ell \ge \max \{ cn, \kappa m  \} > 0  \}$. We want to prove that 
\begin{align}
	E_n \subset \Xi \cup \Big( \bigcup_{(m', n) \in \I }   \Xi_\ell^{ m}   \Big)
\end{align}
and then show both of these sets decay exponentially.

To show $E_n$ is a subset, we must first show that for every $(\omega, x) \in E_n$ there exists some $n' \ge 0$ such that
\begin{align}\label{Shen_eq_7.13}
	Q_0^{n'}(\omega, x) > \max \big\{ cn - \kappa, \kappa \Gamma_0^{n'}(\omega, x)  \big\}.
\end{align}
To see this, set $n_0$ as the first iterate where $Q_0^{n_0}(\omega, x) \ge cn$. In the case where $Q_0^{n_0}(\omega, x) > \kappa \Gamma^{n_0}(\omega, x)$, then we can simply set $n' = n_0$. If this is not the case, then we instead take $n' < n_0$ to be the largest iterate where $T_\omega^{n'}(x) \in \tilde B (\delta)$. By the definition of $n_0$, we have $Q_0^{n'}(\omega, x) < m$, and thus we must have $Q_0^{n'}(\omega , x) >  \kappa\Gamma^{n'}(\omega, x)$ instead since $(\omega, x) \in E_n $. Furthermore, clearly $\Gamma^{n_0}(\omega, x) = \Gamma^{n'}(\omega, x) + 1$, and thus we have
$$
	Q_0^n(\omega, x) >\kappa \Gamma_0^{n_0} - \kappa \ge Q_0^{n_0}(\omega, x) - \kappa \ge cn - \kappa. 
$$
This confirms inequality \ref{Shen_eq_7.13}. Now, for every $(\omega, x) \in E_n \setminus \Xi$, if we set $\ell = Q_1^n(\omega, x)$ and $m = \Gamma_1^n(\omega, x)$, then this implies that $(\omega, x) \in \Xi$. It is now simply a question of showing that $(\ell, m) \in \I$. Indeed, using the fact thats 
$$
	2 \ell = 2(Q^{n'}(\omega, x) - r_\delta(\omega, x)) \ge 2 ((cn - \kappa) - (cn/2 - \kappa)) = cn,
$$
which implies that $m >0$. Furthermore, using the fact we have $2r_\delta(\omega, x) \le n - \kappa < Q_0^n(\omega, x)$, we know that $2 \ell \ge Q_0^n(\omega, x) > \kappa \Gamma_0^n(\omega,x) \ge \kappa n$, thereby implying that $(\ell, m) \in \I$ and therefore that $E_n$ is a subset.

For the estimate on $|\Xi |$, note that by the definition of $r_\delta$, there is a constant $C=C(\kappa)$ such that for all $\omega \in \Omega_\eps$, we have 
$$
	\big|\big\{ x \in \tilde B(\delta) \Hquad | \Hquad r_\delta(x) \ge cn/ 2 - \kappa  \big\}\big| \le C |\tilde B(\delta)| e^{-n/2s} \le C |\tilde B(\delta)| e^{-\rho n},
$$
and thus by using Fubini's theorem, we have
\begin{align}
	|\Xi| \le C |\tilde B(\delta)| e^{-\rho n}.
\end{align}

As for the estimate on $|\Xi_\ell^{ m}|$ for $(\ell, m) \in \I$, for every $x \in \tilde B(\delta)$ we set
$$
	\hat E_\ell^{ m }(x) = \big\{ \omega \in \Omega_\eps \Hquad | \Hquad x \in \Xi_\ell^{ m}   \big\}.
$$
Clearly, we have 
$$
	\hat E_\ell^{ m }(x) \subset \bigcup_{\substack{\mathbf{r} \in \NN^m \\ |\mathbf{r}| = \ell}} \hat Y_{\delta}^m(x, \mathbf{r}).
$$
Clearly, there are ${\ell + m - 1 \choose m - 1} $ of $\mathbf r \in \NN^m$ with $|\mathbf r| = \ell$. Thus, using Lemma \ref{Shen_lemma_7.5} we have
\begin{align*}
	\PP_\eps (\hat E_\ell^{ m}(x)) &\le {\ell + m - 1 \choose m - 1}  \le K_0^m e^{- \rho_0 | \mathbf{\ell}|} \\
	& \le  e^{\rho \ell} e^{-\rho_0 \ell} e^{\kappa \rho m}  
	 \le e^{- \rho(4 \ell - \kappa m)}  
	\le e^{- 2 \rho \ell}.
\end{align*}
Thus, we have
\begin{align*}
	\sum_{(\ell, m) \in \I} \PP_\eps (\Xi_\ell^{ m}) \le \sum_{\ell \ge \floor{n/2}} \sum_{\substack{ m \text{ s.t. } \\  (\ell, m) \in \I}}  \PP_\eps (\Xi_\ell^{ m})
 \le \sum_{\ell \ge \floor{n/2}} e^{- 2 \rho \ell} \frac{2 \ell}{\kappa}  \le C e^{- \rho n}.
\end{align*}
\end{proof}

\subsection{Hyperbolic times and Hyperbolic Return Times}
The idea of hyperbolic times originates in \cite{ALV} and was then adapted to the random case in \cite{BBMD}. In \cite{Du} the definition of the return depth is changed slightly to be in terms of the product of the return depth and the derivative, rather than just the return depth. We give a formal definition here.

\begin{definition}[Hyperbolic times]
	Let us fix two constants $\delta>0$ and $c' >0$. An iterate $n$ is considered a $(\delta, c')$-hyperbolic time for $(\omega, x)$ if for all $0 \le k \le n-1$ we have
	\begin{align}
		\sum_{j=k}^{n-1} r_\delta( T_\omega^j(x)) < c'(n - k)
	\end{align}
\end{definition}

 We should note two particular properties of $(\delta, c')$-hyperbolic times stemming from this definition. The first is that if $n_1 \ge 1$ is a $(\delta, c')$-hyperbolic time for $(\omega, x)$, and $n_2 - n_1  \ge 1$ is a hyperbolic time for $(\sigma^{n_1}\omega, T_\omega^{n_1}(x))$, then $ n_2$ is also a hyperbolic time for $(\omega, x)$. Indeed, for any $n_1 \le k < n_2$, we just have
\begin{align*}
	\sum_{j=k}^{n_2 -1} r_\delta( T_\omega^j(x)) = \sum_{j=k - n_1}^{n_2 - n_1 -1} r_\delta( T_{\sigma^{n_1}\omega}^j(x))  < c'((n_2- n_1)  - (k-n_1)) = c'(n_2 - k). 	
\end{align*}
Likewise, if $0 \le k < n_1$, we have
\begin{align*}
	\sum_{j=k}^{n_2-1} r_\delta( T_\omega^j(x)) &= \sum_{j=k}^{n_1 -1} r_\delta( T_\omega^j(x)) + \sum_{j=n_1}^{n_2-1} r_\delta( T_\omega^j(x))\\
	& < c'(n_1 - k) + c'(n_2 - n_1)
	= c'(n_2 - k). 
\end{align*}
The second property is that if $n \ge 1$ is a $(\delta, c')$-hyperbolic time for $(\omega, x)$, then $n - m$ is a $(\delta, c')$-hyperbolic time for $(\sigma^m \omega, T_\omega^m(x))$ for each $1 \le m < n$. Indeed, we have for every $0 \le k < n - m$
\begin{align*}
	\sum_{j=k}^{n - m-1} r_\delta( T_{\sigma^m\omega}^j(x)) = \sum_{j=k + m}^{n-1} r_\delta( T_{\omega}^j(x)) < c'(n - (k + m))= c'((n - m) - k)).
\end{align*}

We denote by $h(\omega, x)$ the first $(\delta, c')$-hyperbolic time for $(\omega, x)$. We want to show that the set of $(\omega, x)$'s that have not experienced a $(\delta, c')$-hyperbolic time by time $n$, i.e. $h(\omega, x) >n$, decays exponentially as $n$ goes to infinity:
\begin{lemma}\label{ineq:tails_of_hyperbolic_times}
	For any fixed $0 < \delta < 1$ and $c' > c$, all $(\omega, x) \notin B_n$ experience a $(\delta, c')$-hyperbolic time smaller than or equal to $n$. Furthermore, there exist some $C=C(\delta)$ and $\gamma=\gamma(\delta)$ such that, for sufficiently small $\eps$, we have
\[
	(\mathbb P_\eps \times \text{Leb}) \big(\big\{ (\omega, x) \Hquad | \Hquad h(\omega, x) > n   \big\}\big) \le Ce^{-\gamma n}.
\]
\end{lemma}
Before we prove this, we make use of \textit{Pliss's Lemma}:

\begin{lemma}[Pliss's Lemma]
	Let us fix some constants $A \ge c_2 > c_1$, and set $\theta = (c_2 - c_1)/(A - c_1)$. For any real numbers $a_0, a_1, \dots, a_{n-1} \le A$ satisfying
$$
	\sum_{j=0}^{n-1} a_j > c_2 n, 
$$
there exists some constant $\ell \ge \theta n$ and integers $1 \le n_1 < n_2 < \dots < n_\ell \le n$ such that, for every $i = 1, 2, \dots, \ell$ and every $0 \le k < n_i$, we have
$$
	\sum_{j=k}^{n_i-1} a_j > c_1 (n_i - k).
$$ 

\end{lemma}

\begin{proof}[Proof of Lemma \ref{ineq:tails_of_hyperbolic_times}]
	We follow the same steps as in \cite{Du}. First, we want to show that $(\omega,x) \notin E_n$ implies that $h(\omega, x) \le n$. To this end, we set $a_i = c - r_\delta(T_\omega^i(x))$. By definition, for any $(\omega, x) \notin E_n$, we have $\sum_{j=0}^{n-1}a_j > 0$. For any $c' > c$, if we set $A = c$,  $c_2 = 0$ and $c_1 = c - c' <0$ (and thus $A \ge c_2 > c_1$), we can apply Pliss's Lemma: there exists some integer $1 \le n_1 \le n$ such that for all $0 \le k < n_1$, we have
	$$
		\sum_{j=k}^{n_1-1} a_j > c_1(n_1 -k)= (c - c')(n_1 -k).
	$$
Using how we have defined $a_i$, this can be rewritten as
	$$
		\sum_{j=k}^{n_1-1} r_\delta(T_\omega^j(x)) < c'(n_1 -k).
	$$
Thus, $n_1$ is a $(\delta, c')$-hyperbolic time for $(\omega, x)$.

For the tails of first hyperbolic times, we simply use the contrapositive that $h(\omega, x)>n$ implies $(\omega, x) \in E_n$, which we already know from the previous proposition decays exponentially. Thus, the proof is complete.
\end{proof}

\begin{remark}
	The above lemma implies that almost every $(\omega, x) \in \Omega_\eps \times I$ has infinitely many hyperbolic times. Furthermore, the proof of the above not only implies that if $(\omega, x) \in I \setminus E_n$ there will be at least one hyperbolic time less than or equal to $n$, but between $1$ and $n$ there will in fact be at least $(1 - \frac{c}{c'})n$ hyperbolic times. This means that, provided $(\omega, x)$ is not in $E_n$ at time $n$, we can choose whatever lower bound for proportion of times that must be hyperbolic times that we like, depending on our choice of $c$ and $c'$. 
\end{remark}

Now, for our tower construction to work, we must satisfy the expansion, Markov and bounded distortion conditions. We prove these properties in the following proposition:

\begin{proposition}\label{prop:V_x,n}
	For some fixed $c' \in (c, \kappa)$, assume that $(\omega, x)$ has a $(\delta, c')$-hyperbolic time $n$. There exists some open neighborhood $V_{x,n}^\omega$ of $x$ and some constant $\delta_0$ dependent on $\delta$ such that for all $0 \le k < n$, we have the that the map
$$
	T_{\sigma^m \omega}^{n - k}: T_\omega^k(V_{x,n}^\omega) \to B(T_\omega^n(x), \delta_0)
$$
is a diffeomorphism (i.e. the map is Markov) of bounded distortion. Furthermore, there exist constants $C, \lambda'>0$ such that, for all elements $y \in V_{x,n}^\omega$, we have 
$$
	DT_{\sigma^k \omega}^{n-k}(T_\omega^k(y)) \ge C e^{\lambda'(n-k)/2} \quad \text{and} \quad \mathcal{N} (T_{\sigma^k \omega}^{n-k}| T_\omega^k(V_{x,n}^\omega)) <1.
$$
\begin{proof}
	Let us fix some $0 \le k \le n-1$. First, we prove that there exists constants $C, \lambda' >0$ such that
\begin{align}\label{ineq:hyperbolic_time_expansion1}
	DT_{\sigma^k \omega}^{n-k}(T_\omega^k(x)) \ge Ce^{\lambda'(n-k)}.
\end{align}
We note that even though we know that $n$ is a $(\delta, c')$-hyperbolic time for $(\omega, x)$, we do not know anything about the positions of $T_\omega^k(x)$ and $T_\omega^n(x)$, and in particular we do not know if they are in $\tilde B(\delta)$. We need to consider all possible cases.

The simplest possible case is when $T_\omega^k(x) \notin B(\delta) $ and there are no return times between iterate $k$ and $n$, regardless of if $n$ is a return time. In this case, we can simply directly apply Proposition \ref{prop:nonuniform_expansion} to obtain
$$
	DT_{\sigma^k \omega}^{n-k}(T_\omega^k(x)) \ge C e^{\kappa (n-k)}.
$$
or if $T_\omega^n(x) \in B(\delta) $ then 
$$
	DT_{\sigma^k \omega}^{n-k}(T_\omega^k(x)) \ge  e^{\kappa (n-k)}.
$$
However, if there are return times between iterates $n$ and $k$ or if $T_\omega^k(x) \in B(\delta) $ then we need to use Lemma \ref{lemma:return_depths}. 

As in Lemma \ref{lemma:return_depths}, let $\Gamma$ denote the number of returns between iterates $k$ and $n$, and let $\nu_1, \dots, \nu_\Gamma$ be those returns. Setting $y_k = T_\omega^{k}(x)$ and assuming that $y_k \in \tilde B(\delta)$, by Lemma \ref{lemma:return_depths} we have
$$
	\log \Big( \frac{DT_\omega^{n-k}(y_k)|y_k| |T_\omega^{r_1}(y_k)| \dots |T_\omega^{r_\Gamma}(y_k)|}{|\tilde B(\delta)|^\Gamma} \Big) \ge \log C + \kappa (n-k) - \sum_{j=0}^{n-k-1} r_\delta(T_{\sigma^k \omega}^j(y_k)).\\ \\
$$
Or equivalently 
	$$
		DT_\omega^{n-k}(y_k) \ge C \exp \Big( \kappa (n-k) - \sum_{j=0}^{n-k-1} r_\delta(T_{\sigma^k \omega}^j(y_k))  \Big)
	$$
Now, we adapt the proof of Proposition \ref{prop:bounded_return_depth_expansion}. Using the definition of $n$ being a $(\delta, c')$-hyperbolic time, we have
$$
	\sum_{j=0}^{n-k-1} r_\delta(T_{\sigma^k \omega}^j(y_k))  =\sum_{j=k}^{n-1} r_\delta(T_\omega^j(x)) < c' (n-k).
$$
Thus, we have
$$
	\exp(\kappa (n-k) - \sum_{i=n-k}^{n-1} r_\delta (T_\omega^i(y_k)) ) \ge \exp(\kappa (n-k) - c' (n-k) ).
$$
Setting $\lambda' = \kappa - c'$, we therefore have
$$
	DT_{\sigma^k \omega}^{n-k}(y_k) \ge e^{\lambda'n}.
$$
Likewise, for the case where $y_k \notin \tilde B(\delta)$, we just follow the same steps as in the proof of Proposition \ref{prop:bounded_return_depth_expansion}: setting $\nu$ as the first return time for $(\sigma^k \omega, y_k)$, we have
\begin{align*}
	DT_{\sigma^{k+\nu} \omega}^{n-(k+\nu)}(y_{k+\nu}) &\ge \exp \big({\kappa(n-(k+\nu))   - \sum_{i=0}^{n-k-\nu - 1} r_\delta(T_{\sigma^{k+\nu}\omega}^i(y_{k+\nu}))}\big)\\
	&=  \exp \big({\kappa(n-(k+r)) -  \sum_{i=k+ \nu}^{n-1} r_\delta(T_\omega^i(x))}  \big) \\
	&\ge  \exp \big({\kappa(n-(k+r)) -  \sum_{i=k}^{n-1} r_\delta(T_\omega^i(x))}  \big) \\
	&  > \exp(\kappa(n-(k+\nu)) - c'(n-k) \big) .\\
	&= e^{\lambda'(n-k) - \kappa \nu}.
\end{align*}
Then using Proposition \ref{prop:bounded_return_depth_expansion} we have
$$
	DT_\omega^{k+\nu}(y_k) \ge C e^{\kappa \nu},
$$
and thus
$$
	DT_{\sigma^k \omega}^{n-k}(y_k) =  DT_{\sigma^{k+\nu} \omega}^{n-(k+\nu)} (y_{k+\nu}) \cdot DT_{\sigma^k \omega}^\nu(y_k) \ge C e^{\lambda'(n-k)}.
$$
Thus, inequality (\ref{ineq:hyperbolic_time_expansion1}) holds, where we assume for convience we just assume $\lambda' \le \kappa$.

Now, we choose $\delta_0>0$ such that $C' C^{-1}K_2 \delta_0 \delta^{1/s}< \lambda'/2$, where $C'$ is the constant from the admissibility condition, $C$ is from Proposition \ref{prop:bounded_return_depth_expansion}, and $s$ is the order of singularity. Notice that the definition of hyperbolic time implies that for each $0 \le k \le n-1$, we have $r_\delta (y_k) < c'(n-k)$, which therefore implies
\begin{align}\label{ineq:hyperbolic_time_min_distance}
	|T_\omega^k(x)| > \frac{1}{K_2}\delta^{\frac 1 s}e^{-c'(n-k)}.
\end{align}
This implies that for any $0 \le k \le n-1$ and $y \in B(T_\omega^k(x), C^{-1}\delta_0 e^{-\lambda'(n-k)/2})$, we have $|x - y| <|x|$. Indeed, from the condition on $\delta_0$, we have
$$
	\delta_0 < \frac{C}{K_2} \frac{\lambda'\delta^{s}}{2C'} < \frac{C }{K_2}\delta^{\frac 1 s},
$$
where we assume we have taken $\delta$ small enough that $\lambda'\delta^{s}/2C' < \delta^{1/s}$. Thus, we have
\begin{align}\label{ineq:min_distance_k=0}
	C^{-1} \delta_0 e^{-\frac{\lambda'}{2}(n-k)} < \frac{1}{K_2}\delta^{\frac 1 s}e^{-c'(n-k)}, 
\end{align}
where we assume that we have chosen $c'$ sufficiently small such that $0 < c' < \kappa/3$. This means we can use the admissibility condition in the following way:
\begin{align*}
	\Big| \log \frac{DT_{\sigma^k \omega}(T_\omega^k(x))}{DT_{\sigma^k \omega}(y)|}   &\Big| \le C' \frac{|T_\omega^k(x) - y|}{|T_\omega^k(x)|} \\
	& \le C' \frac{C^{-1}\delta_0 e^{-\lambda'(n-k)/2}}{\frac{1}{K_2} \delta^{\frac 1 s}e^{-c'(n-k)}}\\
	&= C'C^{-1} K_2 \delta_0 \delta^{\frac 1 s} e^{-(\lambda'/2 - c')(n-k)}
	< \frac{\lambda'}{2}.
\end{align*}
Clearly, this indicates
\begin{align}\label{ineq:hyperbolic_time_bounded_distortion}
	T_{\sigma^k \omega}'(y) > e^{- \lambda'/2}T_{\sigma^k \omega}'(T_\omega^k(x)).
\end{align}
Now, let us consider the case when $k=n-1$. Then by  Proposition \ref{prop:bounded_return_depth_expansion} and inequality (\ref{ineq:hyperbolic_time_bounded_distortion}), for all $y \in B(T_\omega^{n-1}(x), C^{-1}\delta_0 e^{-\frac{\lambda'}{2}})$ we have
\begin{align*}
	DT_{\sigma^{n-1}\omega}(y) > e^{-\frac{\lambda'}{2}}DT_{\sigma^{n-1}\omega}(T_\omega^{n-1}(x)) \ge C e^{\frac{\lambda'}{2}}.
\end{align*}
This clearly means that, if we consider some subinterval $J \subset  B(T_\omega^{n-1}(x), C^{-1}\delta_0 e^{-\frac{\lambda'}{2}})$, since the above lower bound on the derivative applies to all $y \in B(T_\omega^{n-1}(x), C^{-1} \delta_0 e^{-\frac{\lambda'}{2}})$, we have
$$
	|T_{\sigma^{n-1}\omega}(J)| \ge C e^{\frac{\lambda'}{2}} |J|.
$$
Thus, taking $J =B(T_\omega^{n-1}(x), \delta_0 e^{-\frac{\lambda'}{2}})$ itself, we have
$$
	|T_{\sigma^{n-1}\omega}(J)| \ge C e^{\frac{\lambda'}{2}} C^{-1}  \delta_0 e^{-\frac{\lambda'}{2}} =  \delta_0,
$$
and therefore
\begin{align}\label{subset:k_1}
	B(T_\omega^n(x),  \delta_0) \subset T_{\sigma^{n-1}\omega} (B(T_\omega^{n-1}(x),C^{-1} \delta_0 e^{-\frac{\lambda'}{2}})).
\end{align}
Now, clearly we want to show that for our fixed $k$ we have
\begin{align}\label{subset:hyperbolic_time}
	B(T_\omega^n(x),  \delta_0) \subset T_{\sigma^{k}\omega}^{n-k} (B(T_\omega^{k}(x),C^{-1} \delta_0 e^{-\frac{\lambda'}{2}(n-k)})).
\end{align}
This is automatically true for $k=n-1$ by the calculations above. We now prove this for all $0 \le k  < n-1$.

For the case, $k=n-2$, we use proof by contradiction and assume that (\ref{subset:hyperbolic_time}) is not true. Then there exists a point $z \in (T_{\sigma^{n-2}\omega}^2)^{-1}(B(T_\omega^n(x), \delta_0))$ such that $|T_\omega^{n-2}(x) - z| = C^{-1}\delta_0 e^{-\lambda'}$. Without loss of generality, assume $z < T_\omega^k(x)$. For any point $y \in (z, T_\omega^{n-2}(x))$, clearly we have
$$
	y \in B(T_\omega^{n-2}(x), C^{-1}\delta_0 e^{-\lambda'}).
$$
In addition, trivially by construction we also have
$$
	y \in (T_{\sigma^{n-2}\omega}^2)^{-1}(B(T_\omega^n(x), \delta_0)),
$$
which implies
$$
	T_{\sigma^{n-2}\omega}(y) = T_\omega^{n-1}(y^*) \in (T_{\sigma^{n-1}\omega})^{-1}(B(T_\omega^n(x), \delta_0)) \subset B(T_\omega^{n-1}(x),C^{-1} \delta_0 e^{-\frac{\lambda'}{2}}),
$$
where $y^*$ is just some point that satisfies $T_\omega^{n-2}(y^*) = y$, and where the subset in the above statement comes from simply applying $(T_{\sigma^{n-1}\omega})^{-1}$ to both sides of (\ref{subset:k_1}). Thus, for any $y$ and $T_{\sigma^{n-2}\omega}(y)$ we have by inequaltiy (\ref{ineq:hyperbolic_time_bounded_distortion})
\begin{align*}
	DT_{\sigma^{n-2}\omega}(y) > e^{-\frac{\lambda'}{2}}DT_{\sigma^{n-2}\omega}(T_\omega^{n-2}(x)) 
\end{align*}
and
\begin{align*}
	DT_{\sigma^{n-1}\omega}(T_{\sigma^{n-2}\omega}(y)) > e^{-\frac{\lambda'}{2}}DT_{\sigma^{n-1}\omega}(T_\omega^{n-1}(x)).
\end{align*}
Thus, we have
\begin{align*}
	DT_{\sigma^{n-2}\omega}^2(y) &= DT_{\sigma^{n-1}\omega}(T_{\sigma^{n-2}\omega}(y)) \cdot DT_{\sigma^{n-2}\omega}(y) \\
	&> e^{-\frac{\lambda'}{2}}DT_{\sigma^{n-1}\omega}(T_\omega^{n-1}(x)) \cdot e^{-\frac{\lambda'}{2}}DT_{\sigma^{n-2}\omega}(T_\omega^{n-2}(x))   \\
	&= e^{-\lambda'}DT_{\sigma^{n-2}\omega}^2(T_\omega^{n-2}(x))
	\ge Ce^{\lambda'}.
\end{align*}
This therefore means we have
\begin{align*}
	|T_\omega^n(x) - T_{\sigma^{n-2}\omega}^2(z)| > Ce^{\lambda'} |T_\omega^{n-2}(x) - z| =  Ce^{\lambda'}C^{-1}\delta_0 e^{-\lambda'} = \delta_0.
\end{align*}	
But this implies that $z \notin (T_{\sigma^{n-2}\omega}^2)^{-1}(B(T_\omega^n(x), \delta_0))$, which contradicts our assumption on $z$. Thus, we must have
$$
	B(T_\omega^n(x),  \delta_0) \subset T_{\sigma^{n-2}\omega}^{2} (B(T_\omega^{n-2}(x),C^{-1} \delta_0 e^{-\lambda'})).
$$
One can prove the rest of the cases  $0 \le k  < n-2$ by induction using the same steps as above. We assume that (\ref{subset:hyperbolic_time}) has already been proven for all $k < j \le n -1$. We then assume that (\ref{subset:hyperbolic_time}) does not hold for $k$ for proof by contradiction. Therefore, we can choose a point  $z \in (T_{\sigma^{k}\omega}^{n-k})^{-1}(B(T_\omega^n(x), \delta_0))$ such that $|z - T_\omega^k(x)   |= C^{-1} \delta_0 e^{-\frac{\lambda'}{2}(n-k)}$. We can then show that for all $y \in (z, T_\omega^k(x))$ we have
$$
	T_{\sigma^k \omega}^\ell(y) \in  B(T_\omega^{k+\ell}(x), C^{-1}\delta_0 e^{-\lambda'(n-(k+\ell))/2})
$$
for $\ell = 0, 1, \dots , n-(k+1)$ by construction and using the fact( \ref{subset:hyperbolic_time}) has already been proven for $k < j \le n -1$. We can therefore use (\ref{ineq:hyperbolic_time_bounded_distortion}) to obtain the lower bound on the derivate 
$$
	DT_{\sigma^k \omega}^{n-k}(y) > Ce^{\frac{\lambda'}{2}(n-k)},
$$
which can then be used to show
$$
	|T_\omega^n(x) - T_{\sigma^k \omega}^{n-k}(z)| > \delta_0,
$$
thereby contradicting our assumption and proving our desired result.

We denote the set
$$
	V_{x,n}^\omega = (T_\omega^n)^{-1}\big(B(T_\omega^n(x), \delta_0) \big) \cap B(x, C^{-1} \delta_0 e^{-\frac{\lambda'}{2}n}).
$$
First, we want to show that $T_{\sigma^k \omega}^{n-k}|_{T_\omega^k(V_{x,n}^\omega)}: T_\omega^k(V_{x,n}^\omega) \to  B(T_\omega^n(x), \delta_0)$ is a diffeomorphism for all $0 \le k \le n-1$. It is not immediately obvious why this should be the case. Like in the unimodal case, for all $\omega \in \Omega$, $T_\omega$ is clearly not bijective. In the unimodal case, the map restricted to some subinterval $J \subset I$ is not bijective when $J$ covers any two-sided neighborhood of $0$. However, in our case $T_\omega$ would have to be restricted to some interval $J$ such that there exists a $0\le k < n$ such that $T_\omega^k(J)$ covers all or nearly all of one half of $I$ and a sufficiently large part of the other half, depending on the values of $T_{\sigma^k \omega}(1)$ and $T_{\sigma^k \omega}(-1)$. Despite this, because of the discontinuity at $0$, there is the possibility that if there is a $0\le k < n$ such that $T_\omega^k(J)$ covers any neighbourhood of $0$, then $T_\omega^{k + 1}(J)$ will be split into two components, and the images of these components under further iterations may intersect. Thus, to prove this does not happen for $n$ and $V_{x,n}^\omega $, we need to show that $T_\omega^{k}(V_{x,n}^\omega)$ never intersects $0$ and is contained entirely in one half of $I$ or the other for all $0 \le k < n-1$.

But this is just a consequence of the fact that for all $0 \le k \le n-1$ and $y \in B(T_\omega^k(x), C^{-1}\delta_0 e^{-\lambda'(n-k)/2})$, we have $|x - y| <|x|$. Indeed, notice that using (\ref{subset:hyperbolic_time}) we have
\begin{align*}
	T_\omega^k(V_{x,n}^\omega) &\subset (T_{\sigma^k \omega}^ {n-k}|_{T_\omega^k(V_{x,n}^\omega)})^{-1} \big(B(T_\omega^n(x), \delta_0 \big) \cap T_\omega^k \big( B(x, C^{-1} \delta_0 e^{-\frac{\lambda'}{2}(n-k)}) \big) \\
	&  \subset (T_{\sigma^k \omega}^ {n-k}|_{T_\omega^k(V_{x,n}^\omega)})^{-1} \big(B(T_\omega^n(x), \delta_0 \big)  \\
	& \subset (T_{\sigma^k \omega}^ {n-k}|_{T_\omega^k(V_{x,n}^\omega)})^{-1} \big( T_{\sigma^{k}\omega}^{n-k} (B(T_\omega^{k}(x),C^{-1} \delta_0 e^{-\frac{\lambda'}{2}(n-k)})))  \\
	& = B(T_\omega^{k}(x),C^{-1} \delta_0 e^{-\frac{\lambda'}{2}(n-k)}).
\end{align*}

We now need to show that for each $0 \le k < n-1$, the set $T_\omega^k(V_{x,n}^\omega)  \subset B(T_\omega^{k}(x),C^{-1} \delta_0 e^{-\frac{\lambda'}{2}(n-k)})$ is connected and does not intersect 0. Indeed, this follows from \ref{ineq:hyperbolic_time_min_distance} and \ref{ineq:min_distance_k=0}, which respectively imply that $T_\omega^{k}(x)$ is a minimum distance away from 0 and that the distance between $T_\omega^{k}(x)$ and $y \in B(T_\omega^{k}(x),C^{-1} \delta_0 e^{-\frac{\lambda'}{2}(n-k)})$ is small enough that $y$ must be on the same side of $I$ as $T_\omega^{k}(x)$. Thus, we have the diffeomorphism property.

For the bounded distortion, we make use of the admissibility condition again: for all $y \in T_\omega^k(V_{x,n}^\omega)$ and all $0 \le k < n$, we have
\begin{align*}
	&\Big| \log \frac{DT_{\sigma^k \omega}^{n-k}(T_\omega^k(x))}{DT_{\sigma^k \omega}^{n-k}(y)}   \Big| = \sum_{j = 0}^{n-k-1} \Big| \log \frac{DT_{\sigma^{k+j}\omega}(T_{\sigma^k \omega}^j(y))}{DT_{\sigma^{k+j}\omega}(T_{\omega}^{k+j}(x))}   \Big| \\
	&\le  \sum_{j = 0}^{n-k-1}  C' \frac{|T_\omega^{k+j}(x)- T_{\sigma^k \omega}^j(y)|}{|T_\omega^{k+j}(x)|} 
	 \le  \sum_{j = 0}^{n-k-1} C' \frac{C^{-1}\delta_0 e^{-\lambda'(n-k-j)/2}}{\frac{1}{K_2} \delta^{\frac 1 s}e^{-c'(n-k-j)}}\\
	& \le C'C^{-1} K_2 \delta_0 \delta^{-\frac 1 s} \sum_{j = 0}^{\infty} e^{-(\lambda'/2 - c')j}
	< 1,
\end{align*}
where we have placed an extra condition on $\delta_0$ that
$$
	\delta_0 < \Big( C'C^{-1} K_2 \delta^{-\frac 1 s} \sum_{j = 0}^{\infty} e^{-(\lambda'/2 - c')j} \Big)^{-1}.
$$
Thus, for all $0 \le k < n$ and all $y \in T_\omega^k(V_{x,n}^\omega)$, we have
$$
	e^{-1}DT_{\sigma^k \omega}^{n-k}(y) < DT_{\sigma^k \omega}^{n-k}(T_\omega^k(x)) < e DT_{\sigma^k \omega}^{n-k}(y).
$$
Furthermore, this combined with (\ref{ineq:hyperbolic_time_expansion1}) prove the expansion property for all $0 \le k < n$ and all $y \in T_\omega^k(V_{x,n}^\omega)$:
\begin{align*}
	DT_{\sigma^k \omega}^{n-k}(y) > e^{-1}DT_{\sigma^k \omega}^{n-k}(T_\omega^k(x)) \ge Ce^{\lambda'(n-k)},
\end{align*}
where we just absorb the $e^{-1}$ term into the constant $C$. This completes the proof.
\end{proof}
\end{proposition}

\subsection{Hyperbolic Return Times}

Now we introduce the concept of hyperbolic \textit{return times}:

\begin{definition}[Hyperbolic Return Times]
Let us set sufficiently small $\kappa, c > 0$ and sufficiently small $\delta, \delta_0(\delta) >0$ according the previous lemmas and propositions. For a fixed $c' \in (c, \kappa)$, we consider an iterate $n$ to be a $(\delta, c')$-hyperbolic return time for $(\omega, x)$ if $n$ is a $(\delta, c')$-hyperbolic time and $T_\omega^j(x) \in \tilde B(\delta_0/2)$.
\end{definition}

We denote $h_\omega^*( x)$ as the first hyperbolic return time for the point $x$ under $T_\omega$, and we denote $\hat{R}_\omega(x) = \min\{ s \ge 1 \Hquad | \Hquad T_{\omega}^s(x) \in \tilde B(\delta_0/2) \}$. We wish to show that almost every point has a finite hyperbolic return time:

\begin{proposition}\label{prop:tails_of_hyperbolic_return}
	For sufficiently small $\delta >0$, there exist constants $C_0 = C_0(\delta) >0$ and $\gamma_0= \gamma_0(\delta)>0$ such that, for sufficiently small $\eps>0$, we have
$$
	\big(\mathbb P_\eps \times \mathrm{Leb}\big) \big( \big\{ (\omega, x) : h^*(\omega, x) > n    \big\} \big) \le C_0 e^{- \gamma n}
$$
\begin{proof}
First, we note the following:
\begin{align}
	\big\{h_\omega^*( x) >n  \big\} \subset \big\{h_\omega ( x) &>n/2  \big\} \cup \Big( \bigcup_{j=1}^{n/2}\big\{ h_\omega ( x) = j, \hat R_{\sigma^j \omega}( T_\omega^j(x)) > n/2 \} \Big).
\end{align}

Indeed, if $h_\omega(x) > n > n/2$, then that automatically implies $h_\omega^*(x) > n$, and thus $x \in \big\{ h_\omega(x) > n/2   \big\} $. Likewise, if $n/2 < h_\omega(x) \le n$, then we also have $x \in \big\{ h_\omega(x) > n/2   \big\} $. Finally, if $h_\omega(x) = j$ with $0 < j \le n/2$, we must have $\hat R_{\sigma^j \omega}(T_\omega^j(x)) > n - j \ge n/2$.

We have already shown that $\{h_\omega ( x) >n/2  \}$ decays exponentially in Proposition \ref{ineq:tails_of_hyperbolic_times}, so all that remains to do is to estimate an upper bound for the second union of sets.

Recall from Lemma \ref{lemma:tail_of_return_times} the set
$$
	\Lambda_n^\omega (U) = \big\{ x \in I  : \Hquad T_\omega^j(x) \notin U \text{ for } 0 \le j \le n-1 \big\}.
$$
Let us denote $\Lambda_n^\omega (\delta) = \Lambda_n^\omega (B(\delta)) $. We have that for any sufficiently small $\delta > 0$, all $\omega \in \Omega$, and all $n \in \mathbb Z_+$, if a subinterval $J \subset I_{\pm}$ maps diffeomorphically under $j\ge 1$ iterations onto some interval of length $\tau>0$ and has bounded distortion, then 
$$
	|J \cap \Lambda_n^\omega(\delta)| \le  C e^{-\eta n}|J| .
$$
Indeed, simply recall inequality (\ref{ineq:return_contraction}) combined with the exponential decay result of Lemma \ref{lemma:tail_of_return_times}. 

Now, let us denote $H_j(\omega) = \big\{ x \in I \Hquad | \Hquad h_\omega(x)=j  \big\}$. Recall from Proposition \ref{prop:V_x,n} that if $j$ is a hyperbolic time for every $x \in H_j(\omega)$, then for every $x \in H_j(\omega)$ there will exist a neighbourhood $V_{x, j}^\omega$ of $x$ such that the map $T_\omega^j|_{V_{x, j}^\omega}: V_{x, j}^\omega \to B(T_\omega^j(x), \delta_0)$ is diffeomorphic and has bounded distortion. Using this, we note that $e^{-1/2} |B(T_\omega^j(x), \delta_0)| \le | V_{x,j}^\omega| \le e^{1/2}|B(T_\omega^j(x), \delta_0)|$, and thus
\begin{align*}
	\frac{\big|\big\{ y \in V_{x,j}^\omega \Hquad | \Hquad \hat R_{\sigma^j \omega}(T_\omega^j(x)) > n/2    \big\}\big|}{|V_{x,j}^\omega|} \le e \frac{| B(T_\omega^j(x), \delta_0) \cap \Lambda_{n/2}^{\sigma^j \omega}(\delta_0/2)  |}{|B(T_\omega^j(x), \delta_0)|} \le e^{1 - \eta n/2},
\end{align*}
or equivalently
$$
\big	|\big\{ y \in V_{x,j}^\omega \Hquad | \Hquad \hat R_{\sigma^j \omega}(T_\omega^j(x)) > n/2    \big\}\big| \le e^{1 - \eta n/2}|V_{x,j}^\omega|.
$$
Now, clearly the set $\big\{ V_{x,j}^\omega \Hquad | \Hquad x \in H_j(\omega) \big\}$ forms an open covering of $H_j(\omega)$. Thus, similar to the proof of Lemma \ref{lemma:tail_of_return_times}, we know from the Besicovitch covering lemma that there exist finite subsets $A_1, \dots , A_{c_1}$ of $H_j(\omega)$ such that for every $i = 1, \dots , c_1$, all $V_{x, j}^\omega$ are disjoint for every $x \in A_i$. Furthermore, we have 
$$
	H_j(\omega) \subset \bigcup_{i = 1}^{c_N} \bigcup_{x \in A_i} V_{x,j}^\omega .
$$
Using these, we have
\begin{align*}
	\big|\big\{ y  \in H_j(\omega) \Hquad | \Hquad  \hat R_{\sigma^j \omega}(T_\omega^j(x)) > n/2  \big\}\big| &\le \Big|\bigcup_{i = 1}^{c_N} \bigcup_{x \in A_i} \big\{ y \in V_{x,j}^\omega \Hquad | \Hquad \hat R_{\sigma^j \omega}(T_\omega^j(x)) > n/2    \big\}  \Big| \\
	& \le \sum_{i=1 }^{c_1} \sum_{x \in A_i} \big|\big\{ y \in V_{x,j}^\omega \Hquad | \Hquad \hat R_{\sigma^j \omega}(T_\omega^j(x)) > n/2    \big\}\big| \\
	&\le  e^{1 - \eta n/2} \sum_{i=1 }^{c_1} \sum_{x \in A_i} |V_{x,j}^\omega| 
	 \le C e^{1 - \eta n/2} |H_j(\omega)|.
\end{align*}
Thus, we have
\begin{align*}
	\big|\big\{ h_\omega^*(x) > n \big\}\big| &\le \big|\big\{ h_\omega(x) > n/2 \}| +  \sum_{j = 1}^{n/2}\big|\big\{ h_\omega(x) = j \Hquad | \Hquad  \hat R_{\sigma^j \omega}(T_\omega^j(x)) > n/2  \big\}\big| \\
	& \le Ce^{- \gamma n /2} + C  e^{1 - \eta n/2} \sum_{j=1}^{n/2} |H_j(\omega)| 
	 \le C e^{- \gamma_0 n}.
\end{align*}
\end{proof}
\end{proposition}

\subsection{Construction of the Return Partition}
Now, we use the previously state results regarding hyperbolic return times in order to construct the return partition for our random tower. First, we set $\delta' = \delta_0/2$. Now, by definition, if $n$ is a $(\delta, c')$-hyperbolic return time, then $T_\omega^n(x) \in \tilde B(\delta')$. Furthermore, if $\partial I \cap V_{x,n}^\omega = \emptyset$, then $T_\omega^n|_{V_{x,n}^\omega}: V_{x,n}^\omega \to B(T_\omega^n(x), \delta_0)$ is an expanding diffeomorphism with $\mathcal N(T_\omega^n|_{V_{x,n}^\omega}) < 1$. Moreover, since $\tilde B(\delta ') \subset B(T_\omega^m(x), \delta_0) = B(T_\omega^m(x), 2 \delta')$, we have $\tilde B(\delta ') \subset T_\omega^n (V_{x,n}^\omega).$

Now, we define $J_{x, n}^\omega = (T_\omega^n)^{-1}(B(\delta'))\cap  V_{x,n}^\omega$. Cleary if $J_{x, n}^\omega \cap \{\pm 1  \} = \emptyset$, then the mapping $T_\omega^n |_{J_{x, n}^\omega}: J_{x, n}^\omega\to B(\delta')$ inherits the properties of being diffeomorphic, expanding and of bounded distortion. Furthermore, recalling Proposition \ref{prop:V_x,n} that for all $y \in V_{x,n}^\omega$ we have
$$
	DT_{\omega}^{n}(y) \ge C e^{\lambda'n/2}.
$$
Thus, the maximum size of $J_{x, n}^\omega$ is
$$
	|J_{x, n}^\omega| \le \frac{| B(\delta') |}{Ce^{\lambda'n/2}}  = C \delta' e^{-\lambda' n/2}.
$$

We prove the following proposition that both helps us to define the return partition and also gives us the aperiodicity condition:

\begin{proposition}\label{prop:aperiodicity}
There exist four points $x_1, x_2, x_3, x_4 \in  B(\delta')$ with respective associated $(\delta, c')$-hyperbolic return times $\{t^*_1, t^*_2, t^*_3, t^*_4  \}$ such that $\text{g.c.d.}(t^*_i)=1$. Furthermore, for sufficiently small $\eps>0$, the intervals $J_{x_i, t_i}^\omega$ are all pairwise disjoint and are in $ B (\delta')$. 

\begin{proof}
First, we begin by proving aperiodicity in general. Recall that our original unperturbed contracting Lorenz system is mixing. Thus, we can proceed as in [Remark 3.14, \cite{ABR}]. Since the unperturbed map admits a unique invariant probability measure, it can be lifted to the induced map over $\Delta^\ast$. Moreover, the lifted measure is invariant and mixing for the tower map. Therefore, there exists a partition $\Q^0$ of $ B(\delta')$ and a return time $\tau^0: B(\delta')\to \NN$ such that $\tau_i^0=\tau^0(Q_i)$, $Q_i\in\Q^0$ such that g.c.d.$\big\{\tau_i^0\big\}_{i=1}^{N_0}=1$ for some $N_0>1$. Now, by shrinking $\eps$ if necessary, we can ensure that the first $N_0$ elements of the partition $\Q^\omega$  satisfy $|Q_i^\omega\cap Q_i|\ge |Q_i|/2$ with $\tau_\omega^0(Q_i^\omega)=\tau_i^0$ for $i=1, \dots N_0$ and for all $\omega\in\Omega$. Thus, we may take $\eps_i=|Q_i|/2$. Notice that we define only finitely many domains in this way. Therefore, the tails estimates, distortion, etc. are not affected.

Now, assuming $N_0 \ge 4$, we can simply choose four partition elements $Q_1, Q_2, Q_3, Q_4$ and then pick a point from each, giving us the points $x_1, x_2, x_3, x_4$ with respective $(\delta, c')$-hyperbolic return times  $t^*_1, t^*_2, t^*_3, t^*_4$.
\end{proof}

\end{proposition}
 
Now, we define the return partition via a sequence of partitions. First, we define the set
$$
	\mathcal U_0(\omega) = \big\{ J_{x_i, t^*_i}^\omega \Hquad | \Hquad i = 1, 2, 3, 4 \big\}.
$$
Next, we define the set
\begin{align*}
	\mathcal U_1(\omega) = \Big\{ J_{x,1}^\omega  \Hquad | \Hquad x \in I, \text{1 is a }(\delta, c')\text{-hyperbolic} &\text{ return time for }(\omega, x), \\
		& J_{x,1}^\omega \cap (\bm{U}_0 \cup \partial I \cup \partial \tilde B(\delta')) = \emptyset \Big\}.
\end{align*}
where 
$$
	\bm{U}_0 = \bigcup_{U \in \mathcal U_0}U.
$$
Next, we define
\begin{align*}
	\mathcal U_2(\omega) = \Big\{ J_{x,2}^\omega  \Hquad | \Hquad x \in I, \text{2 is a }(\delta, c')\text{-hyperbolic} &\text{ return time for }(\omega, x), \\
		& J_{x,2}^\omega \cap (\bigcup_{j <2 }\bm{U}_j \cup \partial I \cup \partial \tilde B(\delta')) = \emptyset \Big\}
\end{align*}
with
$$
	\bm{U}_j = \bigcup_{U \in \mathcal U_j}U.
$$
Iterating this process, we obtain
\begin{align*}
	\mathcal U_k(\omega) = \Big\{ J_{x,k}^\omega  \Hquad | \Hquad x \in I, \text{k is a }(\delta, c')\text{-hyperbolic} &\text{ return time for }(\omega, x), \\
		& J_{x,k}^\omega \cap (\bigcup_{j <k }\bm{U}_j \cup \partial I \cup \partial \tilde B(\delta')) = \emptyset \Big\}
\end{align*}
Using these, we define the set
$$
	\mathcal Z_\omega = \bigcup_{j \in \mathbb Z_+} \bm U_j \cap \tilde B (\delta).
$$
Now, we need to show the following:
\begin{proposition}\label{prop:tails_returns}
For sufficiently small $\eps > 0$ and all $\omega \in \Omega_\eps$, there exist constants $C, \gamma > 0$ such that for every $n \in \mathbb Z_+$ we have
	\begin{align}
		|I \setminus \bigcup_{k \le n} \bm{U}_k | \le C e^{-\gamma n}.
	\end{align}
\end{proposition}

For a fixed typical $\omega \in \Omega$ and $x \in I$, we denote $h_i^* = h_{\omega, i}^*(x)$ as the $i$-th hyperbolic return time for $(\omega, x)$. Furthermore, we denote by $p= p_\omega(x, n)$ the number of hyperbolic return times $(\omega, x)$ experiences before and including iterate $n$. By the definition of each $\mathcal U_j$ for $j \in \mathbb Z_+$, if $x \in I \setminus \cup_{k \le n} \bm{U}_k$, then for every $1\le i \le p(x,n)$, either 
\begin{align}
	J_{x, h_i^*}^\omega \cap (\partial I \cup \partial B(\delta')) \neq \emptyset \label{cond:border_intersect}
\end{align}
or there is an iterate $q < h_i^*$ and partition element $U_q \in \mathcal U_q$ such that
\begin{align}
	U_q \cap J_{x, h_i^*}^\omega \neq \emptyset.\label{cond: partition_intersect}
\end{align}
If only condition (\ref{cond:border_intersect}) holds for all $1\le i \le p(x,n)$, then we do nothing. If there is at least one $1\le i \le p(x,n)$ such that condition (\ref{cond: partition_intersect}) holds, then we follow the following procedure. Let us define two sets of increasing values $\{ p_i \}_{i=1}^s$ and $\{ q_i \}_{i=1}^s$. For the first values, we set
\begin{align*}
	q_{ 1} = q_{\omega, 1}( x) =  q_{\omega, 1}( x, n) &:= \min \big\{ 1 \le q \le n \Hquad | \Hquad \bm U_q \cap J_{x, h_j^*}^\omega \neq \emptyset \text{ for some } j \le p  \big\} \\
  \intertext{and}
	p_{ 1} = p_{\omega, 1}( x) = p_{\omega, 1}( x, n) &:= \max \big\{ 1 \le j \le p \Hquad | \Hquad \bm U_{q_1} \cap J_{x, h_j^*}^\omega \neq \emptyset   \big\}.
\intertext{For the second values, we set}
	q_2 = q_{\omega, 2}( x) = q_{\omega, 2}( x, n) &:= \min \big\{ 1 \le q \le n \Hquad | \Hquad \bm U_q \cap J_{x, h_j^*}^\omega \neq \emptyset \text{ for some } p_1 \le j \le p  \big\}\\
\intertext{and}
	p_2 = p_{\omega, 2}( x) = p_{\omega, 2}( x, n) &:= \max \big\{ 1 \le j \le p \Hquad | \Hquad \bm U_{q_2} \cap J_{x, h_j^*}^\omega \neq \emptyset  \big\}.
\end{align*}
We repeat this process until we reach some $p_s$ with some index $ t = t_\omega(x,n)$ such that either $p_t = p$ or for all $ p_t < j \le p$, $J_{x, h_j^*}^\omega$ satisfies only condition (\ref{cond:border_intersect}) and not (\ref{cond: partition_intersect}), i.e. for all $ p_t < j \le p$, we have
$$
	J_{x, h_j^*}^\omega \cap \bigcup_{k \le n} \bm{U}_k = \emptyset
$$
and
$$
	J_{x, h_j^*}^\omega \cap (\partial I \cup \partial B(\delta')) \neq \emptyset.
$$
Now, using these values we have defined, we define the `bad' sets
$$
	X_\omega(n, t) = \Big\{ x \in  I \setminus \bigcup_{k \le n} \bm{U}_k^\omega \Hquad | \Hquad t_\omega(x,n ) = t \Big\}
$$
and
\begin{align*}
	X_\omega(n,t, \{p_i\}, \{q_i\} ) = \Big\{  x \in  I \setminus &\bigcup_{k \le n} \bm{U}_k^\omega \Hquad | \Hquad t_\omega(x,n ) = t, \text{ and for every } 1 \le j \le  t\\ 
 &\text{ we have } p_{\omega, i}(x,n) = p_i \text{ and } q_{\omega, i}(x,n) = q_i   \Big\}.
\end{align*}

Now, we give the proof of Proposition \ref{prop:tails_returns}:
\begin{proof}[Proof of Proposition \ref{prop:tails_returns}]
For all $x \in I \setminus \bigcup_{k \le n} \bm{U}_k^\omega$, we consider four possible cases. 
\\ \\
\textbf{Case 1: $t > c_1 n$.}\\ \\
Assume that $t > c_1 n$, where $0 < c_1 < 1$ is some sufficiently small constant that we will choose later. In this case, we need to show that for some small constant $0 < \alpha < 1$ that we will determine later, there exist constants $C=C(\alpha)>0$ and $\theta = \theta(\alpha) >0$ such that
$$
	\sum_{t \ge c_1 n} |X_\omega(n,t) | \le C\theta^n.
$$

For $J_{x, h^*_{p_t}}^\omega$ there will exist some $V_{q_t} \in \mathcal U_{q_t}$ satisfying $J_{x, h^*_{p_t}}^\omega \cap V_{q_t} \neq \emptyset$. Now, by bounded distortion, we clearly have
$$
	\frac{|J_{x, h^*_{p_t}}^\omega|}{ |V_{q_t}|} \le e \frac{|T_\omega^{q_t}(J_{x, h^*_{p_t}}^\omega)|}{| B(\delta')|},
$$
where we use the fact that $T_\omega^{q_t}(V_{q_t})= B(\delta')$ by the definition of $V_{q_t} \in \mathcal U_{q_t}$. Furthermore, by the mean value theorem, we know that there exists some $\xi \in T_\omega^{q_t}(J_{x, h^*_{p_t}}^\omega)$ such that
$$
	DT_{\sigma^{q_t} \omega}^{h^*_{p_t} - q_t}(\xi) = \frac{| B(\delta')|}{|T_\omega^{q_t}(J_{x, h^*_{p_t}}^\omega)|},
$$
and by the expansion property of hyperbolic times, we have
$$
	DT_{\sigma^{q_t} \omega}^{h^*_{p_t} - q_t}(\xi) \ge C e^{\lambda ' (h^*_{p_t} - q_t)/2}.
$$
Thus, combining these, we have
$$
	\frac{|J_{x, h^*_{p_t}}^\omega|}{ |V_{q_t}|} \le \frac{e}{Ce^{\lambda ' (h^*_{p_t} - q_t) /2}} \le  \frac{e}{Ce^{\lambda ' M(\delta') /2}}. 
$$

Now, for some fixed $0 < \alpha_0 < (6e)^{-1}$ we take $\delta$ to be sufficiently small so that $Ce^{\lambda ' M(\delta') /2} > \alpha_0^{-1}$. If we set $\alpha = 6 e \alpha_0$, then clearly $\alpha <1$. Furthermore, this implies
$$
	\frac{|J_{x, h^*_{p_t}}^\omega|}{ |V_{q_t}|} \le e \alpha_0 = \frac \alpha 6 \quad \implies \quad |J_{x, h^*_{p_t}}^\omega| \le \frac{\alpha}{6}|V_{q_t}|.
$$
Thus, we have
$$
	J_{x, h^*_{p_t}}^\omega \subset (1 + \alpha/6)V_{q_t},
$$
where $ (1 + \alpha/6)V_{q_t}$ denotes the interval with the same centre point as $V_{q_t}$ and with $1 + \alpha/6$ the length. 

Recalling the times $t^*_1, t^*_2, t^*_3, t^*_4$ from Proposition \ref{prop:aperiodicity}, let $t^*_0 = \max_{1 \le i \le 4} \{ t^*_i  \}$. By definition for any $1 \le i \le t - 1$,  if we take some $V_{q_i} \in \mathcal U_{q_i}$ satisfying $V_{q_i} \cap J_{x, h^*_{p_i}}^\omega \neq \emptyset$ and some $V_{q_{i+1}} \in \mathcal U_{q_{i+1}}$ satisfying $V_{q_{i+1}} \cap J_{x, h^*_{p_{i+1}}}^\omega \neq \emptyset$, if $q_i \ge t^*_0$, then clearly we have $q_i < h^*_{p_i} < h^*_{p_{i+1}}$ as well as $q_i < q_{i+1}$. Now, let us consider the maps $T_\omega^{q_i}$ and $T_\omega^{h^*_{p_i}}$. From the fact that we already know $V_{q_{i+1}} \cap J_{x, h^*_{p_{i+1}}}^\omega \neq \emptyset$, we have
$$
	\frac{|T_\omega^{q_i}(J_{x, h^*_{p_i}}^\omega)|}{|B(\delta')|} \le \frac{1}{Ce^{\lambda ' M(\delta') /2}} < \alpha_0.
$$
Additionally, since we know that $h^*_{q_{i + 1}} > h^*_{q_i}$, we therefore have
\begin{align*}
	\frac{|T_\omega^{q_i}(J_{x, h^*_{p_{i+1}}}^\omega)|}{|T_\omega^{q_i}(J_{x, h^*_{p_{i}}}^\omega)|}  &\le e \frac{|T_\omega^{q_i + ( h^*_{p_{i+1}} - q_i)}(J_{x, h^*_{p_{i+1}}}^\omega)|}{|T_\omega^{q_i + (h^*_{p_i} - q_i)}(J_{x, h^*_{p_i}}^\omega)|} \\ 
	&= e \frac{|T_\omega^{h^*_{p_i}}(J_{x, h^*_{p_{i+1}}}^\omega)|}{|B(\delta')|} 
	 \le \frac{e}{Ce^{\lambda ' M(\delta') /2}} 
	< e \alpha_0.
\end{align*}
Thus, we have
\begin{align*}
	\frac{|T_\omega^{q_i}(J_{x, h^*_{p_{i+1}}}^\omega)|}{|B(\delta')|}  = \frac{|T_\omega^{q_i}(J_{x, h^*_{p_{i+1}}}^\omega)|}{|T_\omega^{q_i}(J_{x, h^*_{p_{i}}}^\omega)|} \cdot \frac{|T_\omega^{q_i}(J_{x, h^*_{p_{i}}}^\omega)|}{|B(\delta')|} < e \alpha_0^2,
\end{align*}
i.e.
$$
	|T_\omega^{q_i}(J_{x, h^*_{p_{i+1}}}^\omega)| \le e \alpha_0^2 |B(\delta')|.
$$
Furthermore, since we know that $T_\omega^{q_i}(V_{q_{i+1}}) \cap T_\omega^{q_i}(J_{x, h^*_{p_{i+1}}}^\omega) \neq \emptyset$ and that $ e \alpha_0^2 < 1 $, we therefore know that $T_\omega^{q_i}(V_{q_{i+1}}) \cap B(2\delta') \neq \emptyset$. Thus, we have
$$
	\frac{|T_\omega^{q_i}(V_{q_{i+1}}) |}{|B(\delta')|} \le \frac{1}{Ce^{\lambda ' M(\delta') /2}} < \alpha_0.
$$
Combining all of the previous statements, we therefore have $T_\omega^{q_i}(V_{q_{i+1}}) \subset (1 + 2 \alpha_0 + e\alpha_0^2)B(\delta')$, which furthermore implies
$$
	V_{q_{i+1}} \subset (1 + 2 e \alpha_0 + e^2 \alpha_0^2)V_{q_i} \subset (1 + \alpha)V_{q_i}.
$$
Now, let us define the set
$$
	\mathcal V_i = \Big\{ V \in \mathcal U_{q_i(x)} \Hquad | \Hquad x \in X_\omega(n, t) , V \cap J_{x, h^*_{p_i(x)}}^\omega \Big\}.
$$
Recall that we have shown that for every $x \in X_\omega(n,s)$, we have
$$
	J_{x, h^*_{p_i(x)}}^\omega \subset (1 + \alpha/3) V_t.
$$
Clearly, for each $t^*_0 \le j < t$, $V_{j+1} \in \mathcal V_{j+1}$ there exists some $V_j \in \mathcal V_j$ such that
$$
	V_{j+1} \subset (1 + \alpha) V_j. 
$$
Denoting $\bm V_i = \cup_{V \in \mathcal V_i} V$, clearly we have
\begin{align*}
	|X_\omega(n,t)| &\le \Big| \Big( \bigcup_{x \in X(n,t)} J_{x, h^*_{p_t(x)}}^\omega \Big) \setminus \Big( \bigcup_{i \le n} \bm U_i \Big)   \Big| \\
	& \le \frac \alpha 3 |\bm V_t| 
	\le \frac{\alpha^2}{3} |\bm V_{t-1}| \le \dots 
%
%
	 \le \frac{\alpha^{s-t_0 +1}}{3} |\bm V_{t^*_0}| 
	 \le \frac{2 \alpha^{t - t^*_0 + 1}}{3}.
\end{align*}
Thus, altogether we have
$$
	\sum_{t \ge c_1 n} |X_\omega(n,t) | \le \sum_{t \ge c_1 n}  \frac{2 \alpha^{t - t^*_0 + 1}}{3} \le   \frac{2 \alpha^{-t^*_0 + 1}}{3(1 - \alpha)}\alpha^{c_1 n} = C_1 \theta_1^n , 
$$
where we have set  $C_1 = {2 \alpha^{-t^*_0 + 1}}/{3(1 - \alpha)}$ and $\theta_1 = \alpha^{c_1}$.  \\ \\
\textbf{Case 2}\\ \\
Consider the case where $\# \big\{ 1 \le i \le p \Hquad | \Hquad J_{x, h^*_{p_i(x)}}^\omega \cap \big( \bigcup_{k \le n} \bm U_k \big) = \emptyset  \big\} > c_2 n$, where $0 < c_2 < 1$ is some constant that will determine.

When $J_{x, h^*_{p_i(x)}}^\omega \cap \bigcup_{k \le n} \bm U_k = \emptyset$, we have $J_{x, h^*_{p_i(x)}}^\omega \cap (\partial I \cup \partial B(\delta ' )) \neq \emptyset$. Clearly, there exists a point $b \in (\partial I \cup \partial B(\delta'))$ satsifying
$$
	\# \big\{1 \le i \le p \Hquad | \Hquad b \in   J_{x, h^*_{p_i(x)}}^\omega \big\} > c_2 n /4 
$$

Now, let us denote $k = \max \{ 1 \le i \le p \Hquad | \Hquad b \in J_{x, h^*_{p_i(x)}}^\omega \}$. Since we know the number of integers $1 \le i \le p$ satisfying the desired condition is equal or greater than $c_2 n /4$, clearly the worst case is that all these integers are the bottom values, in which case we must have $k \ge c_2 n /4$. Using similar arguments to those in case 1, we can show that
$$
	|J_{x, h^*_k}^\omega | \le e \frac{|B(\delta ' )|}{Ce^{\lambda_0 h^*_k}} \le C e^{- c_2 \lambda_0 n /4} = C_2 \theta_2^n.
$$
Here we use the fact that $x \in \cup_{b \in \partial I \cup \partial B(\delta')} B(b, C e^{- c_2 \lambda_0 n/4})$. \\ \\
\textbf{Case 3}\\ \\
Let us denote $\{ k_i \}_{i=1}^m = \big\{ 1 \le k \le p \Hquad | \Hquad J_{x, h^*_i}^\omega \cap \bigcup_{j \le n} \bm U_j = \emptyset  \big\}$. Consider the case when we have either $\sum_{j=1}^t (h^*_{p_j + 1} - h^*_{p_j }) > c_3 n$ or $\sum_{j=1}^m (h^*_{k_j + 1} - h^*_{k_j }) > c_3 n$, where $0 < c_3 < 1$ is some constant we determine later.

Suppose we have a strictly increasing sequence $i_1 < \dots  < i_t $ with $i_j \in \{1, \dots , p  \}$ for each $1 \le j \le t$. Let us denote
$$
	\mathcal M_\omega = \mathcal M_\omega(i_1, \dots , i_t) = \Big\{ x  \in I \Hquad | \Hquad  \sum_{j = 1}^t (h^*_{i_j + 1} - h^*_{i_j}) > c_3 n  \Big\}.
$$
Clearly, $\mathcal M_\omega$ can be expressed as the following union:
$$
	\mathcal M_\omega = \bigcup_{\substack{ a_1, \dots, a_{t-1} \in \mathbb Z_+ \\ a_1 + \dots + a_{t-1} < c_3 n}} \mathcal M^*_\omega(a_1, \dots, a_{t-1})
$$
where we denote for fixed $i_1 \le \dots  \le i_t $
\begin{align*}
   \mathcal M_\omega^*(a_1, \dots, a_{t-1}) = \Big\{ x \in  I   \Hquad | \Hquad h^*_{i_j + 1} - h^*_{i_j} = a_j,  &1 \le j \le t-1, \\
	& h^*_{i_t } - h^*_{i_t -1}  > c_3 n  - \sum_{j = 1}^{t-1} a_j \Big\}.
\end{align*}
Furthermore, from Proposition \ref{prop:tails_of_hyperbolic_return} we know that
$$
	| \mathcal M_\omega^*(a_1, \dots, a_{t-1})| \le C^t e^{- \gamma c_3 n}.
$$
Thus, we know that
\begin{align*}
	|\mathcal M_\omega | &\le \sum_{\substack{ a_1, \dots, a_{t-1} \in \mathbb Z_+ \\ a_1 + \dots + a_{t-1} < c_3 n}} |\mathcal M_\omega^*(a_1, \dots, a_{t-1})| \\
	& \le  {c_3 n \choose t-1} C^t e^{- \gamma c_3 n} 
	\le {c_3 n \choose c_1 n} C^{c_1 n} e^{- \gamma c_3 n},
\end{align*}
provided $c_1 < \frac 1 2 c_3$. Now, notice that we have the following:
\begin{align*}
	\Big\{ x \in I \Hquad | \Hquad \sum_{j =1}^{t (x)} (h^*_{ p_j(x)+ 1} - h^*_{ p_j(x)}) \Big\}  \subset \bigcup_{p, t} \bigcup_{\substack{i_1 < \dots < i_{t-1}    \\ i_1, \dots, i_{t-1} \in \{1, \dots , p  \} }}\mathcal M_\omega(i_1, \dots , i_t).
\end{align*}
This combined with the previous estimate gives us
\begin{align*}
	\Big|\Big\{ x \in I \Hquad | \Hquad \sum_{j =1}^{t (x)} (h^*_{ p_j(x)+ 1} - h^*_{ p_j(x)}) \Big\}\Big| &\le  \sum_{p, t} \sum_{\substack{i_1 < \dots < i_{t-1}    \\ i_1, \dots, i_{t-1} \in \{1, \dots , p  \} }} | M_\omega(i_1, \dots , i_t)|  \\ 
	& \le \sum_{p, t} {p \choose t} {c_3 n \choose c_1 n} C^{c_1 n} e^{- \gamma c_3 n}\\
	& \le c_3 M^{-1} n^2 {n \choose c_1 n}{n \choose c_1 n} C^{c_1n}e^{-\gamma c_3 n}\\
	& \le c_3 M^{-1}n^2 e^{2 \mu(c_1)n}C^{c_1 n}e^{- \gamma c_3 n},
\end{align*}
where $\mu(c_1)$ is some suitably chosen constant satisfying $\lim_{c_1 \to 0} \mu(c_1) = 0$. 

Thus, as long as we choose $c_1> 0$ sufficiently small, then there exists some constant $0 < \theta_3 < 1$ such that
$$
	e^{2 \mu(c_1)}C^{c_1}e^{- \gamma c_3} <  \theta_3.
$$
Thus, we have
$$
	\Big|\Big\{ x \in I \Hquad | \Hquad \sum_{j =1}^{t (x)} (h^*_{ p_j(x)+ 1} - h^*_{ p_j(x)}) \Big\}\Big|	\le c_3 M^{-1}n^2 (e^{2 \mu(c_1)}C^{c_1 }e^{- \gamma c_3 })^n < C \theta_3^n.
$$

Now, for the case where $\sum_{j=1}^m (h^*_{k_j + 1} - h^*_{k_j }) > c_3 n$ instead, we simply need to set $c_2 = c_1$ and, using the argument from case 2, this gives us
$$
	\Big|\Big\{ x \in I \Hquad | \Hquad \sum_{j =1}^{t (x)} (h^*_{ k_j(x)+ 1} - h^*_{ k_j(x)}) \Big\}\Big|	\le  C \theta_3^n.
$$
\textbf{Case 4}\\ \\
Suppose that for our fixed $n \in \mathbb Z_+$, our fixed $x \in I \setminus \bigcup_{k \le n} \bm{U}_k^\omega$ is not in any of the previous three cases. We denote the set of such points as $\tilde {X}_\omega^n \subset I  \setminus \bigcup_{k \le n} \bm{U}_k^\omega$. Furthermore, we denote $\hat X_\omega^n = \hat X_\omega^n(t,\{ p_i \}, \{ q_i \})= \tilde {X}_\omega^n \cap X_\omega(n,t,\{ p_i \}, \{ q_i \})$.

By construction, for any $1 \le j \le t$ and any $V \in \mathcal{U}_{q_j}$, the number of elements of the set $ \big\{ J_{x, h^*_{p_j}} \Hquad | \Hquad  x \in \tilde {X}_\omega^n \cap X_\omega(n,t,\{ p_i \}, \{ q_i \}) \big\}$ is either $1$ or $2$. Furthermore, using the same argument as in case 1 with bounded distortion, we can show that
$$
	\frac{|J_{x, h^*_{p_j}}|}{|V|} \le C e^{- \lambda_0(h^*_{p_i} - q_i)}.
$$
Thus, combining these two we obtain
$$
	\Big| \bigcup_{x \in \hat X_\omega^n   } J_{x, h^*_{p_j}}   \Big| \le 2 C e^{- \lambda_0(h^*_{p_i} - q_i)} |\bm{V}_j|.
$$
Now, clearly we need to show that $|\bm{V}_{j+1}|$ decays exponentially. Consider the case where $q_{j + 1} > h^*_{p_j}$. Then for any $V \in \mathcal V_{q_j+1}$ that intersects the set $J_{x, h^*_{p_j}}$, by case 1 we know that $V \subset (1 + \alpha) J_{x, h^*_{p_j}}$. Plugging this into the above estimate, this gives us
$$
	|\bm{V}_{j+1}| \le (1 + \alpha) \Big| \bigcup_{x \in \hat X_\omega^n   } J_{x, h^*_{p_j}}   \Big| \le 2 (1 + \alpha) C e^{- \lambda_0(h^*_{p_i} - q_i)} |\bm{V}_j|.
$$

Alternatively, consider the case when $q_{j + 1} \le  h^*_{p_j}$. Notice that for any $V \in \mathcal V_{j+1}$, if $V \cap J_{x, h^*_{p_j}} \neq \emptyset $, then we in fact have
$$
	V \cap (I \setminus J_{x, h^*_{p_j}}) \neq \emptyset.
$$
Indeed, by construction we have $T_\omega^{q_{j+1} }(V) = B(\delta ')$, which contains the point $0$, whereas $T_\omega^{q_{j+1}}(J_{x, h^*_{p_j}}) $ is bounded away from $0$. Thus, there are points in $T_\omega^{q_{j+1} }(V)$ which are not in $T_\omega^{q_{j+1}}(J_{x, h^*_{p_j}}) $, and therefore the same can be said about the original $V$ and $J_{x, h^*_{p_j}} $. Thus, we have
$$
	|\bm{V}_{j+1}| \le \Big| \bigcup_{x \in \hat X_\omega^n   } J_{x, h^*_{p_j}}   \Big| \le 2 C e^{- \lambda_0(h^*_{p_i} - q_i)} |\bm{V}_j| \le 2 C e^{- \lambda_0(q_{i+1} - q_i)} |\bm{V}_j|.
$$
Combining this all together, we have
\begin{align*}
	|\hat X_\omega^n  | &\le \Big| \bigcup_{x \in \hat X_\omega^n}  J_{x, h^*_{p_t}}  \Big| 
	 \le C e^{- \lambda_0(h^*_{p_t} - q_t)}|\bm{V}_t | \\
	& \le C^2 e^{- \lambda_0(h^*_{p_t} - q_{t-1})}|\bm{V}_{t-1} | e^{\max \{ 0 , q_t - h^*_{p_{t-1}} \} } \\
	&  \quad \quad \quad \quad \quad \quad \vdots \\
	& \le C^t e^{- \lambda_0(h^*_{p_t} - q_{1})} e^{\lambda_0 \sum  \max \{ 0 , q_i - h^*_{p_{i-1}} \} }  |\bm{V}_1 | \\
	& \le 2 C^{c_1 n} e^{- \lambda_0 n} e^{\lambda_0(q_1 + n - h^*_{p_t})} e^{\lambda_0 \sum  \max \{ 0 , q_i - h^*_{p_{i-1}} \} }.
\end{align*}
Clearly, we need to show that the last line above decays exponentially. Indeed, we will show that
$$
	(q_1 + n - h^*_{p_t}) + \sum  \max \big\{ 0 , q_i - h^*_{p_{i-1}} \big\} \le 2 c_3 n.
$$

For any $i < t$ , let $p_{i+1}'(x) = \min \big\{ j > p_i \Hquad | \Hquad  J_{x, h^*_{k}} \cap \bm{V}_{i+1} \neq \emptyset    \big\}$. By construction, this implies that $q_i < h^*_{p_i'}$ and that for any $p_i < k < p_{i+1}'$ we have $ J_{x, h^*_{p_j}} \cap (\partial I \cup \partial B(\delta')) \neq \emptyset$. Likewise, if $p_s \neq p$, then the same holds for any $p_t < k \le p$.  Thus, we can write
\begin{align*}
	(q_1 + n - h^*_{p_t}) + \sum  \max \big\{ 0 , q_i - h^*_{p_{i-1}} \big\} &< (h^*_{p_1'} + n - m_{p_t}) + \sum_{i =1}^{t-1}  \max \big\{ 0 , q_i - h^*_{p_{i-1}} \big\}  \\
	& \le \sum_{i = 1}^s (h^*_{p_i + 1} - h^*_{p_i }) + \sum_{i = 1}^m (h^*_{k_i + 1} - h^*_{k_i }) 
	\le 2 c_3 n,
\end{align*}
where we have used the fact that we assume our $x$ is not in case 3. Thus, we can write
$$
	|\hat X_\omega^n  |  \le  2 (C^{c_1 } e^{-(\lambda_0 - 2 c_3)})^n = 2 \theta_4^n,
$$
where we have set $0 < c_3 < \lambda_0 / (2 + \log C)$ and set $\theta_4 = C^{c_1 } e^{-(\lambda_0 - 2 c_3)}$.

Now, since $\hat X_\omega^n = \hat X_\omega^n (t,\{ p_i \}, \{ q_i \})$ is only for specific values $s, \{ p_i \},$ and $ \{ q_i \}$, we need to sum this over all possible values. Thus, we have
\begin{align*}
	\Big| \bigcup_{t \le c_1 n} \bigcup_{\{ p_i \}, \{ q_i \}  }  \hat X_\omega^n (t,\{ p_i \}, \{ q_i \}) \Big| &\le \sum_{t \le c_1 n} \sum_{\{ p_i \}, \{ q_i \}  } \big| \hat X_\omega^n (t,\{ p_i \}, \{ q_i \})\big| \\
	&\le \sum_{t \le c_1 n} \sum_{\{ p_i \}, \{ q_i \}  } 2 \theta_4^n 
	 \le c_1 n {n \choose c_1 n}{n \choose c_1 n} 2 \theta_4^n \\
	& \le 2 c_1 n (e^{2\hat \mu (c_1) } \theta_4)^n
	 \le C_4 \hat \theta_4^n,
\end{align*}
where $\hat \mu(c_1)$ is some suitably chosen constant satisfying $\lim_{c_1 \to 0} \hat \mu(c_1) = 0$, and where we have chosen $c_1$ small enough so that $\hat \theta_4 = e^{2\hat \mu (c_1) } \theta_4 < 1$. 

Since we have now exhausted all possibilities for $x \in  I \setminus \bigcup_{k \le n} \bm{U}_k^\omega$, we therefore have the desired result:
\begin{align*}
	\Big|   I \setminus \bigcup_{k \le n} \bm{U}_k^\omega \Big| &\le C_1 \theta_1^n + C_2 \theta_2^n + C_3 \theta_3^n + C_4 \hat \theta_4^n 
	 \le C e^{- \gamma n}.
\end{align*}
This completes the proof.
\end{proof}

Thus, we now have exponential tails of return times. This, combined with the random tower axioms having been satisfied, means that we can apply Theorem \ref{thm:randexp} to obtain exponential quenched decay of correlations on the random Young tower we have constructed. Furthermore, using the exact same arguments as at the end of Chapter 3, we can show that almost sure mixing on the random tower implies almost sure random mixing on the original system, thereby proving Theorem \ref{thm:1d}.

\appendix\label{appendix}

\section{Random Towers}

This appendix is based on the work and results of \cite{BBMD, Du}. See also the related work of \cite{BBR, LV} on random towers. We can introduce a random tower for almost every $\omega$ as follows:
\[
	\Delta_\omega = \{ (x, \ell) \in \Delta^{*} \times \mathbb{Z}_+ \Hquad | \Hquad  x  \in \Q^{\sigma^{-\ell} \omega}(\Delta^*), \Hquad j,  \ell \in \mathbb{N}, \Hquad 0 \le \ell \le \tau_{\sigma^{-\ell}\omega}(x) - 1\}.
\]
We can also define the random tower map $F_\omega : \Delta_\omega \to \Delta_{\sigma \omega}$ as
\[
	F_\omega(x, \ell) = \begin{cases}
				(x, \ell + 1), \quad &\ell + 1 < \tau_{\sigma^{-\ell}\omega}(x) \\
				(T_{\sigma^{-\ell}\omega}^{\ell + 1}(x), 0), &\ell + 1 = \tau_{\sigma^{-\ell}\omega}(x)				
			\end{cases}.
\]
Notice that this allows us to construct a partition on the random tower as
\[
	\mathcal{Z}_\omega= \{  F_{\sigma^{-\ell}\omega}^\ell (J^{\sigma^{-\ell}\omega}) \Hquad  | \Hquad J^{\sigma^{-\ell}\omega} \in \Q^{\sigma^{-\ell}\omega} (\Delta^*), \Hquad  \tau_\omega |{J^{\sigma^{-\ell}\omega}} \ge \ell + 1, \Hquad \ell \in \mathbb{Z}_+ \}.
\]
Let us define the separation time on the tower as
$$
	\hat s_\omega((x,\ell),  (y, \ell))=\min\{n\mid F^n_\omega(x, \ell )\in J,  F^n_\omega(y, \ell)\in J', J\neq J' \in \mathcal{Z}_\omega \}.
$$

Assume:
\begin{itemize}
	\item[(C1)] \textbf{Return and separation time:} the return time function $\tau_\omega$ can be extended to the whole tower as $\tau_\omega: \Delta_\omega \to \mathbb{Z}_+$ with $\tau_\omega$ constant on each $J \in \mathcal{Z}_\omega$, and there exists a positive integer $p_0$ such that $\tau_\omega \ge p_0$. Furthermore, if $(x, \ell)$ and $(y, \ell)$ are both in the same partition element $J \in \mathcal{Z}_\omega$, then $\hat s_\omega((x,0), (y,0)) \ge \ell$, and for every $(x, 0), (y,0) \in J \in \mathcal{Z}_\omega$ we have  
	\[
		\hat s_\omega((x,0), (y,0))= \tau_\omega(x, 0) + \hat s_{\sigma^{\tau_\omega}\omega}( F_\omega^{\tau_\omega (x,0)}(x,0),  F_\omega^{\tau_\omega (y,0)}(y,0))
	\]
	\item[(C2)] \textbf{Markov property:} for each $J \in \Q^\omega (\Delta^*)$ the map $ F_\omega^{\tau_\omega} |J  : J \to \Delta^{*}$ is bijective and both $ F_\omega^{\tau_\omega} |J$ and its inverse are non-singular.
	\item[(C3)] \textbf{Bounded distortion:} There exist constants $0 < \gamma < 1$ and $\mathcal{D}> 0$ such that for all $J \in \mathcal{Z}_\omega$ and all $(x, \ell), (y,\ell) = x, y \in J $
		\[
			\Big| \frac{J  F_\omega^{\tau_\omega}(x)}{J F_\omega^{\tau_\omega}(y)} - 1 \Big| \le \mathcal{D} \gamma^{\hat s_{\sigma^{\tau_\omega}\omega}( F_\omega^{\tau_\omega}(x,0),  F^{\tau_\omega}_\omega(y,0))}.
		\]
	where $J F_\omega^{\tau_\omega}$ denotes the Jacobian of $ F_\omega^{\tau_\omega}$.
	\item[(C4)] \textbf{Weak forwards expansion:} the diameters of the partitions $\bigvee^n_{j=0}  ( F_\omega^j)^{-1} \mathcal{Z}_{\sigma^j \omega}$ tend to zero as $n \to \infty$.
	\item[(C5)] \textbf{Return time asymptotics:} there exist constants $B, b > 0$, a full-measure subset $\Omega_1 \subset \Omega$ such that for every $\omega \in \Omega_1$ we have
\[
		m( \{ x \in \Delta^{*} | \tau_\omega > n \}) \le Be^{-bn}
\]
We also have for almost every $\omega \in \Omega$
\[
	m(\Delta_\omega) = \sum_{\ell \in \mathbb{Z}_+} m( \{\tau_{\sigma^{-\ell}\omega} > \ell \}) < \infty
\]
which gives us the existence of a family of finite equivariant sample measures. We also have for almost every $\omega \in \Omega$
\[
	\lim_{\ell_0 \to \infty} \sum_{\ell \ge \ell_0} m(\Delta_{\sigma^{\ell_0} \omega, \ell}) = 0.
\]

	\item[(C6)] \textbf{Aperiodicity:} there exists $N_0 \ge 1$, a full-measure subset $\Omega_2 \subset \Omega$ and a set $\{ t_i \in \mathbb{Z}_+ : \Hquad i = 1,2,...,N \}$ such that g.c.d.$\{ t_i \} = 1$ and there exist $\epsilon_i >0$ such that for every $\omega \in \Omega_2$ and every $i \in \{1,...,N_0\}$ we have $m( \{ x \in \Lambda | \Hquad \tau(x) = t_i \}) > \epsilon_i$. \\
\end{itemize}
Let us also define the following function spaces:
\begin{align}
	\mathcal{F}_\gamma^+ = \{ \varphi_\omega : \Delta_\omega \to \mathbb{R} \Hquad | \Hquad &\exists C_\varphi > 0 \text{ such that } \forall J \in \mathcal{Z}_\omega, \text{ either } \varphi_\omega|J \equiv 0 \\
	&\text{or }\varphi_\omega|J>0 \text{ and } \Big| \log \frac{\varphi_\omega(x)}{\varphi_\omega(y)} \Big| \le C_{\varphi} \gamma^{s_\omega(x,y)}, \forall x,y \in J \}. \nonumber
\end{align}
For almost every $\omega$ let $K_\omega: \Omega \to \mathbb{R}_+$ be a random variable which satisfies $\inf_\Omega K_\omega > 0$ and
\[
	\mathbb{P} (\{ \omega \Hquad | \Hquad K_\omega > n \}) \le e^{-un^v},
\]
where $u,v>0$. We then define the spaces
\begin{align}
	\mathcal{L}^{K_\omega}_\infty &= \{ \varphi_\omega: \Delta_\omega \to \mathbb{R} \Hquad | \Hquad \exists C_{\varphi}^{'} > 0, \sup_{x \in \Delta_\omega}|\varphi_\omega(x)| \le C_{\varphi}^{'} K_\omega \} \\
	\mathcal{F}_\gamma^{K_\omega} &= \{ \varphi_\omega \in \mathcal{L}^{K_\omega}_\infty \Hquad | \Hquad \exists C_{\varphi}'' > 0, |\varphi_\omega(x) - \varphi_\omega(y)| \le C_{\varphi}'' K_\omega \gamma^{s_\omega(x,y)}, \forall x, y \in \Delta_\omega \}
\end{align}
We assign to $\mathcal{L}^{K_\omega}_\infty$ and $\mathcal{F}_\gamma^{K_\omega}$ the norms $|| \varphi ||_{\mathcal{L}_\infty} = \inf C_\varphi '$ and $||\varphi||_\mathcal{F} = \max \{ \inf C_{\varphi_\omega}', \inf C_{\varphi_\omega}''\}$ respectively, which makes them Banach spaces.\\ \\
We can now apply the following theorem from \cite{ BBMD, Du}:

\begin{theorem}\label{thm:randexp}
Let $F_\omega$ satisfy (C1)-(C6), and let $K_\omega$ satisfy the above condition. Then for almost every $\omega \in \Omega$ there exists an absolutely continuous $\hat F_\omega$-equivariant probability measure $\nu_\omega = h_\omega m$ on $\Delta_\omega$, satisfying $(\hat F_\omega)_* \nu_\omega = \nu_{\sigma \omega}$, with $h_\omega \in \mathcal{F}_\gamma^+$. Furthermore, there exists a $\mathbb P$-integrable constant $C(\omega)$, and there exists a full-measure subset $\Omega_2 \subset \Omega$ such that for every $\omega \in \Omega_2$, $\varphi_\omega \in \mathcal{L}_\infty^{K_\omega}$ and $\psi_\omega \in \mathcal{F}_\gamma^{K_\omega}$ there exists a constant $C_{\varphi, \psi}$ such that for all $n \in \mathbb Z_+$ we have
\begin{align}\label{ineq:future}
		\Big| \int (\varphi_{\sigma^n \omega} \circ  F_\omega^n) \psi_\omega dm - \int \varphi_{\sigma^n \omega} d\nu_{\sigma^n \omega} \int \psi_\omega dm \Big| \le C(\omega) C_{\varphi, \psi} e^{-bn}
\end{align}
and 
\begin{align}\label{ineq:past}
		\Big| \int (\varphi_\omega \circ  F_{\sigma^{-n} \omega}^n) \psi_{\sigma^{-n} \omega} dm - \int \varphi_\omega d\nu_\omega \int \psi_{\sigma^{-n} \omega} dm \Big| \le C(\omega)  C_{\varphi, \psi} e^{-bn}.
\end{align}

\end{theorem}

\section*{Acknowledgement }  We thank Wael Bahsoun for suggesting the problem and his continuous attention throughout the entire duration of this project. We also thank Mike Todd for reading earlier version of the paper and making numerous useful comments.  A. Larkin would like to thank the hospitality of the University of Vienna where part of this work was carried. The research of M. Ruziboev is supported by the Austrian Science Fund (FWF): M2816 Meitner Grant.

\end{document}